\documentclass[a4paper,leqno,final]{amsart}

\usepackage{ifdraft}

\ifdraft{\usepackage[draft]{showkeys}}{\usepackage[final]{showkeys}}

\usepackage{fontenc}
\usepackage[utf8]{inputenc}

\usepackage{amsmath,amsfonts,mathscinet,eucal,amsthm,amssymb,bm,extarrows,upgreek,tensor,mathrsfs,textcomp,comment,mathtools,bbm,braket}

\usepackage[shortlabels]{enumitem}

\usepackage[usenames,dvipsnames]{xcolor} \usepackage{tikz}
\usetikzlibrary{patterns,matrix,through,arrows,decorations.pathreplacing,decorations.markings,snakes,shadows,shapes.geometric,positioning,calc,backgrounds,fit}

\def \myweightstyle {}
\def \myweightx {0.09}
\def \myweighty {0.2}
\newcommand{\mydrawdown}[1]%
  {\draw [\myweightstyle] #1++(0,-\myweighty) -- ++(\myweightx,\myweighty);%
    \draw [\myweightstyle] #1++(0,-\myweighty) -- ++(-\myweightx,\myweighty);}
\newcommand{\mydrawup}[1]%
  {\draw [\myweightstyle] #1 -- ++(-\myweightx,-\myweighty);%
    \draw [\myweightstyle] #1 -- ++(\myweightx,-\myweighty);}

\tikzset{anchorbase/.style={baseline={([yshift=-0.5ex]current bounding box.center)}},
  int/.style={thick},
  cross line/.style={preaction={draw=white,line width=6pt,-}},
  wall/.style={thin,double,blue},
  middlearrow/.style={postaction=decorate,decoration={markings,mark=at
    position .55 with {\arrow{stealth};}}},
  middlearrowrev/.style={postaction=decorate,decoration={markings,mark=at
    position .55 with {\arrowreversed{stealth};}}},
  ev/.style={shape=rectangle, draw}
}

\usepackage{todonotes} 

\usepackage[bbgreekl]{mathbbol}

\DeclareSymbolFontAlphabet{\mathbb}{AMSb}
\DeclareSymbolFontAlphabet{\mathbbol}{bbold}
\DeclareMathAlphabet{\mathpzc}{OT1}{pzc}{m}{it}

\DeclareSymbolFont{usualmathcal}{OMS}{cmsy}{m}{n}
\DeclareSymbolFontAlphabet{\mathucal}{usualmathcal}

\numberwithin{equation}{section}

\newtheoremstyle{myplain} {6pt plus 6pt minus 2pt}
{6pt plus 6pt minus 2pt}
{\itshape}
{}
{\bfseries}
{.}
{.5em}
{}

\theoremstyle{myplain}
\newtheorem{theorem}{Theorem}[section]
\newtheorem*{theorem*}{Theorem}
\newtheorem{lemma}[theorem]{Lemma}
\newtheorem{prop}[theorem]{Proposition}
\newtheorem{corollary}[theorem]{Corollary}

\newtheoremstyle{mydefinition} {6pt plus 6pt minus 2pt}
{6pt plus 6pt minus 2pt}
{\itshape}
{}
{\bfseries}
{.}
{.5em}
{}

\theoremstyle{mydefinition}
\newtheorem{definition}[theorem]{Definition}

\newtheoremstyle{myexample} {6pt plus 6pt minus 2pt}
{6pt plus 6pt minus 2pt}
{}
{}
{\scshape}
{.}
{.5em}
{}

\theoremstyle{myexample}
\newtheorem{example}[theorem]{Example}

\newtheoremstyle{myremark} {6pt plus 6pt minus 2pt}
{6pt plus 6pt minus 2pt}
{}
{}
{\scshape}
{.}
{.5em}
{}

\theoremstyle{myremark}
\newtheorem{remark}[theorem]{Remark}

\newcommand{\N}{\mathbb{N}}
\newcommand{\Z}{\mathbb{Z}}

\newcommand{\C}{\mathbb{C}}

\newcommand{\V}{{\mathbb{V}}}
\newcommand{\K}{\mathbb{K}}
\newcommand{\gl}{\mathfrak{gl}}
\newcommand{\catO}{\mathcal{O}}

\newcommand{\suchthat}{\,|\,} 
\newcommand{\mapto}{\rightarrow}
\newcommand{\surto}{\twoheadrightarrow}
\newcommand{\into}{\hookrightarrow}
\newcommand{\len}{\ell}

\newcommand{\blank}{\mathord{-}}

\DeclareMathOperator{\grdim}{gr\,dim}
\DeclareMathOperator{\Hom}{Hom}

\DeclareMathOperator{\End}{End}

\DeclareMathOperator{\Res}{Res}

\DeclareMathOperator{\rad}{rad}

\DeclareMathOperator{\Ann}{Ann}
\newcommand{\id}{\mathrm{id}}
\newcommand{\Id}{\mathrm{Id}}
\newcommand{\abs}[1]{\left|#1\right|}

\newcommand{\gmod}[1]{#1\mathrm{-gmod}}

\newcommand{\rgmod}[1]{\mathrm{gmod-}#1}
\newcommand{\lmod}[1]{#1\mathrm{-mod}}
\newcommand{\rmod}[1]{\mathrm{mod-}#1}

\renewcommand{\epsilon}{\varepsilon}

\renewcommand{\phi}{\varphi}

\newcommand{\up}{{\mathord\wedge}}
\newcommand{\down}{{\mathord\vee}}
\newcommand{\cross}{{\mathord\times}}

\newcommand{\calA}{{\mathcal{A}}}

\newcommand{\calE}{{\mathcal{E}}}
\newcommand{\calF}{{\mathcal{F}}}

\newcommand{\calP}{{\mathcal{P}}}
\newcommand{\calQ}{{\mathcal{Q}}}

\newcommand{\calS}{{\mathcal{S}}}

\newcommand{\sfC}{{\mathsf{C}}}

\newcommand{\sfW}{{\mathsf{W}}}

\newcommand{\sfZ}{{\mathsf{Z}}}

\newcommand{\scrB}{{\mathscr{B}}}

\newcommand{\frakb}{{\mathfrak{b}}}

\newcommand{\frakg}{{\mathfrak{g}}}
\newcommand{\frakh}{{\mathfrak{h}}}
\newcommand{\fraki}{{\mathfrak{i}}}
\newcommand{\frakj}{{\mathfrak{j}}}

\newcommand{\frakm}{{\mathfrak{m}}}
\newcommand{\frakn}{{\mathfrak{n}}}

\newcommand{\frakp}{{\mathfrak{p}}}
\newcommand{\frakq}{{\mathfrak{q}}}

\newcommand{\frakz}{{\mathfrak{z}}}

\newcommand{\bbS}{{\mathbb{S}}}

\newcommand{\bbV}{{\mathbb{V}}}

\newcommand{\bfE}{{\mathbf{E}}}
\newcommand{\bfF}{{\mathbf{F}}}

\newcommand{\bfH}{{\mathbf{H}}}

\newcommand{\bfL}{{\mathbf{L}}}

\newcommand{\bfV}{{\mathbf{V}}}

\newcommand{\frakQ}{{\mathfrak{Q}}}

\newcommand{\bolda}{{\boldsymbol{a}}}
\newcommand{\boldb}{{\boldsymbol{b}}}

\newcommand{\boldh}{{\boldsymbol{h}}}
\newcommand{\boldi}{{\boldsymbol{i}}}
\newcommand{\boldj}{{\boldsymbol{j}}}

\newcommand{\boldx}{{\boldsymbol{x}}}

\newcommand{\boldz}{{\boldsymbol{z}}}

\newcommand{\leadingterm}{\textsc{lt}}

\DeclareMathOperator{\height}{ht}
\newcommand{\op}{{\mathrm{op}}}
\newcommand{\ucalH}{{\mathucal H}}

\newcommand{\ucalZ}{{\mathucal Z}}

\newcommand{\compn}{{\mathbbol{n}}}

\newcommand{\short}{{\mathrm{short}}}

\renewcommand{\epsilon}{\varepsilon}

\DeclareMathOperator{\Add}{Add}

\newcommand{\funcV}{{\mathbb{V}}}

\newcommand{\pres}{\text{-}\mathrm{pres}}
\newcommand{\oDelta}{{\overline{\Delta}}}
\newcommand{\catOZ}{\prescript{\Z}{}{\catO}}

\newcommand{\sbar}{|}
\newcommand{\psbar}{\rfloor}

\newcommand{\sumprime}{\sideset{}{^{(v)}} \sum}

\DeclareMathOperator{\defect}{def}

\newcommand{\twoheadlongrightarrow}{\relbar\joinrel\twoheadrightarrow}

\usepackage{wrapfig} \usepackage{subcaption}

\usepackage[final=true]{hyperref,bookmark}

\hypersetup{
  colorlinks = false,
  urlcolor = false,
  pdfauthor = {Antonio Sartori},
  pdfkeywords = {Symmetric polynomials, Category O, Soergel modules, Khovanov algebra},
  pdftitle = {A diagram algebra for Soergel modules corresponding to smooth Schubert varieties},
  pdfsubject = {},
  pdfpagemode = UseNone,
  bookmarksopen = true,
  bookmarksopenlevel = 3,
  pdfdisplaydoctitle = true }

\title[A diagram algebra for Soergel modules]{A diagram algebra for Soergel modules corresponding to smooth Schubert varieties}
\author{Antonio Sartori}
\address{Mathematisches Institut\\Endenicher Allee 60\\Universit\"at Bonn\\53115 Bonn, Germany}
\email{sartori@math.uni-bonn.de}
\urladdr{http://www.math.uni-bonn.de/people/sartori}
\date{\today}
\thanks{This work has been supported by the Graduiertenkolleg 1150, funded by the Deutsche Forschungsgemeinschaft.}

\begin{document}

\begin{abstract}
  Using combinatorial properties of symmetric polynomials, we compute
  explicitly the Soergel modules for some permutations 
  whose corresponding Schubert varieties are rationally smooth. We
  build from them diagram algebras whose module categories are
  equivalent to the subquotient categories of the BGG category
  $\catO(\gl_n)$ which show up in categorification of
  $\gl(1|1)$--representations.
  We construct diagrammatically the graded cellular
  structure and the properly stratified structure of these
  algebras.
\end{abstract}

\maketitle

\setcounter{tocdepth}{1}
\tableofcontents

\section{Introduction}
\label{sec:introduction}

The BGG category $\catO(\frakg)$, introduced in \cite{MR0407097}, is a fundamental tool for studying the representation theory of
a reductive Lie algebra $\frakg$. Its combinatorial properties are quite
surprising: the Kazhdan-Lusztig conjecture, for example,
relates decomposition numbers in the category $\catO(\frakg)$ with certain
polynomials which appear naturally in the theory of Hecke algebras
\cite{MR560412}.

The explicit structure of this category remained obscure
until the fundamental work of Soergel \cite{MR1029692}, who
constructed a fully faithful functor $\bbV$ from the
additive category of projective objects of $\catO(\frakg)$
to a category of modules (which we call \emph{Soergel
  modules}) over a polynomial ring. In principle, this gives
a way to do explicit computations in the category
$\catO(\frakg)$ (see \cite{MR2017061} for some examples in
small cases); however, determining Soergel modules and their
homomorphisms is not affordable in general. In this paper,
we present such a description for some
subquotient categories of $\catO_0 = \catO_0(\gl_n)$, which
are denoted by $\calQ_k(\compn)$
in \cite{miopaperO}. These are a particular case of the ``generalized
parabolic subcategories'' introduced in \cite{MR1921761} and developed
further in \cite{MR2057398}.  Our interest for these particular
categories is motivated by the fact that they can be used to
categorify tensor powers of the vector representation of $\gl(1|1)$ (see \cite{miopaperO}).

The central role played by category $\catO$ in
categorification makes the ability to do explicit
computations crucial. Khovanov homology \cite{MR1740682},
which is ultimately an $\mathfrak{sl}_2$--categorification
through category $\catO$, has proven very useful in its
application to low dimensional topology, because of its
explicit description.
A diagrammatic approach to understand some parabolic
subcategories of the category $\catO(\gl_n)$ has been
successfully carried out by Brundan and Stroppel: in a
series of four paper (\cite{pre06126156}, \cite{MR2600694},
\cite{MR2781018}, \cite{MR2881300}) they define diagrammatic
algebras (called generalized Khovanov algebras) which are
 isomorphic to the endomorphism rings of projective
generators of the parabolic category $\catO^\frakp(\gl_n)$,
where $\frakp \subseteq \gl_n$ is a maximal parabolic
subalgebra. They are able to interpret diagrammatically
translation functors, and hence compute explicitly
homomorphism spaces between them. Using the machinery they
have constructed, they can connect the category
$\catO(\gl_n)$ to the category of finite dimensional
representations of the super Lie algebra $\gl(r|s)$; they
also find interesting connections to representations of
the walled Brauer algebra \cite{MR2955190}.

The goal of this paper is to develop a similar approach in
the case of the categories $\calQ_k(\compn)$. Our main
result is a diagrammatic/combinatorial definition of graded
algebras $A_{n,k}$ whose module categories are equivalent to
$\calQ_k(\compn)$.

\bigskip The main tools of our construction are Soergel
modules and combinatorics of symmetric polynomials. Soergel
modules for the symmetric group $\bbS_n$ (or for $\gl_n$)
are modules $\sfC_z$ over the polynomial ring in $n$
variables $R=\C[x_1,\ldots,x_n]$, and are labeled by
permutations $z \in \bbS_n$. By definition, they are the
indecomposable direct summands of modules of the type
\begin{equation}
  \label{eq:25}
  B \otimes_{B^{s_{i_1}}} B \otimes \cdots \otimes B \otimes_{B^{s_{i_r}}} B \otimes_B \C
\end{equation}
where $B= R/(R_+)^{\bbS_n}R$ is the ring of the coinvariants and
$B^{s}$ are the invariants under a simple reflection $s \in
\bbS_n$. In particular, if we choose a reduced expression
$s_{i_r}\cdots s_{i_1}$ for $z$, then $\sfC_z$ is the indecomposable
direct summand of \eqref{eq:25} containing $1 \otimes \cdots \otimes
1$.

Although they could seem apparently harmless, identifying explicitly
the direct summands in \eqref{eq:25} is quite tricky, even in some
small examples. Describing homomorphism spaces between such summands
is in general hard. A significant simplification occurs when the
Soergel module $\sfC_z$ happens to be \emph{cyclic}: in this case it
is enough to determine the annihilator of $\sfC_z$, which is the same
as the annihilator of \eqref{eq:25}. Moreover, in order to describe the
homomorphisms between two cyclic $R$--modules it is enough to study the
quotient ideal between the corresponding annihilators.

The condition of a Soergel module being cyclic arises naturally since
it is equivalent to the rational smoothness of the corresponding
Schubert variety in the full flag variety (cf.\ \cite[Appendix]{MR560412});
an easy combinatorial criterion for determining which Schubert
varieties are rationally smooth is given in \cite{MR1934291}. It is
quite surprising that all the Soergel modules needed for understanding
the categories $\calQ_k(\compn)$
satisfy this condition.

We point out that some work has been done to understand Soergel
\emph{bi}modules \cite{MR1173115}, which are the equivariant version
of Soergel modules and are obtained dropping the last $\otimes_B \C$
in \eqref{eq:25} (see \cite{MR2844932}, \cite{2009arXiv0902.4700E};
we would also like to mention that Soergel bimodules have been prove
useful for the recent algebraic proof \cite{2012arXiv1212.0791E} of the
Kazhdan-Lusztig conjecture). On the other side, as far as we know,
this is the first attempt to describe explicitly Soergel modules. We
believe that Soergel modules are in some sense more difficult than
Soergel bimodules because of the lack of symmetry of \eqref{eq:25}.

Despite of the deep background, the computations required along the
paper are elementary, although quite involved. In order to keep the paper
readable also for non-experts, we decided to start the paper with combinatorial properties of symmetric polynomials and postpone the introduction
of category $\catO$ after the definition of the diagram algebra.

\subsubsection*{Overview of the paper}
\label{sec:overview-paper}

Let us now explain in more detail the structure and the results of the paper. We start in Section \ref{soerg:sec:symmetric-polynomials} by studying
quotients of the polynomial ring $R$ modulo ideals $I_\boldb$ generated
by some complete symmetric polynomials in a subset of the variables.
More precisely, if $\boldb=(b_1,\ldots,b_n)$ is a non-increasing sequence
of positive integers, then we let
\begin{equation}
  \label{eq:26}
  I_\boldb = (h_{b_1}(x_1),h_{b_2}(x_1,x_2),\ldots,h_{b_n}(x_1,\ldots,x_n)),
\end{equation}
where the $h_i$'s are complete symmetric polynomials of degree $i$ in
the variables indicated. This definition will prove very helpful to
describe the Soergel modules we are interested in. To study such
ideals, we will use the powerful machinery of Groebner bases, which we
will briefly review.

In Section \ref{soerg:sec:some-canon-bases} we introduce the Hecke algebra and define the Kazhdan-Lusztig polynomials. Then we compute explicitly some Kazhdan-Lusztig basis elements, which we will use later to determine the dimension of the corresponding Soergel modules.

In Section \ref{soerg:sec:soergel-modules} we briefly recall the
definition of Soergel modules and then compute explicitly some of
them. In particular, we compute all Soergel modules $\sfC_{w_k z}$
where $z$ is in the set $D$ of shortest coset representative for
$(\bbS_k \times \bbS_{n-k})\backslash \bbS_n$ and $w_k$ is the longest
element of $\bbS_k$. Our strategy is the following: first, we show
that the Soergel module $\sfC_{w_k z}$ is cyclic. Then for every such
$z$ we choose a clever reduced expression $s_{i_r} \cdots s_{i_1}$ and
we determine some partial symmetric polynomials which lie in the
annihilator of the module \eqref{eq:25}. We deduce that they are
enough to generate the whole annihilator by comparing the dimensions,
using the results we collected in the previous sections. Our first
main result is:
\begin{theorem}[see Theorem \ref{soerg:thm:1}]
  \label{thm:4}
  For each $z \in D$ there is an associated $\boldb$--sequence $\boldb^z$ such that the Soergel module $\sfC_{w_k z}$ is isomorphic to $R/I_{\boldb^z}$.
\end{theorem}

Our next main result is the explicit description of the spaces of
homomorphisms between these Soergel modules modulo the morphisms which
factor through some ``wrong'' Soergel modules (see Theorem
\ref{soerg:thm:2}). This is motivated by the fact that the categories
$\calQ_k(\compn)$
are \emph{subquotient categories} of
$\catO_0$: hence they are described by algebras which are the
endomorphism rings of some projective modules of $\catO_0$
modulo morphisms which factor through some ``wrong'' projective
modules (see Section \ref{sec:category-cato} for the details).

To make the statement of Theorem \ref{soerg:thm:2} clear and usable,
in Section \ref{sec:diagr-algebra-calq_k} we introduce diagrams which
describe the corresponding homomorphism spaces. Putting together all
the homomorphism spaces and translating into our diagram language we
obtain diagram algebras $A_{n,k}$. These algebras resemble the
generalized Khovanov algebras of \cite{pre06126156}. Recall that the
generalized Khovanov algebras categorify tensor powers of
$\mathfrak{sl}_2$--representations, and the diagrams which build these
algebras are the same diagrams of \cite{MR1446615} describing
$\mathfrak{sl}_2$--intertwiners. Unsurprisingly, the same happens for
our algebras $A_{n,k}$: the diagram which we use here are essentially
the same diagrams which we introduce in the related paper
\cite{miopaperO} for describing $\gl(1|1)$--intertwiners; and indeed
the resulting algebras categorify tensor powers of
$\gl(1|1)$--representations.

We remark that the algebras $A_{n,k}$ are \emph{graded}, the grading coming from the one on polynomial rings.
Since the algebras $A_{n,k}$ describe categories
which come from category $\catO$, we could deduce from abstract Lie
theory that they enjoy very nice properties, such as
\emph{cellularity} \cite{MR1376244} and \emph{proper stratifiedness}
\cite{MR1921761} (which is a generalization of
quasi-hereditariness). In the second part of Section
\ref{sec:diagr-algebra-calq_k} we show how these properties can be
proved independently,
by constructing explicitly
cellular, standard and projective modules using our diagrams, and we
obtain:

\begin{theorem}[see Prop.~\ref{soerg:prop:10} and Theorem \ref{soerg:thm:6}]
  \label{thm:7}
  For all $n \in \N$, $0 \leq k \leq n$ the algebra $A_{n,k}$ is graded cellular and properly stratified.
\end{theorem}

In Section \ref{sec:category-cato} we finally explain in detail the
relation with category $\catO$ which underlies all the paper. We recall from \cite{miopaperO} the definition of the categories $\calQ_k(\compn)$
and we state our main result:
\begin{theorem}[see Theorem \ref{soerg:cor:4}]
  \label{thm:5}
  The category $\calQ_k(\compn)$ from \cite{miopaperO} is equivalent to $\rgmod{A_{n,k}}$.
\end{theorem}
This gives an explicit description of the categories $\calQ_k(\compn)$'s categorifying $\gl(1|1)$ and makes possible to compute examples of that categorification. One of the main application is the computation of the endomorphism rings of the functors $\calE_k$ and $\calF_k$ from \cite{miopaperO} (see Theorem \ref{thm:6}).

We remark that there is a connection between the algebra
$A_{n,k}$ and the cohomology of closed attracting varieties
in the Springer fiber of hook type sitting inside the full
flag variety, see \cite{miophd}. In particular, we
conjecture that it is possible to construct a convolution
product on these cohomology rings as in \cite{MR2914857} so
that one can recover the full algebra $A_{n,k}$ using
geometry.

\subsubsection*{Acknowledgements} The present paper is part of the author's PhD thesis. The author would like to thank his advisor Catharina Stroppel for her help and support. The author would also like to thank an anonymous referee for many valuables comments and suggestions.

\section{Symmetric polynomials}
\label{soerg:sec:symmetric-polynomials}

In this section we are going to study some rings obtained as quotients of a polynomial ring modulo an ideal generated by complete symmetric functions in some subsets of variables. We start recalling some easy standard facts about symmetric polynomials.

We let $R=\C[x_1,\ldots,x_n]$ be a polynomial ring. We consider it as a graded ring with $\deg x_i= 2$ for every $i$.

\subsection{Complete symmetric polynomials}
\label{soerg:sec:compl-symm-polyn}

The \emph{complete symmetric polynomials} are defined as
\begin{equation}
  \label{soerg:eq:1b}
  h_j(x_1,\ldots,x_n) = \sum_{1\leq i_1\leq \cdots \leq
    i_j \leq n} x_{i_1}\cdots x_{i_j}
\end{equation}
for every $j \geq 1$ so that for example
$h_2(x_1,x_2)=x_1^2+x_1x_2+x_2^2$. We set also
$h_0(x_1,\ldots,x_n)=1$, while if $n=0$ (i.e., we have zero
variables), we let $h_i()=0$ for every $i\geq 1$.  The
symmetric group $\bbS_n$ acts on $R$ permuting the
variables, and the polynomials $h_i(x_1,\ldots,x_n)$ are
invariant under this action; in fact, they generate the whole algebra
$R^{\bbS^n}$ of invariant polynomials (see
\cite[Section~6]{MR1464693}).

We will consider complete symmetric polynomials in some subset of the variables of $R$. The following formula
helps us to decompose a complete symmetric polynomials in $k$ variables
as complete symmetric polynomials in $\ell$ and $k-\ell$ variables, for
every $\ell=1,\ldots,k-1$:
\begin{equation}
  \label{soerg:eq:2}
  h_j(x_1,\ldots,x_k) = \sum_{n=0}^{j} h_{n}(x_1,\ldots,x_\ell)
  h_{j-n} (x_{\ell+1},\ldots,x_k).
\end{equation}
Another formula allows us to express a complete symmetric polynomials in $k-1$ variables in terms of complete symmetric polynomials in $k$ variables:
\begin{equation}
  \label{soerg:eq:3}
  h_j(x_1,\ldots,x_{k-1})=h_j(x_1,\ldots,x_k)-x_k h_{j-1}(x_1,\ldots,x_k).
\end{equation}
Both \eqref{soerg:eq:2} and \eqref{soerg:eq:3} can be checked easily by comparing which monomials appear on both sides.

For $1\leq i \leq n-1$ let $R^{s_i}$ be the subring of $R$ consisting
of polynomials invariant under the simple transposition $s_i$. We
recall from \cite{MR0342522} the definition of the classical {\em
  Demazure operator} $\partial_i: R \rightarrow R^{s_i}\langle 2\rangle$, given by
 \begin{equation}
   \partial_i: f \longmapsto
   \frac{f-s_i(f)}{x_{i}-x_{i+1}}.\label{soerg:eq:82}
\end{equation}
The operator $\partial_i$ is linear, vanishes on $R^{s_i}$ and satisfies $ \partial_i (fg) = f \partial_i g$ whenever $f \in R^{s_i}$.
Let also $P_i: R \rightarrow R$ be defined by $P_i(f)
= f-x_{i} \partial_i (f)$. It is easy to show that $P_i$ has also values in $R^{s_i}$. 
The operators $\partial_i$ and $P_i$ can be used to define the decomposition $R \cong R^{s_i} \oplus x_i R^{s_i}$ as a $R^{s_i}$--module, by
\begin{equation}
  \label{soerg:eq:161}
    f  \mapsto P_i f \oplus x_i \partial_i f.
\end{equation}

Demazure operators have the nice property of sending complete symmetric polynomials to other complete symmetric polynomials:
\begin{lemma}
  For all $j \geq 1$ we have
  \begin{equation}
    \label{soerg:eq:5}
    \partial_k h_j (x_1,\ldots,x_k) = h_{j-1} (x_1,x_2,\ldots,x_{k+1}).
  \end{equation}
\end{lemma}

We omit the proof, which is a straightforward computation.

\subsection{Ideals generated by complete symmetric polynomials}
\label{soerg:sec:ideal-gener-compl}

We are going to study quotients rings of $R$ generated by some of the $h_i$'s. Let
\begin{equation}
  \scrB'= \{ \boldb = (b_1,\ldots, b_n) \in \N^n
  \suchthat b_i \geq b_{i+1} \geq b_i -1 \}.\label{soerg:eq:94}
\end{equation}
In other words, $\scrB'$ is the
set of weakly decreasing sequences of positive numbers such
that the difference between two consecutive items is at most one. For
every sequence $\boldb \in \scrB'$ let $I_\boldb \subset R$ be the ideal
generated by
\begin{equation}\label{soerg:eq:7}
  h_{b_1} (x_1) , h_{b_2}(x_1,x_2) , \ldots , h_{b_n}(x_1,\ldots,x_n).
\end{equation}
Set also $R_\boldb=R/I_\boldb$.

We will shortly recall the definition of Groebner basis, which are a useful tool for studying ideals in polynomial rings; for a complete reference see \cite[Chapter~2]{MR2290010}. Let us fix a lexicographic monomial order on $R$ with
\begin{equation}
x_n > x_{n-1} > \cdots > x_1.\label{soerg:eq:162}
\end{equation}
With respect to this ordering, each polynomial $p \in R$ has
a leading term $\leadingterm(p)$. Given an ideal $I
\subseteq R$, let $\leadingterm(I) = \{\leadingterm(p)
\suchthat p \in I\}$ be the set of leading terms of elements
of $I$ and let $\langle\leadingterm(I)\rangle$ be the ideal
they generate.  We recall that a finite subset
$\{p_1,\ldots,p_r\}$ of an ideal $I$ of $R$ is called a
\emph{Groebner basis} if the leading terms of the
$p_1,\ldots,p_r$ generate $\langle\leadingterm(I)\rangle$.  Then
we have:

\begin{lemma}
  \label{soerg:lem:8}
  The polynomials \eqref{soerg:eq:7} are a Groebner basis for $I_\boldb$ with respect to the order \eqref{soerg:eq:162}.
\end{lemma}

\begin{proof}
  By \cite[Theorem~2.9.3 and Proposition 2.9.4]{MR2290010} it is enough
  to check that the leading monomials of the polynomials \eqref{soerg:eq:7}
  are pairwise relatively prime. This is obvious.
\end{proof}

\begin{prop}
  \label{soerg:prop:1}
  Let $\boldb \in \scrB'$. The quotient ring $R_\boldb= R/I_\boldb$ has dimension $b_1
  \cdots b_n$, and a $\C$--basis is given by
  \begin{equation}
    \label{soerg:eq:8}
    \{\boldsymbol{x^j}=x_1^{j_1} \cdots x_n^{j_n} \suchthat 0 \leq j_i < b_i \}.
  \end{equation}
\end{prop}

\begin{proof}
By the theory of Groebner bases (cf.\ \cite[Proposition~2.6.1]{MR2290010}), any $f \in R$ can be written uniquely as $f=g+r$, with $g \in I_\boldb$ and $r$ such that no term of $r$ is divisible by any of the leading terms of the Groebner basis \eqref{soerg:eq:7}; that is, $r$ is a linear combination of the monomials \eqref{soerg:eq:8}. This means exactly that the monomials \eqref{soerg:eq:8} are a basis of $R_\boldb$.
\end{proof}

\begin{example}
  \label{soerg:ex:3}
  Let $\boldb=(1,\ldots,1)$. Then $x_i = h_1(x_1,\ldots,x_i)-h_1(x_1,\ldots,x_{i-1})$ lies in $I_\boldb$ for each $i$, hence $I_\boldb = (x_1,\ldots,x_n)$ and $R_\boldb \cong \C$ is one dimensional.
\end{example}

\begin{example}
  \label{soerg:ex:4}
  Let $\boldb=(n,n-1,\ldots,1)$. Then it is easy to show that the ideal $I_\boldb$ is the ideal generated by the symmetric polynomials in $n$ variables with zero constant term, and $R_\boldb$ is the ring of the coinvariants $R/(R_+^{\bbS_n})$, isomorphic to the cohomology of the full flag variety of $\C^n$ (see \cite[\textsection{}10.2, Proposition~3]{MR1464693}). As given by Proposition~\ref{soerg:prop:1}, it has dimension $n!$ and it is well-known that a monomial basis is given by
  \begin{equation}
    \label{soerg:eq:163}
    \{ x_1^{j_1} \cdots x_n^{j_n} \suchthat 0 \leq j_i \leq n-i \}.
  \end{equation}
\end{example}

\subsection{Morphisms between quotient rings}
\label{soerg:sec:morph-betw-quot}

Next, we are going to determine all $R$--module homomorphisms between rings $R_\boldb$.

\begin{prop}
  \label{soerg:prop:2}
  Let $\boldb,\boldb' \in \scrB'$, and let $c_i = \max\{b'_i-b_i,0\}$. Then a $\C$--basis of $\Hom_R (R_\boldb, R_{\boldb'})$ is given by
  \begin{equation}
    \label{soerg:eq:15}
    \{ 1 \mapsto x_1^{j_1}\cdots x_n^{j_n} \suchthat c_i \leq j_i < b'_i \}.
  \end{equation}
\end{prop}

The proof consists of several lemmas.

\begin{lemma}
  \label{soerg:lem:1}
  Let $\boldb \in \scrB'$. Then $h_a(x_1,\ldots,x_i) \in I_\boldb$ for every
  $a \geq b_i$.
\end{lemma}
\begin{proof}
  We prove by induction on $\ell \geq 0$ that $h_{b_i+\ell}
  (x_1,\ldots,x_i) \in I_\boldb$ for every $i=1,\ldots,n$. For $\ell=0$
  the statement follows from the definition. For the
  inductive step, choose an index $i$ and pick $j<i$ maximal such that $b_j=b_i+1$ (or let
  $j=0$ if such an index does not exist) and write using iteratively
  \eqref{soerg:eq:3}:
  \begin{multline*}
    h_{b_i + \ell}(x_1,\ldots,x_i) = h_{b_i+\ell} (x_1,\ldots,x_j) +
    x_{j+1}h_{b_i+\ell-1} (x_1,\ldots,x_{j+1})+\\ + \cdots +
    x_{i-1}h_{b_i+\ell-1}(x_1,\ldots,x_{i-1})+x_i
    h_{b_i+\ell-1}(x_1,\ldots,x_i).
  \end{multline*}
  Since $b_i + \ell = b_j + \ell -1$, the terms on the right all lie in $I_\boldb$ by the inductive hypothesis.
\end{proof}

\begin{lemma}
  \label{soerg:lem:9}
  Let $\boldb=(b_1,\ldots,b_n) \in \scrB'$ and
  \begin{equation}
    \boldb'=(b_1,\ldots,b_{i-1},b_i+1,b_{i+1},\ldots,b_n)\label{soerg:eq:165}
\end{equation}
 for some $i$. Suppose that also $\boldb' \in \scrB'$. Then $I_{\boldb'} \subset I_{\boldb}$ while $x_i I_{\boldb} \subseteq I_{\boldb'}$.
\end{lemma}

\begin{proof}
  It follows directly from Lemma~\ref{soerg:lem:1} that $I_{\boldb'} \subset I_{\boldb}$. For the other assertion, since $h_{b_j} (x_1,\ldots,x_j) \in I_{\boldb'}$ for every $j \neq i$, we only need to prove that $x_i h_{b_i}(x_1,\ldots,x_i) \in I_{\boldb'}$. By \eqref{soerg:eq:3} we have
  \begin{equation}
    \label{soerg:eq:164}
    x_i h_{b_i}(x_1,\ldots,x_i) = h_{b_i+1}(x_1,\ldots,x_i) - h_{b_i + 1}(x_1,\ldots,x_{i-1}).
  \end{equation}
  Since we suppose $\boldb' \in \scrB'$, it follows that $b_{i-1}=b_i +1$, hence the r.h.s. of \eqref{soerg:eq:164} lies in $I_{\boldb'}$.
\end{proof}

We will call two sequences $\boldb,\boldb' \in \scrB'$ that satisfy the hypothesis of Lemma~\ref{soerg:lem:9} (without regarding the order) \emph{near each other}.

\begin{lemma}
  \label{soerg:lem:10}
  Let $\boldb,\boldb' \in \scrB'$ and set $c_i = \max\{b'_i-b_i,0\}$. Then $x_1^{c_1}\cdots x_{n}^{c_n} I_{\boldb} \subseteq I_{\boldb'}$.
\end{lemma}

\begin{proof}
  We can find a sequence $\boldb = \boldb^{(0)},\boldb^{(1)},\ldots, \boldb^{(N)}=\boldb'$ with $\boldb^{(k)} \in \scrB'$ for each $k$ and $N= \sum_i \abs{b_i -b_i'}$ such that $\boldb^{(i)}$ and $\boldb^{(i+1)}$ are near each other. Then the claim follows applying iteratively Lemma~\ref{soerg:lem:9}.
\end{proof}

\begin{lemma}
  \label{soerg:lem:11}
  Let $\boldb,\boldb' \in \scrB'$. Let $c_i = \max\{b'_i-b_i,0\}$. Suppose $p \in R$ is such that  $p I_{\boldb} \subseteq I_{\boldb'}$. Then $x_1^{c_1}\cdots x_{n}^{c_n} \mid p$.
\end{lemma}

\begin{proof}
  We prove the claim by induction on the leading term of $p$, using
  the lexicographic order \eqref{soerg:eq:162}. Let $p'$ be the leading term of $p$ and pick an index
  $1 \leq i \leq n$. By assumption, $p h_{b_i}(x_1,\ldots,x_i) \in
  I_{\boldb'}$. By the theory of Groebner basis, the leading term of $p
  \psi_{b_i}(x_1,\ldots,x_i)$ is divisible by $x_1^{b'_1}\cdots
  x_n^{b'_n}$, and this leading term is just $p' x_i^{b_i}$. It
  follows immediately that $x_1^{c_1} \cdots x_n^{c_n} \mid p'$. By
  Lemma~\ref{soerg:lem:10} we then know that $p' I_\boldb \subseteq I_{\boldb'}$,
  hence also $(p-p') I_\boldb \subseteq I_{\boldb'}$. By induction, we may
  assume that $x_1^{c_1} \cdots x_n^{c_n} \mid (p-p')$, and we are
  done.
\end{proof}

\begin{proof}[Proof of Proposition~\ref{soerg:prop:2}]
  It follows from Lemma~\ref{soerg:lem:10} that the elements of \eqref{soerg:eq:15} indeed define morphisms $R_\boldb \mapto R_{\boldb'}$. By Proposition~\ref{soerg:prop:1} they are linearly independent, and by Lemma~\ref{soerg:lem:11} they are a set of generators.
\end{proof}

\subsection{Duality}
\label{soerg:sec:duality-1}

The category of finite dimensional $R$--modules has a duality, given by
\begin{equation}
  \label{soerg:eq:84}
  M^* = \Hom_\C(M, \C).
\end{equation}
In fact, the vector space $M^*$ is endowed with an $R$--action by setting $(r \cdot f)(m) = f(r\cdot m)$ for all $f \in M^*$, $m \in M$, $r \in R$ (since $R$ is commutative). If $M$ is graded, the dual inherits a grading declaring $(M^*)_j=(M_{-j})^*$.

Now consider some $\boldb \in \scrB'$. The monomial basis
\eqref{soerg:eq:8} of $R_\boldb$ has a unique element of
maximal degree $b=2(b_1+\cdots+b_n-n)$, namely
$\boldx^{\boldsymbol{b}-\boldsymbol{1}}$ where
$\boldsymbol{1}=(1,\ldots,1)$ and
$\boldsymbol{b}-\boldsymbol{1}$ is the sequence
$(b_1-1,\ldots,b_n-1)$. We define a symmetric bilinear form $(
\cdot,\cdot)$ on $R_\boldb$ by letting 
\begin{equation}
  \label{eq:27}
  (\boldx^\boldj,\boldx^{\boldsymbol{j'}}) =
  \begin{cases}
    1& \text{if } \boldj + \boldsymbol{j'} = \boldb - \boldsymbol{1},\\
    0 & \text{otherwise}
  \end{cases}
\end{equation}
on the monomial basis \eqref{soerg:eq:8}, where sequences are added termwise. Since this form is clearly non-degenerate, 
we get an isomorphism of
graded $R$--modules
  \begin{equation}
    \label{soerg:eq:85}
    R_\boldb \cong R_\boldb^*\langle -b \rangle \quad \text{for every } \boldb \in \scrB'.  \end{equation}
The degree shift comes out because the bilinear form pairs the degree $i$ component of $R_\boldb$ with its component of degree $b-i$.

By the properties of a duality, we have
\begin{equation}
  \label{soerg:eq:86}
  \Hom_R (R_\boldb,R_{\boldb'}) \cong \Hom_R (R_{\boldb'}^* , R_{\boldb}^*) \cong \Hom_R (R_{\boldb'}, R_\boldb) \langle b'-b \rangle
\end{equation}
for any $\boldb,\boldb' \in \scrB'$. It is not difficult to see that the composite  isomorphism is given explicitly by
\begin{equation}
  \label{soerg:eq:87}
 \Theta: (1 \mapsto p) \longmapsto \left(1 \mapsto \frac{\boldx^{\boldsymbol{b}-\boldsymbol{1}}}{\boldx^{\boldsymbol{b'}-\boldsymbol{1}}} p\right).
\end{equation}

\subsection{Schubert polynomials}
\label{sec:schubert-polynomials}
We recall some basic facts about Schubert polynomials, referring to \cite{MR1161461} for more details. Let $w \in \bbS_n$  be a permutation; then the operator $\partial_w = \partial_{i_1} \cdots \partial_{i_r}$, where $w=s_{i_1} \cdots s_{i_r}$ is some reduced expression, does not depend on the particular chosen reduced expression and is hence well-defined. Let $w_n \in \bbS_n$ be the longest element. Then one defines the \emph{Schubert polynomial}
\begin{equation}\label{soerg:eq:83}
  \mathfrak S_w (x_1,\ldots,x_n) = \partial_{w^{-1}w_n} x_1^{n-1} x_2^{n-2} \cdots x_{n-1}
\end{equation}
for each $w \in \bbS_n$.
The Schubert polynomials give a basis of $R/(R_+^{\bbS_n})$. It follows from the definition that $\deg \mathfrak S_w (x_1,\ldots,x_n)=2\len(w)$.

For our purposes, it will be more convenient to have a monomial basis of $R/(R_+^{\bbS_n})$, indexed by permutations $w \in \bbS_n$.

\begin{definition}
  For each $w \in \bbS_n$ we define $\mathfrak S'_w(x_1,\ldots,x_n)$ to be the
  leading term of $\mathfrak S_w(x_1,\ldots,x_n)$ in the
  lexicographic order \eqref{soerg:eq:162}.
\end{definition}

Being the leading terms of a basis of $R/(R^{\bbS_n}_+)$, it
follows by the theory of Groebner bases (see \textsection\ref{soerg:sec:ideal-gener-compl}) that also the
monomials $\mathfrak S'_w(x_1,\ldots,x_n)$ give a basis.

\begin{remark}
  \label{soerg:rem:3}
  We already noticed in Example~\ref{soerg:ex:4} that $R/(R_+^{\bbS_n}) \cong R_\boldb$ for $\boldb=(n,n-1,\ldots,1)$. Hence we already have a monomial basis of $R/(R_+^{\bbS_n})$ given by Proposition~\ref{soerg:prop:1}. In fact, this basis coincides with the basis $\{\mathfrak S'_w(x_1,\ldots,x_n) \suchthat w \in \bbS_n\}$; the advantage of using Schubert polynomials is that they give us a way to index these basis elements through permutations.
\end{remark}

There is an easy way to construct the monomials $\mathfrak
S'_w (x_1,\ldots,x_n)$ (cf.\ \cite{MR1241505}): let
$c_{i} = \# \{j<w^{-1}(i) \suchthat w(j)>i \}$; then
$\mathfrak S'_w(x_1,\ldots,x_n) = x_1^{c_1} \cdots
x_{n-1}^{c_{n-1}}$.

\begin{example}
  \label{soerg:ex:5}
  The following table contains the Schubert polynomials and the polynomials $\mathfrak S'_w$ in the case $n=3$.
    \begin{table}[h]
      \renewcommand{\arraystretch}{1.1}
      \setlength{\tabcolsep}{10pt}
      \centering
      \begin{tabular}[]{c|c|c}
        $w \in \bbS_3$ & $\mathfrak S_w$ & $\mathfrak S'_w$ \\
        \hline $e$ & $1$ & $1$ \\
        $s$ & $x_1$ & $x_1$ \\
        $t$ & $x_1+x_2$ & $x_2$ \\
        $st$ & $x_1x_2$ & $x_1x_2$ \\
        $ts$ & $x_1^2$ & $x_1^2$ \\
        $w_3$ & $x_1^2 x_2$ & $x_1^2 x_2$
      \end{tabular}
    \end{table}
\end{example}

\section{Some canonical bases elements}
\label{soerg:sec:some-canon-bases}

In this section we recall the definition of the bar involution and the
canonical basis of the Hecke algebra for the symmetric group
$\bbS_n$. We will then compute some canonical basis elements in the
Hecke algebra for some special permutations of the symmetric group
$\bbS_n$. We will need these expressions to compute the dimension of
Soergel modules in the next section.

We remark that all actions of $\bbS_n$ will be from the right.

\subsection{Hecke algebra}
\label{sec:hecke-algebra}
Let $s_1,\ldots,s_{n-1}$ denote the simple reflections which
generate the symmetric group $\bbS_n$. For $w \in \bbS_n$ we
let $\len(w)$ be the length of $w$. Moreover, we denote
by $\prec$ the Bruhat order on $\bbS_n$.

The Hecke algebra of the symmetric group $W=\bbS_n$ is the
unital associative $\C(v)$--algebra $\ucalH_n$ generated by $\{H_i
\suchthat i=1,\ldots,n-1\}$ with relations
\begin{subequations}
  \begin{align}
    H_i H_j &= H_j H_i \qquad \text{if } \vert i-j\vert>2,\label{eq:7} \\
    H_i H_{i+1} H_i & = H_{i+1} H_i H_{i+1}, \label{eq:8}\\
    H_i^2&=(v^{-1}-v)H_i+1.\label{eq:9}
  \end{align}
\end{subequations}
Note that we use Soergel's normalization \cite{MR1445511}, instead of
Lusztig's one. It follows from \eqref{eq:9} that the
elements $H_i$ are invertible with $H_i^{-1}=H_i
+v-v^{-1}$. For $w \in \bbS_n$ such that $w=s_{i_1} \cdots
s_{i_r}$ is a reduced expression, we define
$H_w=H_{i_1}\cdots H_{i_r}$. Thanks to \eqref{eq:8}, this
does not depend on the chosen reduced expression. The
elements $H_w$ for $w \in W$ form a basis of $\ucalH_n$,
called \emph{standard basis}, and we have
\begin{equation}
  H_w H_i = \begin{cases}
    H_{w s_i} & \text{if }\ell(w s_i)> \ell(w),\\
    H_{w s_i} + (v^{-1}-v) H_{w} & \text{otherwise.}
  \end{cases}
\end{equation}

We can define on $\ucalH_n$ a \emph{bar involution} by
$\overline{H_w}=H^{-1}_{w^{-1}}$ and $\overline{v}=v^{-1}$;
in particular $\overline{H_i}=H_i+v-v^{-1}$. We also have a {\em
  bilinear form} $\langle -,- \rangle$ on $\ucalH_n$ such that the
standard basis elements are orthonormal:
\begin{equation}\label{eq:165}
  \langle H_w , H_{w'} \rangle = \delta_{w, w'} \qquad \text{for all }w, w' \in W.
\end{equation}

By standard arguments one can prove the following:

\begin{prop}[\cite{MR560412}, in the normalization of \cite{MR1445511}]\label{prop:KLBasisOnHeckeAlgebra}
  There exists a unique basis $\{C_w \suchthat w \in
  W\}$ of $\ucalH_n$ consisting of bar-invariant elements such that
  \begin{equation}\label{eq:4}
    C_w = H_w + \sum_{w' \prec w} \mathcal{P}_{w',w}(v) H_{w'}
  \end{equation}
  with $\mathcal{P}_{w',w} \in v\Z[v]$ for every $w' \prec w$.
\end{prop}

The basis $C_w$ is called \emph{Kazhdan-Lusztig
  basis}. We will also call it \emph{canonical basis} of $\ucalH_n$.

\begin{remark}\label{remark:ConstructionOfCanonicalBasis}
  There is an inductive way to construct the canonical basis
  elements. First, note that $C_e = H_e$. Then set $C_i =
  H_i + v$: since $C_i$ is bar invariant, we must have $C_{s_i} =
  C_i$. Now suppose $w= w's_i \succ w'$ : then $C_{w'}
  C_i$ is bar invariant and is equal to $ H_{w}$ plus a $\Z[v,v^{-1}]$--linear combination of some $H_{w''}$ for $w'' \prec w$. It follows that
  \begin{equation}
    \label{eq:10}
    C_{w'}C_i = C_w + p \quad \text{ for some } p \in \bigoplus_{w'' \prec w} \Z C_{w''} .
  \end{equation}
\end{remark}

\subsection{Combinatorics}
\label{soerg:sec:combinatorics}

Let us fix an integer $0 \leq k \leq n$. If $s_1,\ldots,s_{n-1}$ are
the simple reflections in $\bbS_n$, let $W_k$ be the subgroup
generated by $s_1,\ldots,s_{k-1}$ and $W_k^\perp$ be the subgroup
generated by $s_{k+1},\ldots,s_{n-1}$. Notice that $\bbS_k \times \bbS_{n-k} \cong W_k \times W_k^\perp \subseteq \bbS_n$. Let
$w_k$ be the longest element of $W_k$, and let
$D=D_{n,k} $ be the set of shortest coset representatives
\scalebox{0.9}{$\big(\text{\raisebox{-2pt}{$W_k \times W_k^\perp$}} \rotatebox{10}{$\big\backslash$} \text{\raisebox{2pt}{$\bbS_n$}}\big)^{\short}$}.
The set $D $ is in natural bijection with $\up\down$--sequences consisting of $k$ $\up$'s and $n-k$ $\down$'s, by mapping the identity $e \in \bbS_n$ to the sequence
\begin{equation}\label{soerg:eq:181}
  e = \underbrace{\up \cdots \up}_{k} \underbrace{\down \cdots \down}_{n-k}
\end{equation}
and letting $\bbS_n$ act by permutation of positions; in order to obtain a bijection with right coset representatives we regard this as a right action.
From now on, we identify an element $z \in D $ with the corresponding
$\up\down$--sequence.

There are a few ways to encode an element $z\in D $, that we are now going to explain.

\begin{description}[leftmargin=0cm,itemsep=2ex,format=\normalfont\em]
\item[The position sequences] In an $\up\down$--sequence $z \in D $, we
  number the $\up$'s (resp.\ the $\down$'s) from $1$ to $k$ (resp.\ from
  $1$ to $n-k$) from the left to the right. Moreover, we number the
  positions of an $\up\down$--sequence from $1$ to $n$ from the left to
  the right. We let $\up_i^z$ be the position of the $i$-th $\up$ and
  $\down_j^z$ be the position of the $j$-th $\down$ in $z$. For
  example, in the sequence
  \begin{equation*}
    \label{soerg:eq:166}
    z = \up \down \down \up \down \up \up
  \end{equation*}
  we have $\up^z_2=4$ and $\down^z_1=2$.  Notice that both the
  sequences $(\up_1^z,\ldots,\up_k^z)$ and
  $(\down_1^z,\ldots,\down_{n-k}^z)$ uniquely determine $z$.

\item[The $\up$--distance sequence]  We set
  \begin{equation}
    \label{soerg:eq:167}
    z^\up_i = \up^z_i - i \qquad \text{for } i=1,\ldots,k,
  \end{equation}
  so that
  \begin{equation}
    \label{soerg:eq:168}
    (\up^z_1,\ldots,\up^z_k) = (1+z_1^\up, \ldots, k + z_k^\up).
  \end{equation}
  In other words, $z^\up_i$ measures how many steps the $i$-th $\up$
  of the initial sequence $e$ has been moved to the right by the
  permutation $z$. This defines a bijection $z \mapsto \boldz^\up$
  between $D $ and the set
  \begin{equation}
    \{\boldz^\up =
    (z^\up_1,\ldots,z^\up_k) \suchthat 0 \leq z_1^\up \leq \cdots \leq z_k^\up \leq n-k\}.\label{soerg:eq:23}
  \end{equation}
  Define the permutation
  \begin{equation}\label{eq:151}
    t^\up_{i,\ell}=s_i s_{i+1}\cdots s_{i+\ell-1}
  \end{equation}
  for all $i = 1,\ldots,n-1$ and $\ell=1,\ldots,n-i$ (and set
  $t^\up_{i,0}=e$). Then we have a reduced expression for $z$:
  \begin{equation}
    z=t^\up_{k,z^\up_k} \cdots
    t^\up_{1,z^\up_1}.\label{soerg:eq:24}
  \end{equation}

\item[The $\down$--distance sequence]  Analogously, set
  \begin{equation}
    \label{soerg:eq:169}
    z^\down_i = i - \down^z_{k-i} \qquad \text{for }i=k+1,\ldots,n
  \end{equation}
  so that
  \begin{equation}
    \label{soerg:eq:170}
    (\down^z_1,\ldots,\down^z_{n-k}) = (k+1-z^{\down}_{k+1},\ldots,n-z^\down_{n-k}).
  \end{equation}
  In other words, $z^\down_i$ measures how many steps the $(i-k)$-th
  $\down$ of $e$ has been moved to the left by the permutation
  $z$. This defines a bijection $z \mapsto \boldz^\down$ between $D $ and
  the set
  \begin{equation}
     \{\boldz^\down =
    (z^\down_{k+1},\ldots,z^\down_n) \suchthat k \geq z_{k+1}^\down \geq \cdots \geq z_n^\down \geq 0\}.\label{soerg:eq:22}
  \end{equation}
  Define
  \begin{equation}
    t^\down_{k+i,\ell}=s_{k+i-1} s_{k+i-2} \cdots
    s_{k+i-\ell}\label{soerg:eq:25}
  \end{equation}
  for $i=1,\ldots,n-k$ and $\ell=1,\ldots,k$ (and set $t^\down_{k+i,0}=e$). Then we have another reduced expression
  for $z$:
  \begin{equation}
    z =t^\down_{k+1,z^\down_{k+1}}\cdots t^\down_{n,z^\down_n}.\label{soerg:eq:26}
  \end{equation}

\item[The $\boldb$--sequence]  Finally we want to assign to the element $z \in D $ its
  $\boldb$--sequence $\boldb^z$. Let
  \begin{equation}
    \label{soerg:eq:34}
    \scrB = \left\{\boldb = (b_1,\ldots,b_n) \in \N^n \, \left|
        \begin{aligned}
          & k+1 \geq b_1 \geq \cdots \geq b_n=1,\\
          & b_i \leq b_{i+1}+1 \quad \forall\,i=1,\ldots,n-1
        \end{aligned}
      \right.\right\}
  \end{equation}
  and define $\boldb^z \in \scrB$ by
  \begin{equation}
    \label{soerg:eq:171}
    b^z_i= \# \{ j \suchthat \up^z_j > i \} + 1.
  \end{equation}
  In other words, $b^z_i-1$ is the number of $\up$'s on the right of
  position $i$. It is clear that $\boldb^z$ uniquely determines the
  element $z \in D $. In fact, this defines a bijection between $D $ and $\scrB$.
\end{description}

\begin{example}\label{soerg:ex:1}
  Let $n=8$, $k=4$ and consider the element $z=s_4 s_5 s_6 s_3 \in
  D $. The corresponding $\up\down$--sequence and the $\boldb$--sequences are:
  \begin{equation*}
    \renewcommand{\tabcolsep}{0.05cm}
    \begin{tabular}{ccccccccc}
      & $\up$ & $\up$ & $\down$ & $\up$ & $\down$ & $\down$ & $\up$ & $\down$ \\
      $\mathbf b^z =$ & $4$ & $3$ & $3$ & $2$ & $2$ & $2$ & $1$ & $1$\\
    \end{tabular}
  \end{equation*}
  We also have $\boldz^\up =(0,0,1,3)$ and
  $\boldz^\down=(2,1,1,0)$.
\end{example}

\subsection{Canonical basis elements}
\label{soerg:sec:canon-basis-elem}

One of the rare examples of explicitly known canonical basis elements is the following (cf.\ \cite[Prop. 2.9]{MR1445511}):
\begin{lemma}\label{soerg:lem:12}
  Let $w_k$ be the longest element of\/ $\bbS_k$. Then the canonical basis element $C_{w_k}$ is given
  by
  \begin{equation}
    C_{w_k} = \sum_{w' \in \bbS_k} v^{\ell(w_k)-\ell(w')} H_{w'}.\label{soerg:eq:10}
  \end{equation}
  This expression holds both in $\ucalH_k$ and in $\ucalH_n \supset \ucalH_k$ for any $n>k$.
\end{lemma}

In fact, formula \eqref{soerg:eq:10} generalizes verbatim to
the case of the longest element of any parabolic subgroup of
$\bbS_n$.

In the next proposition we will generalize \eqref{soerg:eq:10} and
give explicit formulas for the canonical basis elements $C_{w_k z} $
for $z \in D $. But first we introduce the following notation: we set
  \begin{equation}
    \label{soerg:eq:29}
    \sumprime_{w' \in \bbS_k} f(w') = \sum_{w' \in S_k} v^{-\len (w')} f(w') \quad \text{and} \quad \sumprime_{i=0}^{h} g(i) = \sum_{i=0}^h v^{-i} g(i)
  \end{equation}
  for whatever functions $f$ defined on $\bbS_n$ and $g$ defined on $\{0,\ldots,h\}$.

\begin{prop}\label{soerg:prop:3}
  Let $z \in D $, with $z=t^\down_{k+1,z^\down_{k+1}} \cdots
  t^\down_{n,z^\down_n}$. Then
  \begin{equation}   \label{soerg:eq:176}
    C_{w_k  z} = \sumprime_{w' \in \bbS_k} \sumprime_{i_{k+1} = 0}^{z^\down_{k+1}} \cdots \sumprime_{i_n=0}^{z^\down_{n}} v^{\ell(w_k  z)} H_{w't^\down_{k+1,i_{k+1}}\cdots t^\down_{n,i_n}}.
  \end{equation}
\end{prop}

\begin{proof}
  First, we note that all words $w't^\down_{k+1,i_{k+1}} \cdots
  t^\down_{n,i_n}$ that appear in the expression on the right are
  actually reduced words. This is clear if we look at the action
  of this permutation on the string
  \begin{equation}
    \up_1 \cdots \up_k \down_{k+1} \cdots \down_{n}\label{soerg:eq:27}
  \end{equation}
  from the right: the length of the permutation is the cardinality of
  the set $X$ of the couples of symbols of this string that have been
  inverted. To $X$ belong $\ell(w')$ couples consisting of two
  $\up$'s; moreover, every $\down_{k + \alpha}$ appears in $X$ exactly
  $i_\alpha$ times coupled with some $\up$ or some $\down_{k+\beta}$
  for $\beta < \alpha$. Hence the length of the permutation
  $w't^\down_{k+1,i_{k+1}} \cdots t^\down_{n,i_n}$ is exactly
  $\ell(w') + i_{k+1} + \cdots + i_n$, and therefore this is a reduced
  expression.

  Now, in the r.h.s.\ of \eqref{soerg:eq:176} the coefficient of $H_{w_k  z}$
  is one, while the coefficient of every other basis element
  $H_{w't^\down_{k+1,i_{k+1}} \cdots t^\down_{n,i_n}}$ is divisible by
  $v$. Hence the only thing we have to prove is that the r.h.s\ of
  \eqref{soerg:eq:176} is bar invariant.

  We proceed by induction on the length of $z$, the case $z=0$ being
  given by Lemma~\ref{soerg:lem:12}. Let $h$ be the greatest index such that
  $z^\down_h \neq 0$. Hence we have $z=t^\down_{k+1,z^\down_{k+1}}
  \cdots t^\down_{h,z^\down_h}$. First suppose that $z^\down_h \geq
  2$.  Set $z'=t^\down_{k+1,z^\down_{k+1}} \cdots
  t^\down_{h,z^\down_h-1}$ and $j=h-z^\down_h$ so that $z= z'
  s_j$. We compute:
  \begin{align}
      &C_{w_k  z'} C_j  = \left( \sumprime_{w' \in \bbS_k} \sumprime_{i_{k+1} = 0}^{z^\down_{k+1}} \cdots \sumprime_{i_h=0}^{z^\down_{h}-1} v^{\len(w_k  z')} H_{w't^\down_{k+1,i_{k+1}}\cdots t^\down_{h,i_h}} \right) (H_j + v)\\
    &= \sumprime_{w' \in \bbS_k} \sumprime_{i_{k+1} = 0}^{z^\down_{k+1}} \cdots \sumprime_{i_{h-1}=0}^{z^\down_{h-1}} v^{\len(w_k  z')-{z^\down_h} +1 } H_{w't^\down_{k+1,i_{k+1}}\cdots t^\down_{h-1,i_{h-1}}t^\down_{h,z^\down_h}}\label{soerg:eq:173}\\
    &\quad + \left(\sumprime_{w' \in \bbS_k} \sumprime_{i_{k+1} = 0}^{z^\down_{k+1}} \cdots \sumprime_{i_{h}=0}^{z^\down_{h}-2} v^{\len(w_k  z')} H_{w't^\down_{k+1,i_{k+1}}\cdots t^\down_{h,i_{h}}}\right)H_j \label{soerg:eq:174}\\
    &\quad + \sumprime_{w' \in \bbS_k} \sumprime_{i_{k+1} = 0}^{z^\down_{k+1}}
    \cdots \sumprime_{i_h=0}^{z^\down_{h}-1} v^{\ell(w_k  z')+1} H_{w't^\down_{k+1,i_{k+1}}\cdots
      t^\down_{h,i_h}}.\label{soerg:eq:175}
  \end{align}
  The element $C_{w_k z'} C_j$ is obviously bar-invariant. Moreover, the sum of \eqref{soerg:eq:173} and \eqref{soerg:eq:175}  gives exactly the
  r.h.s.\ of \eqref{soerg:eq:176} for $C_{w_k  z}$; hence we only need to prove
  that \eqref{soerg:eq:174} is bar invariant. But it is easy to check that in \eqref{soerg:eq:174} the term $H_j$ on the right acts as $v^{-1}$; hence \eqref{soerg:eq:174} is equal to the
  r.h.s.\ of \eqref{soerg:eq:176} for $C_{w_k  z''}$, where
  $z''=t^\down_{k+1,z^\down_{k+1}} \cdots t^\down_{h,z^\down_h-2}$,
  and this is bar-invariant by induction.

  Now suppose instead that $z_h^\down=1$, and set  $z'=t^\down_{k+1,z^\down_{k+1}} \cdots t^\down_{h-1,z^\down_{h-1}}$
  so that $z=z' s_{h-1}$. Then a computation shows that  $C_{w_k z'}C_{h-1}$ is equal to the r.h.s.\ of \eqref{soerg:eq:176} for $C_{w_k  z}$, which is therefore bar-invariant (see \cite{miophd} for more details).
\end{proof}

Let $z \in D $, with $z=t^\down_{k+1,z^\down_{k+1}} \cdots
  t^\down_{n,z^\down_n}$, and suppose that for some index $j$ we have $z^\down_j = z^\down_{j+1}$. Then in a similar way as above it is possible to compute explicitly the canonical basis element $C_{s_j w_k  z}$. We refer to \cite{miophd} for a detailed computation and we state instead the following corollary, which is all we will use later:

\begin{corollary}
  \label{soerg:cor:1}
    Let $z \in D $, with $z=t^\down_{k+1,z^\down_{k+1}} \cdots
  t^\down_{n,z^\down_n}$. Suppose that for some index $j$ we have $z^\down_j = z^\down_{j+1}$. Then the canonical basis element $C_{s_j w_k  z}$ is a sum of
  \begin{equation}
    \label{soerg:eq:44}
    k!(z^\down_{k+1}+1) \cdots (z^\down_{j}+1)(z^\down_{j+1}+2) (z^\down_{j+2}+1 ) \cdots (z^\down_n+1)
  \end{equation}
  standard basis elements with monomial coefficients in $v$.
\end{corollary}

\section{Soergel modules}
\label{soerg:sec:soergel-modules}

In this section we will describe some Soergel modules as quotient rings $R_\boldb$ (defined in Section \ref{soerg:sec:symmetric-polynomials}). The strategy is the following: we prove that the ideal $I_\boldb$ is contained in the annihilator and we then use a dimension argument comparing the dimension of $R_\boldb$ (Proposition~\ref{soerg:prop:1}) and of the Soergel module (given by the corresponding canonical basis element computed in Section \ref{soerg:sec:some-canon-bases}). We compute then the homomorphism spaces between these Soergel modules (\textsection\ref{soerg:sec:morph-betw-soerg-1}).

Fix a positive integer $n$ and let $R=\C[x_1,\ldots,x_n]$ be the polynomial ring in $n$ variables. For $0 \leq \ell \leq n$ let $J_\ell$ be the ideal generated by the
non-constant symmetric polynomials in $\ell$ variables
$x_1,\ldots,x_\ell$. Let moreover $B = \C[x_1,\ldots,x_n]/J_n$. For a
simple reflection $s_i \in \bbS_n$, let $B^{s_i}$ denote the
invariants under $s_i$. In the following, we will abbreviate
$\otimes_{B^{s_i}}$ by $\otimes_i$ while $\otimes$ will be simply
$\otimes_B$. We let also $B_i=B \otimes_i B$. 

\subsection{Soergel modules}
\label{sec:soergels-modules}

We define now Soergel modules for the symmetric group $\bbS_n$ by recursion on the Bruhat ordering. First we set $\sfC_e = \C$. Let then $w \in \bbS_n$ be a permutation and choose some reduced expression $w = s_{i_1} \cdots s_{i_r}$ where $s_{i_1},\ldots,s_{i_r}\in \bbS_n$ are simple reflections. Then we have:

\begin{theorem}[\cite{MR1029692}]
  \label{thm:2}
  The $B$--module $B_{i_r}\otimes \cdots \otimes B_{i_1} \otimes \C$ has a unique indecomposable direct summand $\sfC_w$ which is not isomorphic to some $\sfC_{w'}$ for $w' \prec w$. This is the unique indecomposable summand containing $1 \otimes \cdots \otimes 1$. Up to isomorphism, $\sfC_w$ does not depend on the particular reduced expression chosen for $w$.
\end{theorem}

We call the $\sfC_w$'s for $w \in \bbS_n$ \emph{Soergel modules}.

\begin{example}
  Consider a simple reflection $s_i \in \bbS_n$. According
  to the theorem, the indecomposable object $\sfC_i = \sfC_{s_i}$ is a
  summand of $B_i \otimes \C$. But it is immediate to check that the
  two dimensional $B$--module $B_i \otimes \C$ is indecomposable, hence
  $\sfC_i = B_i \otimes \C$. This module is in fact isomorphic to
  $R/(x_1,\ldots,x_{i-1},x_i+x_{i+1},x_{i+2},\ldots,x_n)$.\label{ex:1}
\end{example}

Notice that since $B$ is a quotient of $R$ we have $  \Hom_B(M,N) \cong \Hom_R(M,N)$
for all $M,N \in \lmod{B}$. In other words, the category of $B$--modules embeds as a full subcategory into the category of $R$--modules.
Hence, it is harmful to consider $B$--modules as $R$--modules.

To compute Soergel modules we will need to know their dimension. This is given from the Kazhdan-Lusztig conjecture:

\begin{theorem}
  \label{thm:3}
  For all $w \in \bbS_n$ we have
  \begin{equation}
    \label{eq:3}
    \dim_\C \sfC_w = \sum_{z \preceq w} \calP_{z,w}(1),
  \end{equation}
  where the $\calP_{z,w}$'s are the Kazhdan-Lusztig polynomials \eqref{eq:4}.
\end{theorem}

\begin{proof}
  For the proof, we need some facts about the BGG category $\catO$ (see
  \cite{MR2428237} and Section \ref{sec:category-cato}
  below); we use the notation from
  \textsection\ref{sec:category-cato-1}.  By Theorem
  \ref{thm:9} we have
$\sfC_w \cong \Hom_B (B, \sfC_w) \cong
  \Hom_\catO (P(w_0 \cdot 0), P(w \cdot 0))$. 
Hence
\begin{equation}
\dim_\C \sfC_w = \dim_\C \Hom_\catO (P(w_0 \cdot 0),
P(w \cdot 0)) = [P(w \cdot 0): L(w_0 \cdot 0)],\label{eq:29}
\end{equation}
where the
latter denotes the multiplicity of the simple module $L(w_0
\cdot 0)$ in some composition series of $P(w \cdot
0)$. Since $P(w \cdot 0)$ has a Verma filtration, and since
$[M(z \cdot 0) : L(w_0 \cdot 0)]=1$ for all Verma modules
$M(z \cdot 0)$, we have further that $[P(w \cdot 0): L(w_0
\cdot 0)] = \sum_{z \in \bbS_n} (P(w \cdot 0): M(z \cdot
0))$, where $(P(w \cdot 0):M(z \cdot 0))$ denotes the
multiplicity of $M(z \cdot 0)$ in some Verma filtration of
$P(w \cdot 0)$. The Kazhdan-Lusztig conjecture \cite{MR560412} (see \cite{2012arXiv1212.0791E} for a proof) states
precisely that $(P(w \cdot 0): M( z \cdot 0)) =
\calP_{z,w}(1)$, and this concludes the proof (notice that $\calP_{z,w}(1) = 0$ unless $z \preceq w$).
\end{proof}

\subsection{Some Soergel modules}
\label{soerg:sec:some-soerg-modul}
We determine now explicitly some modules $\sfC_w$. In the following, we will use the notation introduced in \textsection\ref{soerg:sec:combinatorics}.
We recall the following well-known fact:

\begin{lemma}
  As a $\C$--vector space, a basis of $ B \otimes_{i_1}
  B \otimes \cdots \otimes_{i_{r-1}} B \otimes_{i_r} \C$ is
  given by
  \begin{equation} \label{soerg:eq:12}
    \{ x_{i_1}^{\epsilon_1} \otimes \cdots \otimes
    x_{i_r}^{\epsilon_r} \otimes 1 \suchthat \epsilon_j \in \{0,1\}\}.
  \end{equation}
\end{lemma}
\begin{proof}
  The claim follows since each polynomial $f \in R$ can be written uniquely as $f= P_i(f) + x_i \partial_i (f)$, with $P_i(f), \partial_i(f) \in R^{s_i}$, cf.\ \eqref{soerg:eq:161}.
\end{proof}

A key-tool to determine the Soergel modules $\sfC_{w_k z}$ is given by the next proposition; unfortunately it is based on a lemma which uses facts about the BGG category $\catO$, and that we hence postpone to Section \ref{sec:category-cato}.

\begin{prop}
  \label{soerg:prop:4}
  For all $z \in D $ the module $\sfC_{w_k z}$ is cyclic. In particular,
  $\sfC_{w_k z}\cong R/\Ann_R \sfC_{w_k z} \cong B/\Ann_B \sfC_{w_k z}$.
\end{prop}

\begin{proof}
  By Proposition~\ref{soerg:prop:3}, $H_e$ appears exactly once with
  coefficient $v^{\ell(w_k  z)}$ in the canonical basis element $C_{w_k
    z}$. By Lemma \ref{lem:3}, this implies that $\sfC_{w_k z}$ is cyclic.
\end{proof}

\begin{lemma}\label{soerg:lem:dim-of-special-soergel-modules}
  For every $z\in D $ the dimension
  of $\mathsf C_{w_k  z}$ over $\C$ is given by
  \begin{equation}\label{soerg:eq:13}
    \dim_\C \mathsf C_{w_k  z} = k!(z^\down_{k+1}+1)\cdots (z^\down_n+1) = b^z_1 \cdots b^z_n.
  \end{equation}
\end{lemma}

\begin{proof}
  The first equality follows directly from Theorem~\ref{thm:3} and
  Proposition~\ref{soerg:prop:3}. We want to show the second
  equality. As in Example~\ref{soerg:ex:1}, we imagine the
  $\boldb$--sequence written on top of the
  $\up\down$--sequence for $z$. Over the $\up$'s we have the
  numbers between $1$ and $k$, each appearing once: hence
  their contribute is $k!$. Over the $j$-th $\down$, we have
  a number measuring how many $\up$'s are on its right, plus
  one: this coincides with how many times this $\down$ has
  been moved to the left plus one, that is,
  $z^\down_{k+j}+1$. The claim follows immediately.
\end{proof}

\begin{lemma}
  \label{soerg:lem:38}
  The module $\sfC_{w_k }$ is isomorphic to $R/(J_k,x_{k+1},\ldots,x_n)$.
\end{lemma}

\begin{proof}
  Let $J'=(J_k,x_{k+1},\ldots,x_n)$. By Proposition~\ref{soerg:prop:4}, the module $\sfC_{w_k }$ is cyclic over $B$. Choose any reduced expression $s_{i_1}\cdots s_{i_N}$ for $w_k $ and build the corresponding module $B_{w_k } = B_{i_N} \otimes \cdots \otimes B_{i_1} \otimes \C$. Since all polynomials of $J'$ are symmetric in the first $k$ variables, we have $J' \subseteq \Ann_R B_{w_k } \subseteq \Ann_R \sfC_{w_k }$, hence $\sfC_{w_k }$ is a quotient of $R/J'$.  Notice that $J' = I_{\boldb^e}$ for $e \in \bbS_n$ the identity element. By Lemma~\ref{soerg:lem:dim-of-special-soergel-modules} and Proposition~\ref{soerg:prop:1}, $\dim_\C \sfC_{w_k } = \dim_\C R/J'$, hence $\sfC_{w_k } = R/J'$.
\end{proof}

As we said, we will use the same notation of \textsection{}\ref{soerg:sec:combinatorics}. For
$t^\up_{i,\ell}$, see \eqref{eq:151}, let
\begin{equation}
B_{t^\up_{i,\ell}} = B_{i+\ell-1} \otimes B_{i+\ell-2}
\otimes \cdots \otimes B_i\label{soerg:eq:32}
\end{equation}
and for $z=t^\up_{k,z^\up_k} \cdots
t^\up_{1,z^\up_1}$ let
\begin{equation}
B^\up_z = B_{t^\up_{1,z^\up_1}} \otimes
\cdots \otimes B_{t^\up_{k,z^\up_k}}.\label{soerg:eq:33}
\end{equation}
Moreover, for
$t^\down_{i,\ell}$, see \eqref{soerg:eq:25}, let
\begin{equation}
B_{t^\down_{i,\ell}} = B_{i-\ell} \otimes B_{i-\ell+1}
\otimes \cdots \otimes B_{i-1}\label{soerg:eq:37}
\end{equation}
and for $z=t^\down_{k+1,z^\down_{k+1}} \cdots
t^{\down}_{n,z^\down_n}$ let
\begin{equation}
B^\up_z = B_{t^\down_{n,z^\down_n}} \otimes
\cdots \otimes B_{t^\down_{k+1,z^\down_{k+1}}}.\label{soerg:eq:38}
\end{equation}

From Soergel Theorem~\ref{thm:2} and Proposition~\ref{soerg:prop:4},
it follows that $\mathsf C_{w_k  z}$ is isomorphic both to the
$B$--submodule of $ B^\up_z \otimes \mathsf{C}_{w_k }$ generated by
$\underline{1} = 1 \otimes \cdots \otimes 1$ and to the $B$--submodule
of $ B^\down_z \otimes \mathsf{C}_{w_k }$ generated by $\underline{1} =
1 \otimes \cdots \otimes 1$.

The following lemma is the crucial step to determine the annihilator
of $\mathsf C_{w_k  z}$.

\begin{lemma}\label{soerg:lem:ann-of-C-1}
  Let $z \in D $, and let $m$ be the number of nonzero
  $z^\up_i$'s. Then
  \begin{equation}
h_{\ell}(x_1,\ldots,x_{k-m+z^\up_{k-m+1}})\in \Ann \mathsf C_{w_k  z}\label{eq:6}
\end{equation}
 for all $\ell > m$.
\end{lemma}

\begin{proof}
  Let us prove the assertion by induction on the sum $N$ of the
  $z^\up_i$'s (that is, up to a shift, the length of $z$). If $N=0$
  then also $m=0$, and $h_\ell(x_1,\ldots,x_k)\in \Ann \mathsf C_{w_k }
  = (J_k,x_{k+1},\ldots,x_n)$ for every $\ell>1$ by Lemma
  \ref{soerg:lem:38}.

  For the inductive step, let $i=k-m+z^\up_{k-m+1}$, write $z=z's_{i}$
  and compute in $B \otimes_{i} ( B^\up_{z'} \otimes
  \mathsf{C}_{w_k })$:
  \begin{equation}\label{soerg:eq:159}
    \begin{split}
      h_{\ell+1}& (x_1,\ldots,x_{i}) \cdot (1 \otimes 1) \\
      &= \big ( P_i(h_{\ell+1}(x_1,\ldots,x_{i})) +x_i \partial_i
      (h_{\ell+1}(x_1,\ldots,x_{i})) \big)\cdot 1 \otimes 1\\
      &= 1 \otimes (P_i(h_{\ell+1}(x_1,\ldots,x_{i})) \cdot 1) + x_i
      \otimes (\partial_i (h_{\ell+1}(x_1,\ldots,x_{i})) \cdot 1).
    \end{split}
  \end{equation}
  Since $\sfC_{w_kz}$ is a summand of $B \otimes_i (B^\down_{z'} \otimes \sfC_{w_k z})$, it is sufficient to show that \eqref{soerg:eq:159} is zero. In fact, we prove that both terms $P_i(h_{\ell+1}(x_1,\ldots,x_i))$ and $\partial_i(h_{\ell+1}(x_1,\ldots,x_i))$ act as $0$ on $B^\up_{z'}$.

  Let us start considering the second term. Let $y \in D $ be determined by $y_i^\up = z_i^\up$ for $i \neq k-m+1,k-m+2$, while $y_{k-m+1}^\up=0$ and $y_{k-m+2}^\up = z_{k-m+1}^\up$. Notice that our chosen reduced expression for $z$ splits as $z'=yz''$, so that
  \begin{equation}
B^\up_{z'} = B_{i-1} B_{i-2} \cdots B_{k-m+1} \, B_{j} B_{j-1} \cdots B_{i+2} \, B^\up_y = B_{z''} B^\up_y\label{soerg:eq:160}
\end{equation}
for $j=k-m+1+z^{\up}_{k-m+2}$, where we omitted the tensor product signs.
By \eqref{soerg:eq:5}, $\partial_i(h_{\ell+1}(x_1,\ldots,x_i))=h_\ell
  (x_1,\ldots,x_{i+1})$; being symmetric in the variables $x_a$ for $a\neq i,i+1$, this passes through $B_{z''}$ and acts on $B^\up_y \otimes \sfC_{w_k }$. By induction, this action is zero.

  Now let us consider the action of the term $P_i(h_{\ell+1}(x_1,\ldots,x_i))$. Write
  \begin{equation}
  \begin{aligned}
    P_i(h_{\ell+1}(x_1,\ldots,x_{i})) & = h_{\ell+1}(x_1,\ldots,x_i) -
    x_i \partial_i h_{\ell+1}(x_1,\ldots,x_i)\\& =
    h_{\ell+1}(x_1,\ldots,x_i) - x_i h_{\ell}(x_1,\ldots,x_{i+1}).
  \end{aligned}\label{soerg:eq:182}
\end{equation}
  Of these two summands, the second acts as zero exactly as before. For the first one, write $y's_{i+1}=y$ so that $B^\up_y = B \otimes_{i+1} B^\up_{y'}$. Then $h_{\ell+1}(x_1,\ldots,x_i)$ commutes with the first tensor product, and by induction acts as zero on $B^\up_{y'}$.
\end{proof}

\begin{prop}\label{soerg:prop:ann-of-C}
  Let $z \in D $ with corresponding $\boldb$--sequence
  $\boldb^z$. Then the complete symmetric polynomial $
  h_{b^z_i} (x_1,\ldots,x_i)$ lies in $\Ann \mathsf C_{w_k
    z}$ for all $i=1,\ldots,n$.
\end{prop}

\begin{proof}
  We divide the indices $i \in \{1,\ldots,n\}$ corresponding to the
  positions in the $\up\down$--sequence of $z$ in three subsets: we
  call an index $i$ such that $z^\up_i=0$ \emph{initial}, we call an index $i$ for which $b^z_i=1$ \emph{final}, and we call all other indices in between \emph{in the middle}:
  \begin{equation*}
    \underbrace{\up \cdots \up}_{\textit{initial}} \underbrace{\down \cross \cdots \cross}_{\textit{middle}} \underbrace{\up \down \cdots \down}_{\textit{final}}
  \end{equation*}
  where a $\cross$ stands for a $\up$ or a $\down$. Notice that it can
  happen that an index $i$ is both \emph{initial} and {\em
    final} if and only if there are no $\down$'s, that is $k=n$. Since
  in this case we already know the claim, we can exclude it.

  If $i$ is \emph{final}, then $w_k z \in \bbS_i \subset \bbS_n$ (where
  $\bbS_i$ is the subgroup generated by the first $i-1$ simple
  transpositions) and obviously $h_1(x_1,\ldots,x_i)$ annihilates
  $\mathsf C_{w_k z}$.

  If $z$ is not the identity (in which case there are no indexes
  \emph{in the middle}), then $i=k-h+z^\up_{k-h+1}$ is \emph{in the
  middle}, and the statement of Lemma
  \ref{soerg:lem:ann-of-C-1} is that $h_{b^z_i}(x_1,\ldots,x_i) \in \Ann
  \mathsf C_{w_k  z}$. For the other indexes \emph{in the middle}, we can
  use Lemma~\ref{soerg:lem:ann-of-C-1} after letting
  $h_{b^z_i}(x_1,\ldots,x_i)$ commute with some initial tensor symbols
  of $B^\up_z$.

  If $i$ is \emph{initial}, then $z$ is a permutation in the
  subgroup of $\bbS_n$ generated by $s_{i+1},\ldots,s_{n-1}$, hence
  $h_{b^z_i}(x_1,\ldots,x_i)$, when acting on $B^\up_z \otimes \mathsf
  C_{w_k }$, can pass through $B^\up_z$, and we only need to prove that
  $h_{b^z_i}(x_1,\ldots,x_i) \in \Ann \mathsf C_{w_k }$. In fact,
  renaming the indexes this follows from the following general
  statement: $h_{a}(x_1,\ldots,x_\ell) \in J_k$ for every $1 \leq \ell
  \leq k$ and $a > k-\ell$. This well-known fact can easily be proved
  by (reversed) induction on $h$: if $h=k$ the claim is obvious; for
  the inductive step, just use \eqref{soerg:eq:3}.
\end{proof}

We now identify the Soergel modules with the rings $R_\boldb = R/I_\boldb$ defined in \textsection\ref{soerg:sec:ideal-gener-compl}.

\begin{theorem}\label{soerg:thm:1}
  Let $z \in D $ with corresponding $\boldb$--sequence
  $\boldb^z$. Then $\Ann \mathsf C_{w_k z}=I_{\boldb^z}$ and
  $ \mathsf C_{w_k x}\cong R_{\boldb^z}$. A basis of
  $R_{\boldb^z}$ is given by
  \begin{equation}
    \label{soerg:eq:14}
    \left\{x_1^{c_1}\cdots x_{n-1}^{c_{n-1}} \cdot \underline 1 \suchthat
      0 \leq c_{i} < b^z_i
    \right\}.
  \end{equation}
\end{theorem}

\begin{proof}
  Let $\boldb=\boldb^z$. By Proposition~\ref{soerg:prop:ann-of-C},
  $I_\boldb \subseteq \Ann \sfC_{w_k  z}$, so we have a surjective map
  $R/I_\boldb \twoheadrightarrow R/(\Ann \sfC_{w_k  z})$. By Proposition
  \ref{soerg:prop:1} and Lemma~\ref{soerg:lem:dim-of-special-soergel-modules}
  their dimension agree, hence $I_\boldb=\Ann \sfC_{w_k  z}$. The basis of
  $R_\boldb$ is given by Proposition~\ref{soerg:prop:1}.
\end{proof}

Translating Proposition~\ref{soerg:prop:2}, we can determine the homomorphism spaces between the Soergel modules $\sfC_{w_k  z}$:

\begin{corollary}
  \label{soerg:cor:8}
  Let $ z, z' \in D $ with $\boldb$--sequences $\boldb^z , \boldb^{z'}$. Let $c_i = \max \{b_i^{z'}-b_i^z,0\}$ for $i=1,\ldots,n-1$. Then a $\C$--basis of $\Hom_R (\sfC_{w_k  z}, \sfC_{w_k  z'})$ is given by
  \begin{equation}
    \label{soerg:eq:17}
    \{ 1 \mapsto x_1^{j_1}\cdots x_{n-1}^{j_{n-1}} \suchthat c_i \leq j_i < b_i^{z'} \}.
  \end{equation}
\end{corollary}

We will need in the following some other Soergel modules, corresponding to elements $w' \in \bbS_n$ which differ from some $w_k z$ only by a simple reflection, as in the following proposition:

\begin{prop}
  \label{soerg:prop:8}
  Let $z \in D $ with corresponding $\boldb$--sequence
  $\boldb^z$. Suppose $z^\down_j=z^\down_{j+1}$ for some
  index $j$. Let $\ell=j-z^\down_j$, so that $s_j z = z
  s_\ell$. Then $\sfC_{s_j w_k z}$ is the quotient of $R$
  modulo the ideal generated by the complete symmetric
  polynomials
  \begin{equation}
    \label{soerg:eq:45}
    h_{a_i}(x_1,\ldots,x_i) \qquad \text{for }i=1,\ldots,n,
  \end{equation}
  where $a_i = \boldb^z_i $ for $i \neq \ell$ while $a_\ell=\boldb^z_{\ell+1}$.
\end{prop}

Notice that the sequence $\bolda=(a_1,\ldots,a_n)$ is not an element of $\scrB'$, since $a_\ell = a_{\ell-1}+1$.

\begin{proof}
  The proof is analogous to the proof of Theorem~\ref{soerg:thm:1}.
  By Corollary~\ref{soerg:cor:1}, the module $\sfC_{s_j w_k  z}$ is
  cyclic. In particular, it is the submodule generated by $1$ inside
  $B \otimes_\ell \sfC_{w_k  z}$.  First, let us prove that the
  polynomials \eqref{soerg:eq:45} lie in $\Ann \sfC_{s_j w_k  z}$, or
  equivalently that they vanish on $B \otimes_\ell \sfC_{w_k  z}$. This is
  clear for $i \neq \ell$: in this case, these polynomials commute with the
  first tensor product, and then they vanish because they lie in $\Ann
  \sfC_{w_k  z}$ by Theorem~\ref{soerg:thm:1}. For the remaining case, we have
  \begin{multline}
    \label{soerg:eq:46}
    h_{a_\ell}(x_1,\ldots,x_\ell) \cdot (1 \otimes 1)\\
    = 1 \otimes (P_\ell(h_{a_\ell}(x_1,\ldots,x_\ell)) \cdot 1) + x_\ell \otimes (\partial_\ell (h_{a_\ell} (x_1,\ldots,x_\ell)) \cdot 1).
  \end{multline}
  By \eqref{soerg:eq:5} all terms contain $h_{a_\ell}(x_1,\ldots,x_\ell)$ or
  $h_{a_\ell - 1} (x_1,\ldots,x_{\ell+1})$, which both lie in $\Ann
  \sfC_{w_k  z}$, and we are done.

  It remains to prove that the polynomials \eqref{soerg:eq:45} are
  a set of generators. Let $I$ be the ideal generated by them. We know that
  $\sfC_{s_j w_k  z}$ is a quotient of $R/I$. As for Lemma~\ref{soerg:lem:8},
  the polynomials \eqref{soerg:eq:45} are a basis of $I$. As for Proposition
  \ref{soerg:prop:1}, the quotient $R/I$ has dimension $a_1\cdots a_n$. By
  Corollary~\ref{soerg:cor:1} and an argument similar to the proof of Lemma~\ref{soerg:lem:dim-of-special-soergel-modules}, this coincides with the dimension of
  $\sfC_{s_j w_k  z}$, and we are done.
\end{proof}

\subsection{Morphisms between Soergel modules}
\label{soerg:sec:morph-betw-soerg-1}

In each basis set \eqref{soerg:eq:17} there is exactly one
morphism of minimal degree, which we call the \emph{minimal
  degree morphism} $\sfC_{w_k z} \mapto \sfC_{w_k z'}$.  For
each $z \in D $, the vector space $\Hom_R (\mathsf{C}_{w_k
  z},\mathsf{C}_{w_k z})$ is a ring that is naturally
isomorphic to $\mathsf{C}_{w_k z}$. Moreover, for $z,z' \in
D $ the vector space $\Hom_R (\sfC_{w_k z}, \sfC_{w_k z'})$
is naturally a $(\sfC_{w_k z'}, \sfC_{w_k z})$--bimodule. It
follows directly from Corollary \ref{soerg:cor:8} that this
bimodule is cyclic (even more, it is cyclic both as a left
and as a right module), generated by the minimal degree
morphism. In what follows, we will often refer to this fact
saying that the minimal degree morphisms \emph{divides} all
other morphisms.

We let $D'$ be the set of shortest coset representatives for
$W_k^\perp\backslash \bbS_n$. In particular, for every $z
\in D $ we have $z,w_k z \in D'$.

\begin{definition}
\label{soerg:def:2}
  For $z, z' \in D $ we say that a morphism $\sfC_{w_k  z} \mapto \sfC_{w_k  z'}$ is \emph{illicit} if it factors through some $\sfC_{y}$, where $y$ is a longest coset representative for $W_k \backslash \bbS_n$ with $y \notin D'$.
\end{definition}

The definition is motivated by the following. Let
$\sfW_{z,z'}$ be the subspace of $\Hom_R (\sfC_{w_k
  z},\sfC_{w_k z'})$ consisting of all illicit morphisms;
since it is a $(\sfC_{w_k z'},\sfC_{w_k z})$--submodule, we
can define the quotient bimodule
  \begin{equation}
    \label{soerg:eq:60}
    \sfZ_{z,z'} = \Hom_R(\sfC_{w_k  z},\sfC_{w_k  z'}) / \sfW_{z,z'}.
  \end{equation}
This corresponds to the homomorphism space between projective modules in some parabolic category $\catO$ (see Section~\ref{sec:category-cato}).

\begin{lemma}
  \label{soerg:lem:14}
 Let $ z,  z' \in D $, and suppose that for some index $j$ we have
 \begin{equation}
   \label{soerg:eq:50}
   z'^\down_i =
   \begin{cases}
     z^\down_i + 1 & \text{for } i=j,j+1,\\
     z^\down_i & \text{otherwise}.
   \end{cases}
 \end{equation}
In particular $z'=z s_\ell s_{\ell+1}$ for $\ell=j-z^\down_j-1$, and the corresponding $\up\down$--sequence in positions $\ell,\ell+1,\ell+2$ are
 \begin{equation}
   \label{soerg:eq:49}
     z  = \cdots \up \down \down \cdots \qquad \text{and} \qquad
     z'  = \cdots \down \down \up \cdots.
 \end{equation}
Then   $\sfW_{ z,z'} = \Hom_R (\sfC_{w_k z} ,\sfC_{s_kz'})$ and $\sfW_{ z',z} = \Hom_R  (\sfC_{w_k z'} ,\sfC_{s_kz})$
\end{lemma}

\begin{proof}
  It is enough to show that $\phi \in \Hom_R(\sfC_{w_k  z},\sfC_{w_k  z'})$, $\phi: 1 \mapsto x_j x_{j+1} $ and $\psi \in \Hom_R(\sfC_{w_k  z'},\sfC_{w_k  z})$, $\psi: 1 \mapsto 1$ are illicit, since they divide all other morphisms. First of all, note that by construction
 \begin{equation}
   \label{soerg:eq:51}
   b^{z'}_i =
   \begin{cases}
     b^z_i + 1 & \text{for } i=\ell,\ell+1,\\
     b^z_i & \text{otherwise}.
   \end{cases}
 \end{equation}

Let $y=s_j z = z s_{\ell+1}$, and note that $y \notin D'$. We know $\sfC_{w_k y}$ by Proposition~\ref{soerg:prop:8}. Since $\Ann(\sfC_{w_k  z}) \subset \Ann(\sfC_{w_k  y}) \subset \Ann(\sfC_{w_k  z'})$, the morphism $\psi$ can be written as the composition of the natural quotient maps
\begin{equation}
  \label{soerg:eq:52}
  \sfC_{w_k  z'} \stackrel{1}{\longrightarrow} \sfC_{w_k  y} \stackrel{1}{\longrightarrow} \sfC_{w_k  z},
\end{equation}
hence it is illicit.

On the other side, $x_{\ell+1}\Ann(\sfC_{w_k  z}) \subseteq \Ann(\sfC_{w_k  y})$ because by \eqref{soerg:eq:3}
\begin{equation}
  \label{soerg:eq:53}
    x_{\ell+1} h_{b^z_{\ell+1}} (x_1,\ldots,x_{\ell+1}) = h_{b^z_{\ell+1}+1} (x_1,\ldots,x_{\ell+1}) - h_{b^z_{\ell+1}+1}(x_1,\ldots,x_{\ell})
\end{equation}
and $h_{b^z_{\ell+1}+1} (x_1,\ldots,x_\ell) \in \Ann(\sfC_{w_k  y})$ by the arguments of the proof of Lemma~\ref{soerg:lem:1}. Moreover, $x_\ell \Ann(\sfC_{w_k  y}) \subseteq \Ann(\sfC_{w_k  z'})$ because by \eqref{soerg:eq:3} we have
\begin{equation}
  \label{soerg:eq:54}
    x_{\ell} h_{b^z_{\ell}} (x_1,\ldots,x_{\ell}) = h_{b^z_{\ell}+1} (x_1,\ldots,x_{\ell}) - h_{b^z_{\ell}+1}(x_1,\ldots,x_{\ell-1})
\end{equation}
and this is in $\Ann(\sfC_{w_k  z'})$ by Lemma~\ref{soerg:lem:1}.
Hence $\phi$ can be written as the composition
\begin{equation}
  \label{soerg:eq:55}
  \sfC_{w_k  z} \xrightarrow{x_{\ell+1}}  \sfC_{w_k  y} \stackrel{x_\ell}{\longrightarrow} \sfC_{w_k  z'},
\end{equation}
and therefore is illicit.
\end{proof}

\begin{lemma}
  \label{soerg:lem:20}
  Let $z \in D $ and suppose $z^\down_j = z^\down_{j+1}$ for some index
  $j$. Let $\ell=j-z^\down_j$, so that $s_j z=z s_\ell$. Then the
  endomorphism $1 \mapsto x_\ell$ of $\sfC_{w_k  z}$ is illicit.
\end{lemma}

\begin{proof}
  Let $y = s_j w_k  z \notin D' $. We claim that
  $x_\ell \Ann (\sfC_{w_k  z}) \subseteq \Ann(\sfC_y)$ and hence that
  $1 \mapsto x_\ell$ defines a morphism $\sfC_{w_k  z} \mapto \sfC_y$.
  By Theorem~\ref{soerg:thm:1} and Proposition~\ref{soerg:prop:8} the only
  thing to check is that $x_\ell h_{b^z_\ell} (x_1,\ldots,x_\ell) \in
  \Ann(\sfC_y)$. By \eqref{soerg:eq:3} we have
  \begin{equation}
    \label{soerg:eq:47}
    x_\ell h_{b^z_\ell} (x_1,\ldots,x_\ell) = h_{b^z_\ell+1} (x_1,\ldots,x_\ell) - h_{b^z_\ell+1}(x_1,\ldots,x_{\ell-1}) \in \Ann(\sfC_y).
  \end{equation}

  On the other side, again by Theorem~\ref{soerg:thm:1} and Proposition
  \ref{soerg:prop:8}, it is clear that $1 \mapsto 1$ defines a morphism
  $\sfC_y \mapto \sfC_{w_k  z}$. Hence the endomorphism $1 \mapsto x_\ell$
  of $\sfC_{w_k  z}$ factors through $\sfC_y$ and is therefore illicit.
\end{proof}

\begin{lemma}
  \label{soerg:lem:22}
  Let $z \in D $. For every $j$ between $k+1$ and $n-1$ the morphism
  \begin{equation}
    \label{soerg:eq:48}
    1 \longmapsto x_{\ell} x_{\ell+1} \cdots x_{\ell'},
  \end{equation}
  where $\ell= j-z^\down_j$ and $\ell'=(j+1)-z^\down_{j+1} - 1$, is illicit.
\end{lemma}

\begin{proof}
  Let $y \in D $ be defined by $y^\down_i=z^\down_i$ for $i\neq j$,
  while $y^\down_j= z^\down_{j+1}$. From Corollary~\ref{soerg:cor:8} we have
  that $1 \mapsto 1$ and $1 \mapsto x_\ell x_{\ell+1} \cdots
  x_{\ell'-1}$ define morphisms $\sfC_{w_k  z} \mapto \sfC_{w_k  y}$ and
  $\sfC_{w_k  y} \mapto \sfC_{w_k  z}$ respectively. By Lemma~\ref{soerg:lem:20}
  the endomorphism $1 \mapto x_{\ell'}$ of $\sfC_{w_k  y}$ is illicit, and
  so is \eqref{soerg:eq:48}, since it can be expressed as composition of
  these three morphism.
\end{proof}

\begin{theorem}
  \label{soerg:thm:2}
 For all $z, z' \in D $ define a subbimodule $\widetilde\sfW_{z,z'}$ of the homomorphism space $\Hom_R(\sfC_{w_k z},\sfC_{w_k  z'})$ as follows:
 \begin{enumerate}[(i)]
 \item \label{soerg:item:20} if for some index $1 \leq j \leq n-k-1$ we have $\down^{z}_{j} \geq \down^{z'}_{j+1}$ or $\down^{z'}_j \geq \down^{z}_{j+1}$, then we set $\widetilde\sfW_{ z,z'}=\Hom (\sfC_{w_k z} ,\sfC_{w_k z'})$;
 \item \label{soerg:item:21} otherwise we define $\widetilde\sfW_{z,z'}$ to be the subbimodule generated by the morphisms
   \begin{equation}\label{soerg:eq:18}
     1 \mapsto (x_{\down^z_j} x_{\down^z_j+1} \cdots
     x_{\beta(j)}) (x_1^{c_1}\cdots x_{n-1}^{c_{n-1}}) \quad \text{for } 1 \leq j \leq n-k,
 \end{equation}
 where $c_i=\max \{b_i^{z'} -b_i^z,0\}$ and
 \begin{equation}
\beta(j) =
\begin{cases}
  \min\{\down^z_{j+1},\down^{z'}_{j+1}\}-1 & \text{if }j< n-k,\\
  n &\text{if }j=n-k.
\end{cases}
\label{eq:141}
\end{equation}
 \end{enumerate}
 Then we have $\widetilde\sfW_{z,z'} = \sfW_{z,z'}$.
\end{theorem}

\begin{example}\label{soerg:ex:2}
Let us consider the following example:
  \begin{center}
    \begin{tikzpicture}
      \matrix (A) [matrix of math nodes,row sep=0.2em,column sep= 0.1em,row 1/.style={color=black!50,font=\small},column 1/.style={text width=2em, minimum
height=1em,text centered}] at (0,0)
      {& 1 & 2 & 3 & 4 & 5 & 6 & 7 & 8 & 9 & 10 & 11 & 12 & 13 & 14 \\
       \boldb^z &10  & 10 & 9 & 8 & 7   & 6 & 5   & 5 & 5 & 4 & 3   & 2 & 2 & 1 \\
        z& \up   & \node[draw]{\down_1}; & \up & \up & \up   & \up & \up   & \node[draw,rounded corners]{\down_2}; & \node[draw,dashed]{\down_3}; & \up & \up   & \up & \node[draw,rounded corners, dashed]{\down_4}; & \up \\
        z'& \down_1 & \up & \up & \up & \node[draw,rounded corners]{\down_2}; & \up & \up & \up & \up & \up & \down_3 & \up & \up & \node[draw,rounded corners,dashed]{\down_4}; \\
        \boldb^{z'} &11  & 10 & 9 & 8 & 8   & 7 & 6   & 5 & 4 & 3 & 3   & 2 & 1 & 1 \\};
      \draw (A-1-2.south west) -- (A-1-15.south east);
      \draw [snake=brace,mirror snake] (A-5-3.south west) -- node[below] {$x_2 x_3 x_4 $} (A-5-6.south east);
      \draw [snake=brace,mirror snake] (A-5-10.south west) -- node[below] {$x_9 x_{10} x_{11} x_{12} $} (A-5-14.south east);
   \end{tikzpicture}
  \end{center}
For convenience we have written the subscripts of the $\down$'s, indicating their progressive number. We are in case \ref{soerg:item:21}, and the generating morphisms \eqref{soerg:eq:18} of $\widetilde\sfW_{z,z'}$ are

\begin{center}
  \renewcommand{\arraystretch}{1.1}
  \setlength{\tabcolsep}{10pt}
  \begin{tabular}[]{r|l}
    $j$ & morphism\\
    \hline
    $1$ & $1 \longmapsto (x_2 x_3 x_4 ) (x_1 x_5 x_6 x_7)$\\
    $2$ &   $1  \longmapsto (x_8 ) (x_1 x_5 x_6 x_7)$\\
    $3$ &   $1  \longmapsto (x_9 x_{10} x_{11} x_{12} ) (x_1 x_5 x_6 x_7)$\\
    $4$ &   $1  \longmapsto (x_{13}) (x_1 x_5 x_6 x_7)$
  \end{tabular}
\end{center}

In the picture, the case $j=1$ is highlighted in solid and the case $j=3$ is highlighted in dashed.
\end{example}

\begin{proof}[Proof of Theorem~\ref{soerg:thm:2}]
  First, we assume that $\down^z_j \geq \down^{z'}_{j+1}$ for some index $1 \leq j < n-k$. Pick $j$ minimal with this property. Notice that by the minimality of $j$ we have $\down^z_{j-1} < \down^{z'}_{j}$ (if $j>1$), and hence on the left of the $j$-th $\down$ of $z$ there is an $\up$ (this remains true also if $j=1$, since in this case $\down^z_1 \geq \down^{z'}_2 > \down^{z'}_1 \geq 1$). Let $\alpha = \down^z_{j}$ and $\ell= \down^z_{j+1} - \down^z_{j}$, and  define $z^{(1)}$ and $z^{(2)}$ as
\vspace{-2.5ex}  \begin{align}
    z & & \up \down \overbrace{\up \cdots \up \down}^{\ell} \\
    z^{(1)} & = z s_{\alpha+\ell-1} s_{\alpha+\ell-2} \cdots s_{\alpha+1} & \up \down \down \cdots \up \up \\
    z^{(2)}& = z^{(1)} s_{\alpha-1} s_h  & \down \down \up \cdots \up \up
\end{align}
where on the right we pictured the corresponding $\up\down$--sequences between positions $\alpha-1$ and $\alpha+\ell$ (and we included $z$ for clarity).
The composition
\begin{align}
  \label{soerg:eq:59}
  \sfC_{w_k  z} \xrightarrow{\phantom{1}1\phantom{1}} \sfC_{w_k  z^{(1)}} \xrightarrow{x_{\ell-1}x_\ell} \sfC_{w_k  z^{(2)}}
\end{align}
is illicit by Lemma~\ref{soerg:lem:14}.
Composing with the minimal degree morphism $\sfC_{w_k  z^{(2)}} \mapto \sfC_{w_k  z'}$ we obtain the minimal degree morphism $\sfC_{w_k  z} \mapto \sfC_{w_k  z'}$, which is therefore illicit. It follows that $\Hom_R(\sfC_{w_k  z},\sfC_{w_k  z'}) = \sfW_{ z,z'}$.

A straightforward dual argument (cf.\ \textsection\ref{soerg:sec:duality-1}) proves that $\Hom_R(\sfC_{w_k  z'},\sfC_{w_k  z}) = \sfW_{z',z}$. Swapping $z$ and $z'$ it follows that $\Hom_R(\sfC_{w_k  z},\sfC_{s_k z'})=\sfW_{z,z'}$ if $\down^{z'}_j \geq \down^z_{j+1}$.

  Now assume we are in case \ref{soerg:item:21} and fix an index $j$. First, let us consider the case $\down^{z'}_{j+1} < \down^z_{j+1}$, so that $\beta(j) = \down^{z'}_{j+1}$. Let $\gamma=\down^z_j$, $\delta=\down^{z'}_{j+1}$, $\epsilon= \down^z_{j+1}$. Define $z^{(1)}$, $z^{(2)}$ and $z^{(3)}$ by
  \begin{align}
    z & & \down \up \cdots \up \up \cdots \up \down \\
    z^{(1)} & = z s_{\gamma} s_{\gamma +1} \cdots s_{\delta-1} & \up \cdots \up \down \up \cdots \up \down\\
    z^{(2)}& = z^{(1)} s_{\epsilon-1} s_{\epsilon -2} \cdots s_{\delta+1} & \up \cdots \up \down \down \up \cdots \up\\
    z^{(3)}& = z^{(2)} s_{\delta-1} s_{\delta} & \up \cdots \down \down \up \up \cdots \up
\end{align}
where on the right we pictured the corresponding $\up\down$--sequences between positions $\gamma$ and $\epsilon$. The composition
\begin{equation}
  \label{soerg:eq:61}
  \sfC_{w_k  z} \xrightarrow{1} \sfC_{w_k  z^{(1)}} \xrightarrow {x_{\delta+1}x_{\delta+2} \cdots x_{\epsilon-1}} \sfC_{w_k  z^{(2)}} \xrightarrow{x_{\delta-1}x_\delta} \sfC_{w_k  z^{(3)}}
\end{equation}
is illicit by Lemma~\ref{soerg:lem:14}. By construction, the composition of \eqref{soerg:eq:61} with the minimal degree morphism $\sfC_{w_k  z^{(3)}} \mapto \sfC_{w_k  z'}$ equals the morphism \eqref{soerg:eq:18} from $\sfC_{w_k z } $ to $\sfC_{w_k  z'}$, that is therefore illicit.

Let us now consider the other case $\down^z_{j+1} \leq \down^{z'}_{j+1}$. By Lemma~\ref{soerg:lem:22} the endomorphism of $\sfC_{w_k  z}$ defined by
\begin{equation}
  \label{soerg:eq:62}
  1 \mapsto x_{\down^z_j} x_{\down^z_j + 1} \cdots x_{\down^z_{j+1}}
\end{equation}
is illicit. This morphism divides the morphism \eqref{soerg:eq:18}, which is therefore illicit.

To conclude the proof we are left to check that in case \ref{soerg:item:21} $\sfW_{z,z'} \subseteq \widetilde\sfW_{z,z'}$. Unfortunately we cannot check this directly. Instead, by Lemma~\ref{lem:17} in the next section we have that the dimensions of the quotients of $\Hom_R(\sfC_{w_k  z},\sfC_{z_k z'})$ by $\sfW_{z,z'}$ and $\widetilde\sfW_{z,z'}$ agree. This implies that $\sfW_{z,z'} = \widetilde\sfW_{z,z'}$.
\end{proof}

\subsection{Grading}
\label{soerg:sec:grading}

In order to keep the computations more transparent, we decided to
postpone the introduction of the grading until now. The ring $R$ is
graded with $\deg x_i = 2$. Since the ideal $J_n$ is homogeneous, $B$
is also graded, and the graded definition of the module $B_i$ is $B_i = B
\otimes_i B \langle -1 \rangle$. By Soergel theorems all $\sfC_w$ for $w \in \bbS_n$ are graded. In the graded version of the module
$\sfC_{w_k  z}$ the cyclic generator is in degree $--\len(w_k  z)$. Then \eqref{eq:3} has the following graded version:
  \begin{equation}
    \label{eq:5}
    \grdim_\C \sfC_z = v^{-\len(w_k z)} \sum_{w' \preceq z} \calP_{w',z}(v^2).
  \end{equation}

The spaces $\sfW_{z,z'}$ are homogeneous submodules, and the quotients $\sfZ_{z,z'}$ are then graded modules.

By our discussion in \textsection\ref{soerg:sec:duality-1}, and with the opportune degree shifting we put on the modules $\sfC_{w_k  z}$, it follows that all modules $\sfC_{w_k  z}$ are graded self-dual. In particular
\begin{equation}
  \label{soerg:eq:88}
  \Hom_R(\sfC_{w_k  z},\sfC_{w_k  z'}) \cong \Hom_R(\sfC_{w_k  z'},\sfC_{w_k  z})
\end{equation}
as graded vector spaces for all $z,z' \in D $. An explicit isomorphism was described in \eqref{soerg:eq:87}.

\section{The diagram algebra}
\label{sec:diagr-algebra-calq_k}
We want now to define some diagram algebras over $\C$,
analogous to the generalized Khovanov algebras defined in \cite{pre06126156}. We will use some diagrams which represents morphisms between the Soergel modules we studied in the previous section.

We point out that the major difficulty is the definition of the multiplication of two diagrams, which is not simply stacking one on the top of the other (as in many other diagram algebras), but instead a quite involved process. Brundan and Stroppel use Khovanov's TQFT to define this multiplication. Since there is not an analogous of such a TQFT in our case, we construct the multiplication in an indirect way using composition of morphisms between Soergel modules. A drawback of our definition of the multiplication is that it is not clear how one can define diagrammatically bimodules for the diagram algebra, as in \cite{MR2600694}.

Using the same techniques of \cite{pre06126156} we will describe explicitly the graded cellular structure (\textsection\ref{soerg:sec:cellularity}) and the graded properly stratifies structure (\textsection\ref{soerg:sec:prop-strat-struct}).

\subsection{Diagrams}
\label{soerg:sec:diagrams}
We start introducing the diagrams on which our algebras will be build. We
will redefine some keywords that are commonly used in Lie theory (such as
\emph{weight} and \emph{block}) in a diagrammatic sense. 

\subsubsection{Weights} A \emph{number line} $\bfL$ is a horizontal
line containing a finite number of \emph{vertices} indexed by a set of
consecutive integers in increasing order from left to right. Given a
number line, a \emph{weight} is obtained by labeling each of the
vertices by $\up$ or $\down$.

On the set of weights there is the partial order called \emph{Bruhat
  order}, generated by $\up \down \succ \down \up$. For weights
$\lambda, \mu$ declare $\lambda \sim \mu$ if $\mu$ can be obtained
from $\lambda$ by permuting $\up$'s and $\down$'s.

\subsubsection{Blocks}
\label{soerg:sec:blocks}
A \emph{block} $\Gamma$ is a $\sim$--equivalence class of weights. From
now on, let us fix a block $\Gamma$. Let also $k$ be the number of
$\up$'s and $n-k$ be the number of $\down$'s of any weight of
$\Gamma$.  The weights of $\Gamma$
can be identified with $\up\down$--sequences  in the sense of \textsection{}\ref{soerg:sec:combinatorics}, and hence with elements of $D_{n,k}$. For a weight $\lambda$, we can then define
as in \textsection{}\ref{soerg:sec:combinatorics} the position sequences
$(\up^\lambda_1,\ldots,\up^\lambda_k)$ and
$(\down^\lambda_1,\ldots,\down^\lambda_{n-k})$ and the $\boldb$--sequence
$\boldb^\lambda$.

\subsubsection{Enhanced weights}
\label{soerg:sec:enhanced-weights}

An \emph{enhanced weight} $\lambda^\sigma$ is a weight
$\lambda$ together with a bijection $\sigma$ between the
vertices labeled $\up$ in $\lambda$ and the set
$\{1,\ldots,k\}$. By numbering the $\up$'s from the left to
the right we may view $\sigma$ as en element in $\bbS_k$ and
call it the underlying permutation. We call $\lambda$ the
\emph{underlying weight}. We will also say that we obtain
the enhanced weight $\lambda^\sigma$ by enhancing the weight
$\lambda$ with the permutation $\sigma$. Notice that there
are exactly $k!$ enhanced weights with the same underlying
weight.

We define a partial order on the set of enhanced weights by the following rule:
\begin{equation}
  \label{soerg:eq:74}
  \lambda^\sigma \preceq \mu^\tau \Longleftrightarrow
    \lambda \prec \mu  \text{ or }
    ( \lambda = \mu \text{ and } \len(\sigma) \leq \len(\tau)).
\end{equation}

\subsubsection{Fork diagrams}
An \emph{$m$--fork} is a tree with a unique branching point (the \emph{root}) of valency $m$; the other $m$ vertices of the tree are called the \emph{leaves}. A $1$--fork will be also called a \emph{ray}. This is an example of a $5$--fork:
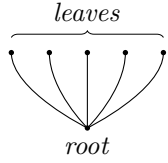
\begin{equation*}
  \tikzstyle{vertex}=[circle,fill, minimum size=2pt,inner sep=0pt]
  \begin{tikzpicture}
  \draw[] (0.500000,0) .. controls (0.500000,-0.500000) and (1.500000,-1.000000) ..  (1.500000,-1.000000);
 \draw[] (1.000000,0)  .. controls (1.000000,-0.500000) and (1.500000,-1.000000) ..  (1.500000,-1.000000);
 \draw[] (1.500000,0)  .. controls (1.500000,-0.500000) and (1.500000,-1.000000) ..  (1.500000,-1.000000);
 \draw[] (2.000000,0)  .. controls (2.000000,-0.500000) and (1.500000,-1.000000) ..  (1.500000,-1.000000);
 \draw[] (2.500000,0)  .. controls (2.500000,-0.500000) and (1.500000,-1.000000) ..  (1.500000,-1.000000);
      \draw [snake=brace,very thin,yshift=0.2cm] (0.5,0) -- node[above=0.1cm] {\emph{leaves}} (2.5,0);
      \node[below] at (1.5,-1) {\emph{root}};
 \node[vertex] at (0.5,0) {};
 \node[vertex] at (1,0) {} ;
 \node[vertex] at (1.5,0) {};
 \node[vertex] at (2,0) {};
 \node[vertex] at (2.5,0) {};
 \node[vertex] at (1.5,-1) {};
\end{tikzpicture}
\end{equation*}

Let $\bfV$ be the set of vertices of the number line $\bfL$, and let
$\bfH_-$ (resp.\ $\bfH_+$) be the half-plane below (resp.\ above)
$\bfL$.  A \emph{lower fork diagram} is a diagram made by the number
line $\bfL$ together with some forks contained in $\bfH_-$, such that
the leaves of each $m$--fork are $m$ distinct consecutive vertices in $\bfV$; we require each vertex of $\bfV$ to be a leaf of some fork. The forks and rays of a lover fork diagram will be also
called \emph{lower forks} and {\em lower rays}.

\emph{Upper rays}, \emph{upper forks} and \emph{upper fork diagrams} are defined in an analogous way.
If $c$ is a lower fork diagram, the mirror image $c^*$ through the
horizontal number line is an upper fork diagram, and vice versa. The
following are examples of a lower fork diagram $c$ and its mirror image $c^*$:
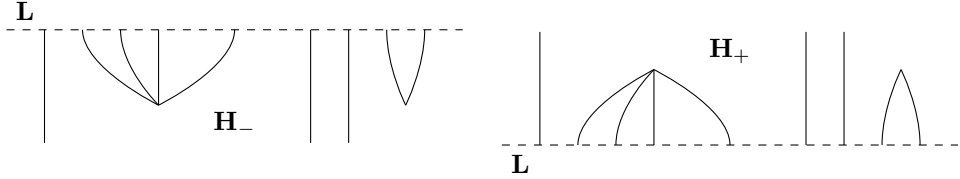
\begin{equation*}
  \begin{tikzpicture}
 \draw[dashed] (-0.5,0) node[above right] {$\bfL$} -- (5.5,0);
 \draw[] (0.000000,0) -- ++(0,-1.500000);
 \draw[] (0.500000,0) .. controls (0.500000,-0.500000) and (1.500000,-1.000000) ..  (1.500000,-1.000000);
 \draw[] (1.000000,0)  .. controls (1.000000,-0.500000) and (1.500000,-1.000000) ..  (1.500000,-1.000000);
 \draw[] (1.500000,0)  .. controls (1.500000,-0.500000) and (1.500000,-1.000000) ..  (1.500000,-1.000000);
 \draw[] (2.500000,0)  .. controls (2.500000,-0.500000) and (1.500000,-1.000000) ..  (1.500000,-1.000000);
 \draw[] (3.5,0) -- ++(0,-1.5);
 \draw[] (4.000000,0) -- ++(0,-1.500000);
 \draw[] (4.500000,0) .. controls (4.500000,-0.500000) and (4.750000,-1.000000) ..  (4.750000,-1.000000);
 \draw[] (5.000000,0)  .. controls (5.000000,-0.500000) and (4.750000,-1.000000) ..  (4.750000,-1.000000);
 \node at (2.5,-1.25) {$\bfH_-$};
  \end{tikzpicture}
  \hspace{0.5cm}
  \begin{tikzpicture}[baseline=0pt]
 \draw[dashed] (-0.5,0) node[below right] {$\bfL$} -- (5.5,0);
 \node at (2.5,+1.25) {$\bfH_+$};
 \draw[] (0.000000,0) -- ++(0,1.500000);
 \draw[] (0.500000,0)  .. controls (0.500000,0.500000) and (1.500000,1.000000) ..   (1.500000,1.000000);
 \draw[] (1.000000,0)  .. controls (1.000000,0.500000) and (1.500000,1.000000) ..   (1.500000,1.000000);
 \draw[] (1.500000,0)  .. controls (1.500000,0.500000) and (1.500000,1.000000) ..   (1.500000,1.000000);
 \draw[] (2.500000,0)  .. controls (2.500000,0.500000) and (1.500000,1.000000) ..   (1.500000,1.000000);
 \draw[] (3.5,0) -- ++(0,1.5);
 \draw[] (4.000000,0) -- ++(0,1.500000);
 \draw[] (4.500000,0)  .. controls (4.500000,0.500000) and (4.750000,1.000000) ..   (4.750000,1.000000);
 \draw[] (5.000000,0)  .. controls (5.000000,0.500000) and (4.750000,1.000000) ..   (4.750000,1.000000);
  \end{tikzpicture}
\end{equation*}

\subsubsection{Oriented diagrams} If $c$ is a lower fork diagram and
$\lambda$ is a weight with the same underlying number
line, we can glue them to obtain a diagram $c\lambda$. We call $c
\lambda$ an \emph{unenhanced oriented lower fork diagram} if:
\begin{itemize}
\item each $m$--fork for $m\geq 1$ is labeled with exactly one $\down$
  and $m-1$ $\up$'s;
\item the diagram begins at the left with a (possibly empty) sequence
  of rays labeled $\up$, and there are no other rays labeled $\up$
  in $c$.
\end{itemize}
Notice that by definition each $\up$ and $\down$ of
$\lambda$ labels some fork of $c$. Analogously, we call $\mu
d$ an \emph{unenhanced oriented upper fork diagram} if $d^*\!\mu$ is an unenhanced oriented lower fork diagram. The
\emph{orientation} of an unenhanced oriented lower (or
upper) fork diagram is the corresponding weight.

An (enhanced) \emph{oriented lower fork diagram} $c\lambda^\sigma$ is an unenhanced oriented lower fork diagram $c \lambda$ together with a permutation $\sigma \in \bbS_k$ such that $\lambda^\sigma$ is an enhanced weight. Similarly we define an (enhanced) \emph{oriented upper fork diagram}. If not explicitly specified, our oriented lower/upper fork diagrams will always be enhanced.

For $m\geq 1$ and $1 \leq i \leq m$ we define $\lambda(m,i)$ to be the weight formed by one $\down$ and
$m-1$ $\up$'s, where the $\down$ is at the $i$-th place. Note that a lower fork diagram $c$ consisting of only a lower $m$--fork
admits exactly $m!$ orientations, and they are exactly the
$\lambda(m,i)^\sigma$ for $i \in \{1,\ldots,m\}$, $\sigma \in \bbS_{m-1}$.

By a \emph{fork diagram} we mean a diagram of the form $ab$ obtained by
gluing a lower fork diagram $a$ underneath an upper fork diagram $b$,
assuming that they have the same underlying number lines. An {\em
  unenhanced oriented fork diagram} is a fork diagram $a\lambda b$
obtained by gluing an oriented lower fork diagram $a \lambda$ and an
oriented upper fork diagram $\lambda b$, as in the picture:

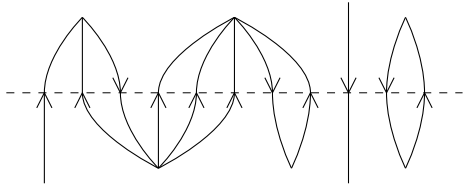
\begin{equation*}
  \begin{tikzpicture}
 \draw[dashed, very thin] (-0.5,0) -- (5.5,0);
 \draw[] (0.000000,0)  .. controls (0.000000,0.500000) and (0.500000,1.000000) ..   (0.500000,1.000000);
 \draw[] (0.500000,0)  .. controls (0.500000,0.500000) and (0.500000,1.000000) ..   (0.500000,1.000000);
 \draw[] (1.000000,0)  .. controls (1.000000,0.500000) and (0.500000,1.000000) ..   (0.500000,1.000000);
 \draw[] (1.500000,0)  .. controls (1.500000,0.500000) and (2.500000,1.000000) ..   (2.500000,1.000000);
 \draw[] (2.000000,0)  .. controls (2.000000,0.500000) and (2.500000,1.000000) ..   (2.500000,1.000000);
 \draw[] (2.500000,0)  .. controls (2.500000,0.500000) and (2.500000,1.000000) ..   (2.500000,1.000000);
 \draw[] (3.000000,0)  .. controls (3.000000,0.500000) and (2.500000,1.000000) ..   (2.500000,1.000000);
 \draw[] (3.500000,0)  .. controls (3.500000,0.500000) and (2.500000,1.000000) ..   (2.500000,1.000000);
 \draw[] (4.000000,0) -- ++(0,1.2);
 \draw[] (4.500000,0)  .. controls (4.500000,0.500000) and (4.750000,1.000000) ..   (4.750000,1.000000);
 \draw[] (5.000000,0)  .. controls (5.000000,0.500000) and (4.750000,1.000000) ..   (4.750000,1.000000);
 \draw[] (0.000000,0) -- ++(0,-1.2);
 \draw[] (0.500000,0) .. controls (0.500000,-0.500000) and (1.500000,-1.000000) ..  (1.500000,-1.000000);
 \draw[] (1.000000,0)  .. controls (1.000000,-0.500000) and (1.500000,-1.000000) ..  (1.500000,-1.000000);
 \draw[] (1.500000,0)  .. controls (1.500000,-0.500000) and (1.500000,-1.000000) ..  (1.500000,-1.000000);
 \draw[] (2.000000,0)  .. controls (2.000000,-0.500000) and (1.500000,-1.000000) ..  (1.500000,-1.000000);
 \draw[] (2.500000,0)  .. controls (2.500000,-0.500000) and (1.500000,-1.000000) ..  (1.500000,-1.000000);
 \draw[] (3.000000,0) .. controls (3.000000,-0.500000) and (3.250000,-1.000000) ..  (3.250000,-1.000000);
 \draw[] (3.500000,0)  .. controls (3.500000,-0.500000) and (3.250000,-1.000000) ..  (3.250000,-1.000000);
 \draw[] (4.000000,0) -- ++(0,-1.2);
 \draw[] (4.500000,0) .. controls (4.500000,-0.500000) and (4.750000,-1.000000) ..  (4.750000,-1.000000);
 \draw[] (5.000000,0)  .. controls (5.000000,-0.500000) and (4.750000,-1.000000) ..  (4.750000,-1.000000);
 \draw[yshift=-0.1cm] (-0.100000,-0.100000) -- (0.000000,0.100000) -- (0.100000,-0.100000);
 \draw[yshift=-0.1cm] (0.400000,-0.100000) -- (0.500000,0.100000) -- (0.600000,-0.100000);
 \draw[yshift=0.1cm] (0.900000,0.100000) -- (1.000000,-0.100000) -- (1.100000,0.100000);
 \draw[yshift=-0.1cm] (1.400000,-0.100000) -- (1.500000,0.100000) -- (1.600000,-0.100000);
 \draw[yshift=-0.1cm] (1.900000,-0.100000) -- (2.000000,0.100000) -- (2.100000,-0.100000);
 \draw[yshift=-0.1cm] (2.400000,-0.100000) -- (2.500000,0.100000) -- (2.600000,-0.100000);
 \draw[yshift=0.1cm] (2.900000,0.100000) -- (3.000000,-0.100000) -- (3.100000,0.100000);
 \draw[yshift=-0.1cm] (3.400000,-0.100000) -- (3.500000,0.100000) -- (3.600000,-0.100000);
 \draw[yshift=0.1cm] (3.900000,0.100000) -- (4.000000,-0.100000) -- (4.100000,0.100000);
 \draw[yshift=0.1cm] (4.400000,0.100000) -- (4.500000,-0.100000) -- (4.600000,0.100000);
 \draw[yshift=-0.1cm] (4.900000,-0.100000) -- (5.000000,0.100000) -- (5.100000,-0.100000);
  \end{tikzpicture}
\end{equation*}

An \emph{(enhanced) oriented fork diagram} is obtained by additionally enhancing the corresponding weight.

\subsubsection{Degrees}
Define the \emph{degree} of an unenhanced oriented lower (or upper) $m$--fork
by setting $\deg (c \lambda(m,i)) = \deg
(\lambda(m,i) c^*) = (i-1)$. Define then the
degree of an unenhanced oriented lower (resp.\ upper) fork diagram to be the sum of the degrees of all the lower (resp.\ upper) forks. Finally, the degree of an unenhanced oriented fork diagram $a \lambda b$ is
\begin{equation}
 \deg(a \lambda b) = \deg (a \lambda) + \deg (\lambda b ).\label{soerg:eq:105}
\end{equation}

Moreover, define the degree of a permutation $\sigma$ as $\deg(\sigma) = 2\len(\sigma)$. Then we define the degree of enhanced oriented diagrams by
\begin{align}
  \deg( a \lambda^\sigma) & = \deg(a \lambda) + \deg(\sigma),\\
  \deg(\lambda^\sigma b) & = \deg( \lambda b) + \deg(\sigma),\\
  \deg(a \lambda^\sigma b) & = \deg(a \lambda b) + \deg(\sigma) = \deg(a \lambda) + \deg (\lambda b)+ \deg(\sigma)
\end{align}
In particular, enhancing with the neutral element $e \in \bbS_k$ preserves the degree.

\begin{example}
  \label{soerg:ex:6}
  Consider the fork diagram $a \lambda b$ given by:
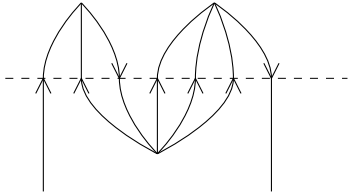
\begin{equation*}
  \begin{tikzpicture}
 \draw[dashed, very thin] (-0.5,0) -- (4,0);
 \draw[] (0.000000,0)  .. controls (0.000000,0.500000) and (0.500000,1.000000) ..   (0.500000,1.000000);
 \draw[] (0.500000,0)  .. controls (0.500000,0.500000) and (0.500000,1.000000) ..   (0.500000,1.000000);
 \draw[] (1.000000,0)  .. controls (1.000000,0.500000) and (0.500000,1.000000) ..   (0.500000,1.000000);
 \draw[] (1.500000,0)  .. controls (1.500000,0.500000) and (2.2500000,1.000000) ..   (2.2500000,1.000000);
 \draw[] (2.000000,0)  .. controls (2.000000,0.500000) and (2.2500000,1.000000) ..   (2.2500000,1.000000);
 \draw[] (2.500000,0)  .. controls (2.500000,0.500000) and (2.2500000,1.000000) ..   (2.2500000,1.000000);
 \draw[] (3.000000,0)  .. controls (3.000000,0.500000) and (2.2500000,1.000000) ..   (2.2500000,1.000000);
 \draw[] (0.000000,0) -- ++(0,-1.500000);
 \draw[] (0.500000,0) .. controls (0.500000,-0.500000) and (1.500000,-1.000000) ..  (1.500000,-1.000000);
 \draw[] (1.000000,0)  .. controls (1.000000,-0.500000) and (1.500000,-1.000000) ..  (1.500000,-1.000000);
 \draw[] (1.500000,0)  .. controls (1.500000,-0.500000) and (1.500000,-1.000000) ..  (1.500000,-1.000000);
 \draw[] (2.000000,0)  .. controls (2.000000,-0.500000) and (1.500000,-1.000000) ..  (1.500000,-1.000000);
 \draw[] (2.500000,0)  .. controls (2.500000,-0.500000) and (1.500000,-1.000000) ..  (1.500000,-1.000000);
 \draw[] (3.000000,0) -- ++(0,-1.5);
 \draw[yshift=-0.1cm] (-0.100000,-0.100000) -- (0.000000,0.100000) -- (0.100000,-0.100000);
 \draw[yshift=-0.1cm] (0.400000,-0.100000) -- (0.500000,0.100000) -- (0.600000,-0.100000);
 \draw[yshift=0.1cm] (0.900000,0.100000) -- (1.000000,-0.100000) -- (1.100000,0.100000);
 \draw[yshift=-0.1cm] (1.400000,-0.100000) -- (1.500000,0.100000) -- (1.600000,-0.100000);
 \draw[yshift=-0.1cm] (1.900000,-0.100000) -- (2.000000,0.100000) -- (2.100000,-0.100000);
 \draw[yshift=-0.1cm] (2.400000,-0.100000) -- (2.500000,0.100000) -- (2.600000,-0.100000);
 \draw[yshift=0.1cm] (2.900000,0.100000) -- (3.000000,-0.100000) -- (3.100000,0.100000);
  \end{tikzpicture}
\end{equation*}
We have
 $ \deg(a \lambda) = 1 $ and $\deg (\lambda b) = 2+3=5$,
so that $\deg(a \lambda b) = 6$. We can enhance the diagram with any permutation $\sigma \in \bbS_5$, and then $\deg(a \lambda^\sigma b) = 6 +2 \len(\sigma)$.
\end{example}

\subsubsection{The lower fork diagram associated to a
  weight} There is a natural way to associate a lower fork diagram to a weight $\lambda$:
\begin{lemma}
  For each weight $\lambda$ there is a unique lower fork
  diagram, denoted $\underline \lambda$, such that
  $\underline \lambda \lambda^e$ is an oriented lower fork
  diagram of degree $0$.\label{lem:1}
\end{lemma}

\begin{proof}
  Suppose that some oriented lower fork diagram $c
  \lambda^e$ of degree $0$ exists. Recall that, by the
  definition of orientation, each fork of $c$ is labeled by
  at most one $\down$ of $\lambda$; by the assumption on the
  degree, this $\down$ has to be the leftmost label of the
  corresponding fork. As a consequence, each $m$--fork of
  $c$, with the only exception of some initial rays labeled
  by $\up$, has to be labeled by the weight $\lambda(m,1)$.
  In other words, the lower fork diagram $c$ is obtained in
  the following way: examine the weight $\lambda$ from the
  left to the right and find all maximal subsequences
  consisting of a $\down$ followed by some (eventually
  empty) set of $\up$'s; draw a lower fork under each of
  these subsequences, and then draw lower rays under the
  remaining $\up$'s which are at the beginning of
  $\lambda$. Hence $c$ exists and is uniquely determined.
\end{proof}

Analogously we let $\overline
\lambda = (\underline \lambda)^*$ be the unique upper fork
diagram such that $\lambda^e \overline \lambda $ is an
oriented upper fork diagram of degree $0$.

\begin{example}
  \label{ex:2}
  As an example, let us illustrate the procedure of constructing $\underline \lambda$ for $\lambda = \up\up\down\up\up\up\down\down\up\down$. First, we circle all maximal subsequences consisting of a $\down$ followed by $\up$'s:
  \begin{equation*}
    \begin{tikzpicture}[baseline=0.8cm]
      \mydrawup{(0,1)};
      \mydrawup{(0.5,1)};
      \mydrawdown{(1,1)};
      \mydrawup{(1.5,1)};
      \mydrawup{(2,1)};  
    \mydrawup{(2.5,1)};
    \mydrawdown{(3,1)};
    \mydrawdown{(3.5,1)};
    \mydrawup{(4,1)};
    \mydrawdown{(4.5,1)};
    \draw [black!50,rounded corners = 2mm] (0.8,1.15) rectangle (2.7,0.65);
    \draw [black!50,rounded corners = 2mm] (2.8,1.15) rectangle (3.2,0.65);
    \draw [black!50,rounded corners = 2mm] (3.3,1.15) rectangle (4.2,0.65);
    \draw [black!50,rounded corners = 2mm] (4.3,1.15) rectangle (4.7,0.65);
  \end{tikzpicture}
\end{equation*}
Then we draw a lower fork under each of such subsequences, and lower rays under the remaining $\up$'s at the beginning of $\lambda$.
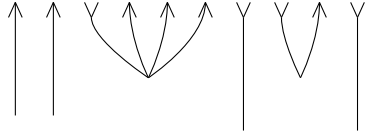
\begin{equation*}
  \begin{tikzpicture}
    \draw[] (0.000000,0) -- ++(0,-1.500000);
    \draw[] (0.500000,0) -- ++(0,-1.500000);
    \draw[] (1.000000,-0.2) .. controls (1.00000,-0.500000) and (1.75000000,-1.000000) ..  (1.75000000,-1.000000);
    \draw[] (1.500000,0)  .. controls (1.500000,-0.500000) and (1.75000000,-1.000000) ..  (1.75000000,-1.000000);
    \draw[] (2.00000,0)  .. controls (2.000000,-0.500000) and (1.75000000,-1.000000) ..  (1.75000000,-1.000000);
    \draw[] (2.50000,0)  .. controls (2.500000,-0.500000) and (1.75000000,-1.000000) ..  (1.75000000,-1.000000);
    \draw[] (3.000000,-0.2) -- ++(0,-1.500000);
    \draw[] (3.5000000,-0.2) .. controls (3.5000000,-0.500000) and (3.750000,-1.000000) ..  (3.750000,-1.000000);
    \draw[] (4.00000,0)  .. controls (4.000000,-0.500000) and (3.750000,-1.000000) ..  (3.750000,-1.000000);
    \draw[] (4.5000000,-0.2) -- ++(0,-1.500000); 
      \mydrawup{(0,0)};
      \mydrawup{(0.5,0)};
      \mydrawdown{(1,0)};
      \mydrawup{(1.5,0)};
      \mydrawup{(2,0)};  
    \mydrawup{(2.5,0)};
    \mydrawdown{(3,0)};
    \mydrawdown{(3.5,0)};
    \mydrawup{(4,0)};
    \mydrawdown{(4.5,0)};
  \end{tikzpicture}
\end{equation*}
The resulting lower fork diagram is then
  \begin{equation*}
   \underline \lambda =  \;\; \begin{tikzpicture}[baseline=-0.7cm]
    \draw[] (0.000000,0) -- ++(0,-1.500000);
    \draw[] (0.500000,0) -- ++(0,-1.500000);
    \draw[] (1.000000,0) .. controls (1.00000,-0.500000) and (1.75000000,-1.000000) ..  (1.75000000,-1.000000);
    \draw[] (1.500000,0)  .. controls (1.500000,-0.500000) and (1.75000000,-1.000000) ..  (1.75000000,-1.000000);
    \draw[] (2.00000,0)  .. controls (2.000000,-0.500000) and (1.75000000,-1.000000) ..  (1.75000000,-1.000000);
    \draw[] (2.50000,0)  .. controls (2.500000,-0.500000) and (1.75000000,-1.000000) ..  (1.75000000,-1.000000);
    \draw[] (3.000000,0) -- ++(0,-1.500000);
    \draw[] (3.5000000,0) .. controls (3.5000000,-0.500000) and (3.750000,-1.000000) ..  (3.750000,-1.000000);
    \draw[] (4.00000,0)  .. controls (4.000000,-0.500000) and (3.750000,-1.000000) ..  (3.750000,-1.000000);
    \draw[] (4.5000000,0) -- ++(0,-1.500000); 
  \end{tikzpicture}
\end{equation*}
\end{example}

For weights $\mu$ and $\lambda$, we use the notation $\mu \subset
\lambda$ to indicate that $\mu \sim \lambda$ and $\underline \mu
\lambda^e$ is an oriented lower fork diagram.

\begin{lemma}\label{soerg:lem:26}
  Let $\lambda,\mu$ be two weights in the same block $\Gamma$. If
  $\underline \lambda = \underline \mu$ then $\lambda = \mu$.
  If $\underline\mu \lambda$ is oriented then $\mu \preceq \lambda$ in the Bruhat order.
\end{lemma}

\begin{proof}
  Being in the same block, the weights $\lambda$ and $\mu$
  have the same number of $\up$'s and $\down$'s; let $h$ be
  the number of $\down$'s. Consider the $h$ rightmost forks
  of $\underline \lambda$ and let $a_1,\ldots,a_h$ be their
  initial positions; then $\lambda$ is uniquely determined
  by the condition of having $\down$'s in the positions
  $a_1,\ldots,a_h$ and $\up$'s elsewhere. Hence the first
  claim follows.

  Now, given the lower fork diagram $\underline \mu$, let
  $F_1,\ldots,F_h$ denote its $h$ rightmost forks. Let also
  $\Gamma_\mu = \{ \lambda\in \Gamma \suchthat \underline
  \mu \lambda \text{ is oriented}\}$. Then $\lambda \in
  \Gamma_\mu$ if and only if each $\down$ of $\lambda$
  labels exactly one of the $F_i$'s. Since $\mu$ is the
  weight of $\Gamma_\mu$ with the $\down$'s in the leftmost
  positions, it follows that $\mu$ is the minimal element in
  $\Gamma_\mu$ with respect to the Bruhat order.
\end{proof}

In particular, given our fixed block $\Gamma$, it follows that every
lower fork diagram $a$ (such that $a\mu$ is oriented for some $\mu \in
\Gamma$) determines a unique weight $\lambda$ with $\underline
\lambda=a$.
 In what follows, we will sometime interchange $a$ and
$\lambda$ in the notation: for example, we will write $\down_j^a$ for
$\down_j^\lambda$ or $\boldb^a$ for $\boldb^\lambda$ and so on.

We collect now some lemmas that we will need later.

\begin{lemma}
  \label{soerg:lem:21}
  Let $\lambda,\mu$ be two weights in the same block $\Gamma$.
  \begin{enumerate}[(i)]
  \item \label{soerg:item:23} The lower fork diagram $\underline \lambda \mu$ is oriented if
    and only if
    \begin{equation}
      \label{soerg:eq:21}
      \down^\lambda_i \leq \down^\mu_i < \down^{\lambda}_{i+1} \quad \text{for all }i \in 1,\ldots,n-k-1.
    \end{equation}
    \item \label{soerg:item:22} There exists an oriented fork diagram $\underline \lambda
    \eta \overline \mu$ for some $\eta \in \Gamma$ if and only if
    \begin{equation}
      \label{soerg:eq:132}
      \down^\lambda_i < \down^\mu_{i+1} \quad \text{and} \quad \down^{\mu}_{i} < \down^\lambda_{i+1} \quad \text{for all }i \in 1,\ldots,n-k-1.
    \end{equation}
  \end{enumerate}
\end{lemma}

\begin{proof}
  It is clear that \ref{soerg:item:22} follows from \ref{soerg:item:23}, so let us
  prove \ref{soerg:item:23}. It is easy to see that the lower fork
  diagram $\underline \lambda \mu$ is oriented if and only if each
  lower fork of $\underline \lambda$ is labeled by exactly
  one $\down$: this is exactly the same as \eqref{soerg:eq:21}.
\end{proof}

\begin{lemma}
  \label{soerg:lem:19}
  Consider weights $\lambda,\mu \in \Gamma$ with the corresponding $\boldb$--sequences $\boldb^\lambda,\boldb^\mu$.
  \begin{enumerate}[(a)]
  \item \label{soerg:item:1} If $\mu \succeq \lambda$, then $b^\mu_i \leq b^\lambda_i$ for all $i=1,\ldots,n$.
  \item \label{soerg:item:2} If $\underline \lambda \mu$ is oriented, then $b^\lambda_i - b^\mu_i \leq 1$ for all $i=1,\ldots,n$.
  \item \label{soerg:item:3} If $\underline \lambda \eta^e \overline \mu$ is oriented (for some weight $\eta \in \Gamma$), then $\abs{b^\lambda_i-b^\mu_i}\leq 1$ for all $i=1,\ldots,n$.
  \end{enumerate}
\end{lemma}

\begin{proof}
  If $\mu \succeq \lambda$ then the $i$-th $\down$ of $\mu$ is not on the right of the $i$-th $\down$ of $\lambda$, and the first claim follows.

  Let $\underline \lambda \mu$ be oriented. By Lemma~\ref{soerg:lem:21} we
  have $\down^\lambda_i \leq \down^\mu_i < \down^\lambda_{i+1}$.  This
  means that for every vertex $v \in \bfV$ there is at most one $\up$
  more to the right of $v$ in $\lambda$ than in $\mu$. This is exactly
  \ref{soerg:item:2}.

  The last claim follows from the second: if $\underline \lambda
  \eta^\sigma \overline \mu$ is oriented (for some weight $\eta$ with
  $\boldb$--sequence $\boldb^\eta$), then $b^\lambda_i -b^\mu_i = b^\lambda_i
  - b^\eta_i + b^\eta_i -b^\mu_i \in \{1-1,1+0,0-1,0+0\}$.
\end{proof}

Since we have identified $\Gamma$ with $D_{n,k}$, we can define the length $\len(\lambda)$ of any weight $\lambda \in \Gamma$ to be the length of the corresponding permutation in $D_{n,k}$.

\begin{lemma}
  \label{soerg:lem:24}
  Consider weights $\lambda,\eta$ in the same block $\Gamma$. Then
  \begin{equation}
    \label{soerg:eq:65}
    \deg (\underline \lambda \eta^\sigma ) = \len(\lambda)-\len(\eta)+ 2\len(\sigma).
  \end{equation}
\end{lemma}

\begin{proof}
   Since $\underline \lambda \eta$ is oriented, the weight $\eta$ is obtained from $\lambda$ permuting the $\up$'s and $\down$'s on each lower fork of $\underline \lambda$. The degree of $\underline \lambda \eta$ is the sum of how much each $\down$ of $\lambda$ has been moved to the right to reach the corresponding $\down$ of $\eta$; hence it is just the length of this permutation. In other words, if we let $z,z' \in D_{n,k}$  be the permutations corresponding to $\lambda,\eta$ respectively, then we have $z=z' y$ for some $y \in \bbS_n$ with $\len(z')=\len(z)+\len(y)$, and $\deg(\underline \lambda \eta) = \len(y)$.
\end{proof}

\subsection{The algebra structure}
\label{soerg:sec:algebra-structure}

We connect now our diagrams with the commutative algebra from Section~\ref{soerg:sec:soergel-modules}.
Let us fix a block $\Gamma$ with $k$ $\up$'s
and $n-k$ $\down$'s.

\subsubsection{Relations with polynomial rings}
\label{soerg:sec:relat-with-polyn}

We associate to the weight $\lambda$ the ring $R_{
  \lambda}=R_{\boldb^\lambda}=R/I_{\boldb^\lambda}$ (defined in \textsection\ref{soerg:sec:ideal-gener-compl}), and we want to describe $\sfZ_{z,z'}$ from \eqref{soerg:eq:60} diagrammatically.

Given an oriented lower fork diagram
$\underline \lambda \eta^\sigma$, we define the polynomial
\begin{equation}
p_{\underline  \lambda \eta^\sigma} =\mathfrak
S'_\sigma(x_{\up^\eta_1},\ldots,x_{\up^\eta_k}) \cdot \prod_{j=1}^{n-k} x_{\down^\lambda_j} x_{\down^\lambda_j+1} \cdots x_{\down^\eta_j-1} \in R.  \label{soerg:eq:1}
\end{equation}
with $\mathfrak
S'_\sigma(x_{\up^\eta_1},\ldots,x_{\up^\eta_k}) $ as defined
in \textsection\ref{sec:schubert-polynomials}. Notice that
the terms on the right always make sense because, since
$\underline \lambda \eta^\sigma$ is oriented, $\down_j^\eta
\geq \down_j^\lambda$ for all indices $j$ (cf.\
Lemma~\ref{soerg:lem:21}). Often we will consider
$p_{\underline \lambda \epsilon^\sigma}$ as a polynomial in
the quotient $R_\lambda$, but it will be convenient to have
a chosen lift in $R$. Notice that we have
\begin{equation}
  \label{soerg:eq:66}
  \deg(p_{\underline \lambda \eta^\sigma}) = 2(\len(\sigma) + \len(\lambda) - \len(\eta)).
\end{equation}

\begin{prop}
  \label{soerg:prop:6}
  Let $ \lambda,\mu \in \Gamma$ be weights,
and let $z, z' $ be the corresponding elements of $D$.
Let $\ucalZ_{\mu,\lambda}$ be the graded vector space with homogeneous basis
\begin{equation}
  \label{soerg:eq:64}
  \{ \underline \mu \eta^\sigma \overline \lambda \suchthat \underline \mu \eta^\sigma \overline \lambda \text{ is an oriented fork diagram} \}.
\end{equation}
With $\widetilde\sfW_{z,z'}$ as defined in Theorem~\ref{soerg:thm:2} we have an isomorphism of graded vector spaces
  \begin{equation}
    \label{soerg:eq:20}
    \begin{split}
     \Psi: \ucalZ_{\mu,\lambda} & \longrightarrow \Hom_R(\sfC_{w_k z},\sfC_{w_k z'})/ \widetilde\sfW_{z,z'}\\
      \underline \mu \eta^\sigma \overline \lambda & \longmapsto (1 \mapsto p_{\underline \mu \eta^\sigma}) + \widetilde\sfW_{z,z'}.
    \end{split}
  \end{equation}
\end{prop}

\begin{proof}
First, note that $p_{ \underline \mu \eta^\sigma} = p_{ \underline \mu \eta^e} \mathfrak S'_\sigma(x_{\up^\eta_1},\ldots,x_{\up^\eta_k})$. By definition we have $p_{\underline \mu \eta^e}=x_1^{\epsilon_1}\cdots x_n^{\epsilon_n}$ with $\epsilon_j=b^\mu_j-b^\eta_j$. By Lemma~\ref{soerg:lem:19}, $b^\lambda_j \geq b^\eta_j$ for every $j$, hence $\epsilon_j \geq b^\mu_j - b^\lambda_j$. By Corollary~\ref{soerg:cor:8}, the map $1 \mapsto p_{\underline \mu \eta^\sigma}$ induces a well-defined morphism in $\Hom_R (\sfC_{w_k  z}, \sfC_{w_k  z'})$, hence also in the quotient.

Let us show that \eqref{soerg:eq:20} is homogeneous of degree $0$. The degree of the morphism $1 \mapsto p_{\underline \mu \eta^\sigma}$ in $\Hom_R(\sfC_{w_k z}, \sfC_{w_k z'})$ is $\deg (p_{\underline \mu \eta^\sigma}) - \len(w_k  z') + \len(w_k  z)$, that is the same as $\deg (p_{\underline \mu \eta^\sigma}) - \len(z') + \len(z) = \deg (p_{\underline \mu \eta^\sigma}) - \len(\mu) + \len(\lambda)$. By \eqref{soerg:eq:66} this is $\len(\lambda) + \len(\mu) - 2\len(\eta) + 2 \len(\sigma)$. By Lemma~\ref{soerg:lem:24}, this is the same as $\deg( \underline \mu \eta^\sigma \overline \lambda )$.

Next, we want to see that $p_{ \underline \mu \eta^\sigma}$ is always a monomial of the basis \eqref{soerg:eq:17}. For that, note that by definition $\epsilon_j=1$ exactly when $b^\mu_j=b^\eta_j + 1$. Moreover, the monomial $\mathfrak S'_\sigma(x_{\up^\eta_1},\ldots,x_{\up^\eta_k})=x_{1}^{i_1}\cdots x_{n}^{i_n}$ is by construction in the basis of $R_{\eta}$, that means that $i_j < b^\eta_j$ for every $j$. It follows that $i_j + \epsilon_j < b^\mu_j$, hence $p_{\underline \mu \eta^\sigma}$ is a monomial of the basis \eqref{soerg:eq:17}.

We claim now that none of the $p_{\underline \mu \eta^\sigma}$ is in $\widetilde\sfW_{z,z'}$. Note that by construction the indeterminate $x_{\down^\eta_j}$ does not appear in $p_{\underline \mu \eta^\sigma}$. By Lemma~\ref{soerg:lem:21} we have $\down_j^\lambda \leq \down_j^\eta$ and both $\down_j^\eta < \down_{j+1}^\lambda$ and $\down_j^\eta < \down_{j+1}^\mu$. This means that both $x_{\down^\lambda_j} \cdots x_{\down^\lambda_{j+1}-1}$ and $x_{\down^\lambda_j} \cdots x_{\down^\mu_{j+1}-1}$ do not divide $p_{\underline \mu \eta^\sigma}$.

To conclude the proof, we need to construct an inverse of $\Psi$. Take
a basis monomial $\mathfrak m = x_1^{i_1}\cdots x_1^{i_n} \in \Hom_R
(\sfC_{w_k  z}, \sfC_{w_k  z'})$ that does not lie in $\widetilde\sfW_{z,z'}$. For
every $j$, let $\ell_j$ be the maximum such that $x_{\down_j^\mu}
x_{\down_j^{\mu}+1} \cdots x_{\ell_j-1}$ divide $\mathfrak m$. As
$\mathfrak m$ does not lie in $\widetilde\sfW_{z,z'}$, it should be $\ell_j <
\down_{j+1}^\lambda$ and $\ell_j < \down_{i+1}^\mu$. Form a weight
$\eta$ in the same block of $\lambda$ and $\mu$ with the $\down$'s in
positions $\ell_1,\ldots,\ell_{n-k}$. By Lemma~\ref{soerg:lem:21}
the diagram $\underline \mu \eta \overline \lambda$ is oriented. Let $\mathfrak
m'$ be the quotient of $\mathfrak m$ by $p_{\underline \mu
  \eta^e}$. By construction, $b^\mu_j = b^\eta_j$ if $x_j$ does not
appear in $p_{\underline \mu \eta^e}$, and $b^\mu_j = b^\eta_j+1$ if
$x_j$ appears (with coefficient $1$) in $p_{\underline \mu
  \eta^e}$. Hence, it is clear that $\mathfrak m'$ is a monomial
$\mathfrak S'_\sigma(x_{\up^\eta_1},\ldots,x_{\up^\eta_k})$. By
construction, it follows that in this way we get an inverse of the map
\eqref{soerg:eq:20}, that is hence an isomorphism.
\end{proof}

As a consequence we obtain the following result, that completes the proof of Theorem~\ref{soerg:thm:2}:

\begin{lemma}
  \label{lem:17}
  For all $z,z' \in D$ we have
  \begin{equation}
\dim_{\C} \Hom_R(\sfC_{w_k z},\sfC_{w_k z'})/\widetilde\sfW_{z,z'} = \dim_{\C} \sfZ_{z,z'}.\label{eq:161}
\end{equation}
\end{lemma}

\begin{proof}
  By Proposition~\ref{soerg:prop:6}, the dimension of $ \Hom_R(\sfC_{w_k z},\sfC_{w_k z'})/\widetilde\sfW_{z,z'}$ is the same as $\dim_\C \ucalZ_{\mu,\lambda}$, where $\lambda,\mu \in \Gamma$ are the weights corresponding to $z,z'$. This dimension is simply $k!$ times the number of unenhanced weights $\eta$ such that $\underline \mu \eta \overline \lambda$ is oriented. By Lemma \ref{lem:2} this is the same as $\dim \sfZ_{z,z'}$.
\end{proof}

Being $\Gamma$ and $D_{n,k}$ identified, we will often write
$\sfC_\lambda$ for $\sfC_{w_k z}$, where $z \in D_{n,k}$ is
the element corresponding to $\lambda$. If $a=\underline
\lambda$ and $b=\overline \lambda$ we will even write
$\sfC_a$ or $\sfC_b$ instead of $\sfC_\lambda$. We will do
similarly for $\sfW_{z,z'}$ and $\sfZ_{z,z'}$.

\subsubsection{The algebra structure}
\label{soerg:sec:algebra-structure-1}

Thanks to Proposition~\ref{soerg:prop:6}, we can define a graded algebra
$A=A_\Gamma$ over $\C$. As a graded vector space, a homogeneous basis is given
by
\begin{equation}
 \{(\underline \alpha \lambda^\sigma \overline \beta) \suchthat \text{for all }
\alpha, \lambda, \beta \in \Gamma , \sigma \in \bbS_k\text{ such that } \alpha \supset
\lambda \subset \beta \}\label{soerg:eq:67}
\end{equation}
 that is the same as
 \begin{equation}
 \{(a \lambda^\sigma b) \suchthat \text{for all oriented fork diagrams }
a\lambda b \text{ with } \lambda \in \Gamma \}.\label{soerg:eq:68}
\end{equation}

The degree on this basis is given by the degrees on fork
diagrams. For $\lambda \in \Gamma$ we write $e_\lambda$ for
$(\underline \lambda \lambda \overline \lambda)$. Note that the
vectors $e_\lambda$ give a basis for the degree $0$ component of $A$.

\begin{example}
  \label{ex:3}
  Let us consider a block $\Gamma$ of weights with 2 $\up$'s and 1 $\down$, that is
  \begin{equation}
    \label{eq:30}
    \Gamma= \{\lambda_1 = \up \up \down, \;\lambda_2=\up \down \up,\; \lambda_3 = \down \up \up\}.
  \end{equation}
 Then the basis $\{e_{\lambda_i}\}$ of the degree $0$ component is given by
 \begin{equation}
   \label{eq:31}
   e_{\lambda_1} =\,\,
   \begin{tikzpicture}[baseline=-0.16cm]
     \draw[] (3.000000,-0.8) -- (3.0,0.60000);
     \draw[] (3.500000,-0.8) -- (3.5,0.600000);
     \draw[] (4.000000,-0.8) -- (4,0.600000);
    \mydrawup{(3,0)};
    \mydrawup{(3.5,0)};
    \mydrawdown{(4,0)};
   \end{tikzpicture}\,
,\quad   e_{\lambda_2} =\,\,
   \begin{tikzpicture}[baseline=-0.16cm]
     \draw[] (3.000000,-0.8) -- (3.0,0.60000);
    \draw[yscale=0.7] (3.5000000,-0.2) .. controls (3.5000000,-0.500000) and (3.750000,-1.000000) ..  (3.750000,-1.000000);
    \draw[yscale=0.7] (4.00000,0)  .. controls (4.000000,-0.500000) and (3.750000,-1.000000) ..  (3.750000,-1.000000);
    \mydrawup{(3,0)};
    \mydrawdown{(3.5,0)};
    \mydrawup{(4,0)};
    \draw[yscale=0.7] (3.5000000,-0.2) .. controls (3.5000000,0.500000) and (3.750000,0.8500000) ..  (3.750000,0.8500000);
    \draw[yscale=0.7] (4.00000,0)  .. controls (4.000000,0.500000) and (3.750000,0.8500000) ..  (3.750000,0.8500000);
   \end{tikzpicture}\,,
\quad   e_{\lambda_3} =\,\,
   \begin{tikzpicture}[baseline=-0.16cm]
    \draw[yscale=0.7] (3.000000,-0.2) .. controls (3.000000,-0.500000) and (3.50000,-1.000000) ..  (3.50000,-1.000000);
    \draw[yscale=0.7] (3.5000000,0) .. controls (3.5000000,-0.500000) and (3.50000,-1.000000) ..  (3.50000,-1.000000);
    \draw[yscale=0.7] (4.00000,0)  .. controls (4.000000,-0.500000) and (3.50000,-1.000000) ..  (3.50000,-1.000000);
    \mydrawdown{(3,0)};
    \mydrawup{(3.5,0)};
    \mydrawup{(4,0)};
    \draw[yscale=0.7] (3.000000,-0.2) .. controls (3.000000,0.500000) and (3.50000,0.8500000) ..  (3.50000,0.8500000);
    \draw[yscale=0.7] (3.5000000,0) .. controls (3.5000000,0.500000) and (3.50000,0.8500000) ..  (3.50000,0.8500000);
    \draw[yscale=0.7] (4.00000,0)  .. controls (4.000000,0.500000) and (3.50000,0.8500000) ..  (3.50000,0.8500000);
   \end{tikzpicture}
 \end{equation}
\end{example}

From Proposition~\ref{soerg:prop:6} we get the following:
\begin{corollary}\label{soerg:cor:5}
  There is an isomorphism of graded vector spaces
  \begin{equation}
    A \cong \bigoplus_{z,z' \in D} \Hom(\mathsf C_{w_k z},\mathsf C_{w_k  z'})/\sfW_{z,z'}.\label{soerg:eq:69}
  \end{equation}
  This defines a graded algebra structure on $A$.\label{soerg:cor:2}
\end{corollary}

The product of two basis vectors of $A$ can be
computed explicitly using the isomorphism
\eqref{soerg:eq:69} as explained in details in the following
Remark~\ref{soerg:rem:2}. Unfortunately, we are not able to describe the
multiplication in the algebra $A$ purely in terms of
diagrams. Nevertheless, the diagrammatic
description proves useful to find other properties of the
algebra $A$, as we will explain in the following.

\begin{remark}
  Explicitly, the multiplication of the basis vectors $(a
  \lambda^\sigma b)$ and $(c \mu^{\sigma'} d)$ can be computed in the following
  way. First, if $b^* \neq c$ then set it to be zero. Now suppose $b =
  c^*$. Then take $p_{c \mu^{\sigma'}}$ and $p_{a \lambda^{\sigma}}$ in $R$
  and multiply them. By construction,
  the result gives a well defined morphism of the corresponding
  Soergel modules: write it as a linear combination of the basis
  \eqref{soerg:eq:17} and translate it in the diagrammatic algebra $A$ using
  the isomorphism of Proposition~\ref{soerg:prop:6}.\label{soerg:rem:2}
\end{remark}

\begin{example}
  \label{soerg:ex:7}
  Let
  \begin{equation*}
    a \lambda b =
    \begin{tikzpicture}[baseline=-0.5ex]
      \draw[dashed] (-0.5,0) -- (2.5,0);
 \draw[] (0.000000,0)  .. controls (0.000000,0.500000) and (0.500000,1.000000) ..   (0.500000,1.000000);
 \draw[] (0.500000,0)  .. controls (0.500000,0.500000) and (0.500000,1.000000) ..   (0.500000,1.000000);
 \draw[] (1.000000,0)  .. controls (1.000000,0.500000) and (0.500000,1.000000) ..   (0.500000,1.000000);
 \draw[] (1.500000,0)  .. controls (1.500000,0.500000) and (1.750000,1.000000) ..   (1.750000,1.000000);
 \draw[] (2.000000,0)  .. controls (2.000000,0.500000) and (1.750000,1.000000) ..   (1.750000,1.000000);
 \draw[] (0.000000,0) .. controls (0.000000,-0.500000) and (0.750000,-1.000000) ..  (0.750000,-1.000000);
 \draw[] (0.500000,0)  .. controls (0.500000,-0.500000) and (0.750000,-1.000000) ..  (0.750000,-1.000000);
 \draw[] (1.000000,0)  .. controls (1.000000,-0.500000) and (0.750000,-1.000000) ..  (0.750000,-1.000000);
 \draw[] (1.500000,0)  .. controls (1.500000,-0.500000) and (0.750000,-1.000000) ..  (0.750000,-1.000000);
 \draw[] (2.000000,0) -- ++(0,-1.500000);
 \draw[,yshift=-0.100000 cm] (-0.100000,-0.100000) -- (0.000000,0.100000) -- (0.100000,-0.100000);
 \draw[xshift=0.5cm,yshift=-0.100000 cm] (0.400000,-0.100000) -- (0.500000,0.100000) -- (0.600000,-0.100000);
 \draw[xshift=-0.5cm,yshift=0.100000 cm] (0.900000,0.100000) -- (1.000000,-0.100000) -- (1.100000,0.100000);
 \draw[,yshift=-0.100000 cm] (1.400000,-0.100000) -- (1.500000,0.100000) -- (1.600000,-0.100000);
 \draw[,yshift=0.100000 cm] (1.900000,0.100000) -- (2.000000,-0.100000) -- (2.100000,0.100000);
\end{tikzpicture} \qquad \text{and} \qquad
c \mu d =
\begin{tikzpicture}[baseline=-0.5ex]
      \draw[dashed] (-0.5,0) -- (2.5,0);
 \draw[] (0.000000,0)  .. controls (0.000000,0.500000) and (0.250000,1.000000) ..   (0.250000,1.000000);
 \draw[] (0.500000,0)  .. controls (0.500000,0.500000) and (0.250000,1.000000) ..   (0.250000,1.000000);
 \draw[] (1.000000,0)  .. controls (1.000000,0.500000) and (1.500000,1.000000) ..   (1.500000,1.000000);
 \draw[] (1.500000,0)  .. controls (1.500000,0.500000) and (1.500000,1.000000) ..   (1.500000,1.000000);
 \draw[] (2.000000,0)  .. controls (2.000000,0.500000) and (1.500000,1.000000) ..   (1.500000,1.000000);
 \draw[] (0.000000,0) .. controls (0.000000,-0.500000) and (0.500000,-1.000000) ..  (0.500000,-1.000000);
 \draw[] (0.500000,0)  .. controls (0.500000,-0.500000) and (0.500000,-1.000000) ..  (0.500000,-1.000000);
 \draw[] (1.000000,0)  .. controls (1.000000,-0.500000) and (0.500000,-1.000000) ..  (0.500000,-1.000000);
 \draw[] (1.500000,0) .. controls (1.500000,-0.500000) and (1.750000,-1.000000) ..  (1.750000,-1.000000);
 \draw[] (2.000000,0)  .. controls (2.000000,-0.500000) and (1.750000,-1.000000) ..  (1.750000,-1.000000);
 \draw[,yshift=-0.100000 cm] (-0.100000,-0.100000) -- (0.000000,0.100000) -- (0.100000,-0.100000);
 \draw[,yshift=0.100000 cm] (0.400000,0.100000) -- (0.500000,-0.100000) -- (0.600000,0.100000);
 \draw[,yshift=-0.100000 cm] (0.900000,-0.100000) -- (1.000000,0.100000) -- (1.100000,-0.100000);
 \draw[xshift=0.5cm,yshift=-0.100000 cm] (1.400000,-0.100000) -- (1.500000,0.100000) -- (1.600000,-0.100000);
 \draw[xshift=-0.5cm,yshift=0.100000 cm] (1.900000,0.100000) -- (2.000000,-0.100000) -- (2.100000,0.100000);
\end{tikzpicture}
\end{equation*}
Let also $\sigma= s_1\in \bbS_3$, $\tau=e \in \bbS_3$. We
want to compute the product $(a \lambda^\sigma b)(c \mu^\tau
d)$. First notice that $b^*=c$ (otherwise the product would
be trivially zero). By \eqref{soerg:eq:1} we have
\begin{align}
  \label{soerg:eq:187}
  p_{a \lambda^\sigma} & = x_1 \cdot x_1 x_4\\
  p_{c \mu^\tau d} & = 1 \cdot x_1
\end{align}
(for the computation of the polynomials $\mathfrak S'_\sigma$ and $\mathfrak S'_\tau$ we refer to Example~\ref{soerg:ex:5}). The product is $p_{a \lambda^\sigma} p_{c \mu^\tau} = x_1^3 x_4$. The $\boldb$--sequence of $a$ is $(4,3,2,1,1)$, hence $x_1^3 x_4$ is not an element of the monomial basis \eqref{soerg:eq:17} of $R_a$. We need to do some computations in the ring $R_a$:  using the relations $x_1+x_2+x_3+x_4 \equiv 0$ and $x_1^4 \equiv 0$ we have
\begin{equation}
  \label{soerg:eq:188}
  x_1^3 x_4 \equiv - x_1^4  - x_1^3 x_2 - x_1^3 x_3 \equiv -x_1^3 x_2 -x_1^3 x_3.
\end{equation}
This is now a linear combination of monomials of the basis
\eqref{soerg:eq:17}. The monomial $- x_1^3 x_2$, although
not zero in $R_a$, is of type \eqref{soerg:eq:18}, hence
defines an illicit morphism and is zero in the quotient. We
are left only with the monomial $\frakm = x_1^3 x_3$. This
is an element of \eqref{soerg:eq:17} and, according to
Theorem~\ref{soerg:thm:2}, does not define an illicit
morphism. We need to translate it into a diagram via
Proposition~\ref{soerg:prop:6}. The $\up\down$--sequence
corresponding to $a$ is $\down \up \up \up \down$; in
particular, the indices of the $\down$'s are $1,5$. Now,
$x_5$ does not divide $\frakm$, and the biggest index $i$
such that $x_1x_2\cdots x_i \mid \frakm$ is $1$. Hence the
monomial $\frakm$ corresponds to a diagram $a \eta^\pi d$
where $\eta$ has $\down$'s in positions $2,5$. Moreover, the
permutation $\pi$ is determined by $\mathfrak
S'_\pi(x_1,x_3,x_4)= x_1^2 x_3$. By
Example~\ref{soerg:ex:5}, $\pi$ is the longest element of
$\bbS_3$. Hence $(a \lambda^\sigma b)(c\mu^\tau d)= -(a
\eta^\pi d)$, where
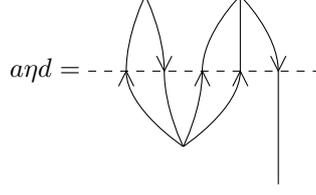
\begin{equation*}
  a \eta d =
\begin{tikzpicture}[baseline=-0.5ex]
  \draw[dashed] (-0.5,0) -- (2.5,0);
 \draw[] (0.000000,0)  .. controls (0.000000,0.500000) and (0.250000,1.000000) ..   (0.250000,1.000000);
 \draw[] (0.500000,0)  .. controls (0.500000,0.500000) and (0.250000,1.000000) ..   (0.250000,1.000000);
 \draw[] (1.000000,0)  .. controls (1.000000,0.500000) and (1.500000,1.000000) ..   (1.500000,1.000000);
 \draw[] (1.500000,0)  .. controls (1.500000,0.500000) and (1.500000,1.000000) ..   (1.500000,1.000000);
 \draw[] (2.000000,0)  .. controls (2.000000,0.500000) and (1.500000,1.000000) ..   (1.500000,1.000000);
 \draw[] (0.000000,0) .. controls (0.000000,-0.500000) and (0.750000,-1.000000) ..  (0.750000,-1.000000);
 \draw[] (0.500000,0)  .. controls (0.500000,-0.500000) and (0.750000,-1.000000) ..  (0.750000,-1.000000);
 \draw[] (1.000000,0)  .. controls (1.000000,-0.500000) and (0.750000,-1.000000) ..  (0.750000,-1.000000);
 \draw[] (1.500000,0)  .. controls (1.500000,-0.500000) and (0.750000,-1.000000) ..  (0.750000,-1.000000);
 \draw[] (2.000000,0) -- ++(0,-1.500000);
 \draw[,yshift=-0.100000 cm] (-0.100000,-0.100000) -- (0.000000,0.100000) -- (0.100000,-0.100000);
 \draw[,yshift=0.100000 cm] (0.400000,0.100000) -- (0.500000,-0.100000) -- (0.600000,0.100000);
 \draw[,yshift=-0.100000 cm] (0.900000,-0.100000) -- (1.000000,0.100000) -- (1.100000,-0.100000);
 \draw[,yshift=-0.100000 cm] (1.400000,-0.100000) -- (1.500000,0.100000) -- (1.600000,-0.100000);
 \draw[,yshift=0.100000 cm] (1.900000,0.100000) -- (2.000000,-0.100000) -- (2.100000,0.100000);
\end{tikzpicture}
\end{equation*}
\end{example}

\medskip
By construction, $p_{\underline \lambda \lambda^e} = 1$ for any $\lambda \in \Gamma$. Under the isomorphism of Proposition~\ref{soerg:prop:6}, the element $e_\lambda$ is sent to $\id_{\sfC_{w_k  z}} \in \End_R(\sfC_{s_k z})$, where $z \in D$ corresponds to $\lambda$; hence the elements $e_\lambda$ satisfy
\begin{equation}
  \label{soerg:eq:70}
  e_\lambda (a \mu^\sigma b) =
  \begin{cases}
    a \mu^\sigma b & \text{if } a= \overline \lambda, \\
    0  & \text{otherwise},
  \end{cases}
  \qquad
  (a \mu^\sigma b) e_\lambda =
  \begin{cases}
    a \mu^\sigma b & \text{if } b= \underline \lambda, \\
    0  & \text{otherwise}
  \end{cases}
\end{equation}
for any basis element $a \mu^\sigma b \in A$. That is, the vectors $\{e_\lambda \suchthat \lambda \in \Gamma\}$ are mutually orthogonal idempotents whose sum is the identity $1 \in A$. The decomposition \eqref{soerg:eq:69} can be written as
\begin{equation}
  \label{soerg:eq:71}
  A = \bigoplus_{\lambda,\mu \in \Gamma} e_\lambda A e_\mu.
\end{equation}
A basis of the summand $e_\lambda A e_\mu$ is
\begin{equation}
  \label{soerg:eq:72}
  \{\underline \lambda \eta^\sigma \overline \mu \suchthat \text{ for all }\eta \in \Gamma, \sigma \in \bbS_k \text{ such that } \lambda \supset \eta \subset \mu  \}.
\end{equation}

\subsubsection{Duality}
\label{sec:duality}
Recall from \textsection\ref{soerg:sec:duality-1} that for every $z,z' \in D$ we have an isomorphism
\begin{equation}
  \begin{aligned}
    \Theta: \Hom_R(\sfC_{w_k  z}, \sfC_{w_k z'})& \longrightarrow
    \Hom_R(\sfC_{w_k  z'},\sfC_{w_k  z}),\\
    (1 \mapsto p) & \longmapsto (1 \mapsto \boldx^{\boldsymbol{b}-\boldsymbol{b'}}p ),
  \end{aligned}
\label{soerg:eq:56}
\end{equation}
where $\boldb$ and $\boldb'$ are the $\boldb$--sequences of $z$ and $z'$,
respectively, $\boldsymbol{b}-\boldsymbol{b'}=
(b_1-b'_1,\ldots,b_n-b'_n)$ and the notation is as in
\eqref{soerg:eq:87}.

\begin{lemma}
  \label{soerg:lem:27}
  Let $\lambda,\mu \in \Gamma$ and let $z,z'$ be the corresponding
  elements of $D_{n,k}$. We have
  $\Theta(\sfW_{z,z'})=\sfW_{z',z}$. Therefore the isomorphism
  $\Theta$ descends to an isomorphism $\Theta: \sfZ_{z,z'} \mapto
  \sfZ_{z',z}$ that fits with the duality on diagrams:
  \begin{equation}
    \label{soerg:eq:40}
    \Theta(\Psi(\underline \mu \eta^\sigma \overline \lambda)) = \Psi(\underline \lambda \eta^\sigma \overline \mu)
  \end{equation}
  for every enhanced weight $\eta^\sigma$ such that $\underline \mu \eta^\sigma \overline \lambda$ is oriented.
\end{lemma}

\begin{proof}
  Let $\boldb$, $\boldb'$ be the $\boldb$--sequences of $\lambda $ and $\mu$, respectively. Note that
  \begin{equation}
    \label{soerg:eq:90}
    \frac{\boldx^{\boldsymbol{b-1}}}{\boldx^{\boldsymbol{b'-1}}} = \boldx^{\boldsymbol{b-b'}} = \prod_{\down_j^\lambda < \down_j^\mu} (x_{\down^\lambda_j} \cdots x_{\down^\mu_j -1}) \prod_{\down_j^\mu < \down_j^\lambda} (x_{\down^\mu_j}^{-1} \cdots x_{\down^\lambda_j -1}^{-1})
  \end{equation}
  as an element in $\C[x_1^{\pm 1},\ldots,x_n^{\pm 1}]$. If $(1
  \mapsto \frakm)$ is a monomial morphism of the basis \eqref{soerg:eq:17}
  of $\Hom_R(\sfC_{w_k  z'},\sfC_{w_k  z})$, it follows immediately
  that $ (1 \mapsto \frakm )\in \sfW_{z',z}$ if and only if $(1 \mapsto
  \boldx^{\boldsymbol {b'-b}} \frakm) \in \sfW_{z,z'}$, hence
  $\Theta(\sfW_{z',z})=\sfW_{z,z'}$.

  Moreover, it follows from equation
  \eqref{soerg:eq:1} for the polynomials $p_{\underline \lambda \eta^\sigma}$ and
  $p_{\underline \mu \eta^\sigma}$ that $p_{\underline \mu \eta^\sigma} =x^{\boldsymbol{b-b'}} p_{\underline \lambda \eta^\sigma}$,
  hence     $\Theta(\Psi(\underline \mu \eta^\sigma \overline \lambda)) = \Psi(\underline \lambda \eta^\sigma \overline \mu)$.
\end{proof}

As a corollary, we have that the linear map $\star:A \rightarrow A$ defined by
\begin{equation}
  \label{soerg:eq:92}
  (a \lambda b)^\star = (b^* \lambda a^*)
\end{equation}
is an algebra anti-isomorphism.

As follows from the definition, the algebra $A$ only depends on the number of $\up$'s and $\down$'s in the block $\Gamma$.

\begin{definition}
  We define $A_{n,k}=A_\Gamma$ for some block $\Gamma$
  with $k$ $\up$'s and $n-k$ $\down$'s.\label{def:1}
\end{definition}

\subsection{Graded cellular structure}
\label{soerg:sec:cellularity}
We construct here explicitly the graded cellular structure of the algebra $A$. The construction follows the one in \cite{pre06126156}.

\begin{prop}
  \label{soerg:prop:9}
  Let $(a \lambda b)$ and $(c \mu d)$ be basis vectors of $A$. Then the product $(a \lambda^\sigma b) (c \mu^\tau d)$ is equal to:
  \begin{equation}\label{eq:39}
\left\{    \begin{aligned}
      &0 && \text{if } b \neq c^*,  \\[5pt]
      &  (a \mu^\tau d)
&& \parbox{4cm}{if $b=c^*=\overline\lambda$, $\sigma = e$ and\\$(a\mu d)$ is oriented,} \vspace{5pt}\\[5pt]
      &\displaystyle\sum_{\len(\tau')> \len(\tau)} t^{\tau'}_{(a\lambda^\sigma c)}(\mu^\tau) \cdot (a \mu^{\tau'} d) + (\dag) && \parbox{4.5cm}{if $b=c^*$, $(a\mu d)$ is oriented, and either $b \neq \overline \lambda$ or $\sigma \neq e$,} \vspace{5pt}\\[5pt]
      &(\dag) && \text{otherwise},
    \end{aligned}\right.
  \end{equation}
  where:
  \begin{enumerate}[(i)]
  \item \label{item:1} the scalars $t^{\tau'}_{(a \lambda^\sigma c)}(\mu^{\tau})$ are independent of $d$;
  \item \label{item:2} $(\dag)$ denotes a linear combination of basis vectors of $A$
    of the form $(a \nu^\chi d)$ with $\nu \succ \mu$;
  \end{enumerate}
\end{prop}

\begin{proof}
  If $b \neq c^*$ the claim is obvious, so let us suppose $b=c^*$. Suppose moreover that there is some weight $\nu$ such that $a\nu d$ is oriented (or equivalently that $\sfZ_{d,a}$ is not trivial) otherwise the claim is also obvious.

Of course we have
  \begin{equation}
    (a \lambda^\sigma b) ( c \mu^\tau d) = \sum_{\nu \in \Gamma, \chi \in \bbS_k} C(\nu^\chi)\, (a \nu^\chi d).\label{soerg:eq:177}
  \end{equation} 
 for some coefficients $C(\nu^\chi) \in \C$.
Let us first prove that only terms with $\nu^\chi \succeq \mu^\tau$ occur in the sum, i.e.\ if  $C(\nu^\chi) \neq 0$ then $\nu^\chi \succeq \mu^\tau$.

 Before continuing, let us stress the subtlety in the argument. We want to understand which element of $A$ corresponds to
 the morphism $1 \mapsto p_{a \lambda^\sigma} p_{c \mu^\tau}$: in
 general this morphism is not a monomial morphism of the basis
 \eqref{soerg:eq:17}, and we have to use the relations defining $R_a$ to
 rewrite it as a linear combination of the monomial morphisms
 \eqref{soerg:eq:17}.

 Let us fix some $\nu^\chi$ such that $C(\nu^\chi) \neq
 0$. First, let us prove that $\nu \succeq \mu$. By
 definition, $\nu \succeq \mu$ is equivalent to $\down_j^\nu
 \geq \down_j^\mu$ for all $j = 1,\ldots,n-k$. Fix an index
 $j$. If $\down_j^a \geq \down_j^\mu$, then also $\down_j^\nu
 \geq \down_j^\mu$ by Lemma \ref{soerg:lem:21}
 \ref{soerg:item:23}. Hence suppose $\down_j^a <
 \down_j^\mu$.
By construction, the monomial
\begin{equation}
(x_{\down_j^a} x_{\down_j^a+1} \cdots x_{\down_j^\lambda-1})(x_{\down_j^c} x_{\down_j^c +1} \cdots x_{\down_j^\mu -1})\label{soerg:eq:75}
\end{equation}
divides $p_{a \lambda^\sigma} p_{c \mu^\tau}$. In particular, since $\down_j^\lambda \geq \down_j^b = \down_j^c$, also $x_{\down_j^a} x_{\down_j^a +1} \cdots  x_{\down_j^\mu -1}$ divides $p_{a \lambda^\sigma} p_{c \mu^\tau}$. Hence, if $p_{a
  \lambda^\sigma} p_{c \mu^\tau}$ is a monomial of the basis
\eqref{soerg:eq:17}, we can conclude that $\down_j^\nu \geq
\down_j^\mu$. Otherwise, we get the same conclusion using the
technical Lemma~\ref{soerg:lem:25} below.

Now to check that $\nu^\chi \succeq \mu^\tau$ we have to show that in the case $\nu = \mu$ we have  $\len(\chi) \geq \len(\tau)$. So let us suppose $\nu=\mu$. Since the multiplication is graded, we must have
\begin{equation}
  \label{soerg:eq:186}
  \deg(a \lambda^\sigma b) + \deg(c \mu^\tau d) = \deg (a \mu^\chi d).
\end{equation}
If $a=\underline \rho$ we write $\len(a)$ for $\len(\rho)$, and similarly for $b,c,d$. Then, using Lemma~\ref{soerg:lem:24}, we get from \eqref{soerg:eq:186}
\begin{equation}
  \label{soerg:eq:189}
  2\len(\chi) = 2 \len(\tau) + 2 \len(\sigma) + 2 \len(b) - 2 \len(\lambda).
\end{equation}
Since $\lambda^\sigma b$ is oriented, by Lemma~\ref{soerg:lem:26} the diagram $b$ corresponds to some weight that is smaller or equal than $\lambda$ in the Bruhat order. This implies that $\len(\lambda) \leq \len(b)$ (notice that under the identification of $\Gamma$ with $D_{n,k}$, the Bruhat order on weights corresponds to the opposite of the usual Bruhat order on permutations). It follows that $\len(\chi) \geq \len(\tau)$. Hence we have shown that
  \begin{equation}
    (a \lambda^\sigma b) ( c \mu^\tau d) = \sum_{\len(\tau') \geq \len(\tau)} C(\mu^{\tau'})\, (a \mu^{\tau'} d) + \sum_{\nu \succ \mu, \chi \in \bbS_k} C(\nu^\chi)\, (a \nu^\chi d).\label{eq:38}
  \end{equation}

Now suppose that $C(\mu^\xi)\neq 0$ for some $\xi \in \bbS_k$ with $\len(\xi)=\len(\tau)$. If we substitute in \eqref{soerg:eq:189} $\chi=\xi$, we get $2 \len (\sigma)+ 2 \len(b) - 2 \len(\lambda) =0$. Since $\len(b) \geq \len(\lambda)$, we must have $\len(\sigma)=0$ and $\len(b) = \len(\lambda)$. This implies $\sigma=e$ and $b= \overline \lambda$.
It is easy to see that in this case the morphism $1 \mapsto p_{a
  \lambda^\sigma} p_{c \mu^\tau} $ is an element of the monomial basis
\eqref{soerg:eq:17}, and hence we have exactly $(a \lambda^\sigma b)(c
\mu^\tau d) = (a \mu^\tau d)$. This shows the second case of \eqref{eq:39} and also that if either $b \neq \overline \lambda$ or $\sigma \neq e$ then we can rewrite \eqref{eq:38} as
  \begin{equation}
    (a \lambda^\sigma b) ( c \mu^\tau d) = \sum_{\len(\tau') > \len(\tau)} C(\mu^{\tau'})\, (a \mu^{\tau'} d) + \sum_{\nu \succ \mu, \chi \in \bbS_k} C(\nu^\chi)\, (a \nu^\chi d).\label{eq:40}
  \end{equation}
Since $C(\mu^{\tau'})$ is automatically zero unless $a\mu d$ is oriented, this concludes the proof of \eqref{eq:39} and \ref{item:2}

We are left to show \ref{item:1}.
 In order to determine the coefficients of \eqref{soerg:eq:177}, consider the expression of the polynomial $p_{a \lambda^\sigma} p_{c \mu^\tau}$ in the basis \eqref{soerg:eq:14} of $R_{a}$:
   \begin{equation}
     \label{eq:37}
          p_{a \lambda^\sigma} p_{c \mu^\tau} = \sum_{\boldj \in J} \alpha_\boldj \boldx^{\boldj}.
   \end{equation}
Define $J''\subseteq J$ to be the subset of tuples $\boldj$ such that the morphism $(1 \mapsto \boldx^\boldj) \in \Hom_R(\sfC_d,\sfC_a)$ dies in the quotient $\sfZ_{d,a}$, since it is divided by some morphism of the type \ref{soerg:item:21} of Theorem \ref{soerg:thm:2}. Let also $J' = J \setminus J''$. Fix some $j \in J'$; by Proposition \ref{soerg:prop:6}, the basis morphism $(1 \mapsto \boldx^\boldj) \in \sfZ_{d,a}$ corresponds to a diagram $a \nu^\chi d$: then we have $C(\nu^\chi)=\alpha_j$. Notice that the unique dependence on $d$ is in determining the subset $J'' \subseteq J$.

Now suppose $a\mu d$ is oriented, fix some $\tau' \in \bbS_k$ and let $(1 \mapsto \boldx^\boldj) \in \sfZ_{d,a}$ be the morphism of the basis \ref{soerg:eq:17} corresponding to the diagram $a\mu^{\tau'} d$. By the definition of orientation, for all $i$ we have $\down^d_i \leq \down^\mu_i< \down^d_{i+1}$ and $\down^a_i \leq \down^\mu_i <\down^a_{i+1}$. Since $\boldx^\boldj \in \C[x_{\up^\mu_1},\ldots,x_{\up^\mu_k}]$, neither $x_{\down^d_i} x_{\down^d_{i}+1} \cdots x_{\down^d_{i+1}}$ nor $x_{\down^d_i} x_{\down^d_{i}+1} \cdots x_{\down^a_{i+1}}$ can divide $\boldx^\boldj$. Hence for all $d$ such that $a\mu d$ is oriented we have  $(1 \mapsto \boldx^\boldj) \notin \sfW_{d,a}$ and with the notation of the preceding paragraph $\boldj \in J'$. Hence $C(\mu^{\tau'})$ is independent of $d$, proving \ref{item:1}.
\end{proof}

\begin{lemma}
  \label{soerg:lem:25}
  Fix some $\boldb \in \scrB$ and let $m$ be an index such that
  $b_{m-1}=b_m$. Suppose that $x_m x_{m+1} \cdots x_{m+\ell}$ divides
  some polynomial $p \in R$. Write $p = \sum_{\boldi} c_\boldi \boldx^\boldi$ in
  $R_\boldb$, where $\boldx^\boldi$ are monomials of the basis
  \eqref{soerg:eq:17}. Then $x_m x_{m+1} \cdots x_{m+\ell}$ divides all
  monomials $x^\boldi$ for which $c_\boldi \neq 0$.
\end{lemma}

\begin{proof}
  We will use the relations defining the ideal $I_\boldb$ to write the
  expression of $p$ as a linear combination of basis monomials. Of
  course, it is sufficient to examine the case in which $p=\boldx^\boldj$ is
  a monomial.

  Consider the maximum $r$ for which $j_r\geq b_r$: if there is no
  such $r$, then $p$ is a monomial of the basis \eqref{soerg:eq:17} and we
  are done. If $r<m$ or $r> m+\ell$ then using the relation
  $h_{b_r}(x_1,\ldots,x_r)$ we can rewrite $p$ as a linear
  combination of monomials $\boldx^{\boldsymbol{j'}}$ with $j'_r < j_r$
  and $x_m x_{m+1} \cdots x_{m + \ell } \mid \boldsymbol{x^{j'}}$: so
  by an induction argument we may suppose $m <r < m+\ell$. If $\ell
  \geq 1$ we can write
  \begin{equation}
    x_{r-1 }x_{r}^{j_{r}} = x_{r-1} h_{j_r}(x_1,\ldots,x_r) - \sum_{s=0}^{j_\ell-1} x_{\ell-1} x_r^s h_{j_r-s} (x_1,\ldots,x_{r-1}). \label{soerg:eq:78}
\end{equation}
Since $h_{j_r}(x_1,\ldots,x_r) \in I_\boldb$ because $j_r \geq b_r$,
and also $ x_{r-1} h_{j_r} (x_1,\ldots,x_{r-1}) \in I_\boldb$ by
\eqref{soerg:eq:3}, the expression \eqref{soerg:eq:78} gives in $R_\boldb$
  \begin{equation}
    x_{r-1 }x_{r}^{j_{r}} \equiv \sum_{s=1}^{j_r-1} x_{r-1} x_r^s h_{j_r-s} (x_1,\ldots,x_{r-1}) \mod{I_\boldb}. \label{soerg:eq:79}
\end{equation}
In the special case $\ell=0$, $r=m$, we write instead
  \begin{equation}
    x_m^{j_m} = h_{j_m}(x_1,\ldots,x_m) - \sum_{s=0}^{j_m-1} x_m^s h_{j_m-s} (x_1,\ldots,x_{m-1}) ,
 \end{equation}
 that in $R_\boldb$ is
  \begin{equation}
    x_m^{j_m} \equiv
  - \sum_{s=1}^{j_m-1} x_m^s h_{j_m-s} (x_1,\ldots,x_{m-1}) \mod{I_\boldb}, \label{soerg:eq:80}
\end{equation}
 since $j_m \geq b_{m-1},b_m$. Both in \eqref{soerg:eq:79} and \eqref{soerg:eq:80},
on the r.h.s.\ we have a sum of monomials $\boldsymbol{x^{j'}}$ with $1
\leq j'_r < j_r$: by an induction argument on $j_r$, the claim
follows.
\end{proof}

The main result of this subsection is the graded cellular
algebra structure of $A$ in the sense of \cite{MR1376244},
\cite{MR2671176}. A \emph{graded cellular algebra} is an
associative unital algebra $H$ together with a \emph{cell
  datum} $(X,I,C,\deg)$ such that:
\begin{enumerate}[label=(GC\arabic*),ref=\arabic*]
\item \label{soerg:item:10} $X$ is a finite partially ordered set;
\item \label{soerg:item:11} $I(\lambda)$ is a finite set for each $\lambda \in X$;
\item\label{soerg:item:12}  $C:\dot\bigcup_{\lambda \in X} I(\lambda) \times I(\lambda) \rightarrow H$, $(i,j) \mapsto C^\lambda_{i,j}$ is an injective map whose image is a basis of $H$;
\item \label{soerg:item:13} the map $H \rightarrow H$, $C^\lambda_{i,j} \mapsto C^\lambda_{j,i}$ is an algebra anti-automorphism;
\item \label{soerg:item:14}if $\lambda \in X$ and $i,j \in I(\lambda)$ then for any $x \in H$ we have that
  \begin{equation}
    \label{soerg:eq:93}
    x C_{i,j}^\lambda \equiv \sum_{i' \in I(\lambda)} r_x (i',i)C_{i',j}^\lambda \pmod{H_{>\lambda}},
  \end{equation}
  where the scalar $r_x(i',i)$ is independent of $j$ and $H_{>\lambda}$ is the subspace of $H$ spanned by $\{C^\mu_{h,l} \suchthat \mu > \lambda \text{ and } k,l \in I(\mu)\}$;
\item \label{soerg:item:15} $\deg: \dot \bigcup_{\lambda \in X} I(\lambda) \rightarrow \Z$, $i \mapsto \deg_i^\lambda$ is a function such that the $\Z$--grading on $H$ defined by declaring $\deg C^\lambda_{i,j}= \deg_i^\lambda + \deg_j^\lambda$ makes $H$ into a graded algebra.
\end{enumerate}

We have:
\begin{prop}
  \label{soerg:prop:10}
  The algebra $A$ is a graded cellular algebra with cell datum $((\Gamma \times \bbS_k ,\preceq),I,C,\deg)$ where:
  \begin{enumerate}[(a)]
  \item $I(\lambda^\sigma) = \{\alpha \in \Gamma \suchthat \alpha \subset \lambda \}$;
  \item $C$ is defined by setting $C_{\alpha,\beta}^{\lambda^\sigma} = (\underline \alpha \lambda^\sigma \overline \beta)$;
  \item $\deg_\alpha^{\lambda^\sigma} = \deg(\underline \alpha \lambda^\sigma) - \len(\sigma)$.
  \end{enumerate}
\end{prop}

\begin{proof}
  Conditions (GC\ref{soerg:item:10}-\ref{soerg:item:12}) and (GC\ref{soerg:item:15}) are direct consequences of the definitions. Condition (GC\ref{soerg:item:13}) follows from Lemma~\ref{soerg:lem:27}. Condition (GC\ref{soerg:item:14}) follows from Proposition~\ref{soerg:prop:9}.
\end{proof}

\subsection{Properly stratified structure}
\label{soerg:sec:prop-strat-struct}
As before, let us fix a block $\Gamma$ and let $A=A_\Gamma$. We construct now explicitly a properly stratified structure on $A$.
The construction is similar to the one of \cite{pre06126156}.

An
$A$--module will always be a finite dimensional graded left
$A$--module. Let $\gmod{A}$ be the category of such modules. If
$M=\bigoplus M_i$ is a graded $A$--module then we will write $M\langle
j \rangle$ for the same module structure but with new grading defined
by $(M\langle j \rangle)_i= M_{i-j}$. If $M,N$ are graded $A$--modules
then $\Hom_A (M,N)$ is a graded vector space.

\subsubsection{Irreducible and projective \texorpdfstring{$A$}{A}-modules}
\label{soerg:sec:irred-proj-a}

As we already noticed, the algebra $A$ is unital with $1 =
\sum_{\lambda \in \Gamma} e_\lambda$. Let $A_{>0}$  be the sum of all components of $A$ of strictly positive degree. Then
\begin{equation}
  \label{soerg:eq:63}
  A/A_{>0} = \bigoplus_\lambda e_\lambda \C e_\lambda \cong \bigoplus_{\lambda \in \Gamma} \C
\end{equation}
is a split semisimple algebra, with a basis given by the images of the
idempotents $e_\lambda$. The image of $e_\lambda$ spans a one
dimensional $A/A_{>0}$--modules, and hence also a one dimensional
$A$--module which we denote $L(\lambda)$. Thus $L(\lambda)$ is a copy
of the field concentrated in degree $0$, and $(a \mu^\sigma b) \in A$ acts on
it as $1$ if $(a \mu^\sigma b) = (\underline \lambda \lambda^e \overline
\lambda)$ and as $0$ otherwise. The modules
\begin{equation}
  \label{soerg:eq:77}
  \{ L(\lambda)\langle j \rangle \suchthat \lambda \in \Gamma, j \in \Z \}
\end{equation}
give a complete set of isomorphism classes of irreducible graded $A$--modules.

For any finite dimensional graded $A$--module $M$, let $M^*$
denote its graded dual. That is, $(M^*)_j = \Hom_\C
(M_{-j},\C)$ and $x \in A$ acts on $f \in M^*$ by $xf(m) =
f(x^\star m)$. As $e_\lambda^\star=e_\lambda$ we have that
\begin{equation}
  \label{soerg:eq:81}
  L(\lambda)^* \cong L(\lambda)
\end{equation}
for each $\lambda \in \Gamma$.

For each $\lambda \in \Gamma$ let also $P(\lambda) = A e_\lambda$. This is a graded $A$--module with basis
\begin{equation}
  \label{soerg:eq:96}
  \{ (\underline \nu \mu^\sigma \overline \lambda) \suchthat \text{for all } \nu, \mu \in \Gamma \text{ and } \sigma \in \bbS_k \text{ with } \nu \subset \mu \supset \lambda\}.
\end{equation}
The module $P(\lambda)$ is a projective module; in fact, it is the projective cover of $L(\lambda)$ in $\gmod{A}$. The modules
\begin{equation}
  \label{soerg:eq:97}
  \{ P(\lambda)\langle j \rangle \suchthat \lambda \in \Gamma, j \in \Z \}
\end{equation}
give a complete set of isomorphism classes of indecomposable projective $A$--modules.

\subsubsection{Cell modules and standard modules}
\label{soerg:sec:cell-modul-stand}

We introduce now \emph{standard modules}. The terminology will be motivated at the end of the section. For $\mu \in \Gamma$, define $\Delta(\mu)$ to be the vector space with basis
\begin{equation}
  \label{soerg:eq:98}
  \{ (\underline \lambda \mu^\tau \sbar \suchthat \text{for all } \lambda \in \Gamma, \tau \in \bbS_k \text{ such that } \lambda \subset \mu \}
\end{equation}
or, equivalently,
\begin{equation}
  \label{soerg:eq:100}
  \{ ( c \mu^\tau \sbar \suchthat \text{for all oriented lower fork diagrams } c \mu^\tau\}.
\end{equation}
We put a grading on $\Delta(\mu)$ by defining the degree of $(c \mu^\tau\sbar$ to be $\deg ( c \mu^\tau)$, and we make it into an $A$--module through
\begin{equation}
  \label{soerg:eq:99}
  (a \lambda^\sigma b) (c \mu^\tau\sbar =
  \begin{cases}
    \sum_{\tau' \in \bbS_k} t^{\tau'}_{(a \lambda^\sigma b) } (\mu^\tau) (a \mu^{\tau'}  \sbar & \text{if } b=c^* \text{ and } (a \mu) \text{ is oriented},\\
    0 & \text{otherwise},
  \end{cases}
\end{equation}
where $ t^{\tau'}_{(a \lambda^\sigma b) } (\mu^\tau)$ is the scalar defined by Proposition~\ref{soerg:prop:9}. Note that $t^{\tau'}_{(a \lambda^\sigma b) } (\mu^\tau)$ was defined only for $\tau'=\tau$ or for $\len(\tau')>\len(\tau)$; otherwise we set $t^{\tau'}_{(a \lambda^\sigma b) } (\mu^\tau)=0$.

\begin{prop}
  \label{soerg:thm:3}
  For $\lambda \in \Gamma$ enumerate the distinct elements of the set $\{\mu \in \Gamma \suchthat \mu \supset \lambda\}$ as $\mu_1,\mu_2,\ldots,\mu_m=\lambda$ so that if $\mu_i \prec \mu_j$ then $i>j$. Set $M(0) = \{0\}$ and for $i=1,\ldots,m$ define $M(i)$ to be the subspace of $P(\lambda)$ generated by $M(i-1)$ and the vectors
  \begin{equation}
    \label{soerg:eq:101}
    \{ (c \mu_i^\tau \overline \lambda) \suchthat \text{for all oriented lower fork diagrams } c \mu_i^\tau\}.
  \end{equation}
  Then
  \begin{equation}
    \label{soerg:eq:102}
    \{0\} = M(0) \subset M(1) \subset \cdots \subset M(m)=P(\lambda)
  \end{equation}
  is a filtration of $P(\lambda)$ as an $A$--module such that
  \begin{equation}
    \label{soerg:eq:103}
    M(i)/M(i-1) \cong \Delta(\mu_i) \langle \deg \mu_i \overline \lambda \rangle
  \end{equation}
  for each $i=1,\ldots,m$.
\end{prop}

\begin{proof}
  It follows from Proposition~\ref{soerg:prop:9} that $M(i)$ is indeed a submodule of $P(\lambda)$. The map
  \begin{equation}
    \label{soerg:eq:104}
    \begin{aligned}
      f_i: \Delta(\mu_i) \langle \deg \mu_i \overline \lambda \rangle  & \longrightarrow M(i)/M(i-1)\\
      (c\mu_i^\tau\sbar & \longmapsto (c \mu_i^\tau \overline \lambda) + M(i-1)
    \end{aligned}
  \end{equation}
  gives an isomorphism of graded vector spaces. This map is of degree zero because
  \begin{equation}
    \label{soerg:eq:106}
    \deg(c \mu_i^\tau \overline \lambda) = \deg(c \mu_i^\tau) + \deg( \mu_i \overline \lambda).
  \end{equation}
  Through this vector space isomorphism we can transport the $A$--module structure of $M(i)/M(i-1)$ to $\Delta(\mu_i)$. Using Proposition~\ref{soerg:prop:9} we see that the module structure we get on $\Delta(\mu_i)$ is given by \eqref{soerg:eq:99}. Hence \eqref{soerg:eq:99} defines indeed an $A$--module structure on $\Delta(\mu_i)$ and \eqref{soerg:eq:104} is an isomerism of $A$--modules. Since any weight $\mu$ arises as $\mu_i$ for some $\lambda$ as in the statement of the theorem (take for example $\lambda=\mu$, $i=m$), we conclude also that \eqref{soerg:eq:99} defines an $A$--module structure for every $\mu$.
\end{proof}

\begin{remark}
  \label{rem:6}
  In particular, it follows from
  Proposition~\ref{soerg:thm:3} that each projective module
  $P(\lambda)$ has a \emph{standard filtration}, that is a
  filtration with standard modules. The function
  $\height(\lambda) = \len(\lambda)$, where we consider
  $\lambda$ as an element of $D_{n,k}$ via our bijection,
  plays the role of an \emph{height function}, as in
  \cite{MR2341265}. All subquotients of the standard
  filtration \eqref{soerg:eq:102} of $P(\lambda)$ have the
  form $\Delta(\mu)\langle h \rangle$, where $h \geq 0$ and
  $\height(\mu)=\height(\lambda) - h$, and $h=0$ is possible
  only if $\mu = \lambda$.
\end{remark}

Let us now define \emph{cell modules}. Let $\mu^\tau \in \Gamma \times \bbS_k$ be an enhanced weight and define $V(\mu^\tau)$ to be the vector space with basis
\begin{equation}
  \label{soerg:eq:107}
  \{ (\underline \lambda \mu^\tau\psbar \suchthat \text{for all }\lambda \in \Gamma \text{ such that } \lambda \subset \mu \}
\end{equation}
or, equivalently,
\begin{equation}
  \label{soerg:eq:108}
  \{( c \mu^\tau \psbar \suchthat \text{for all oriented lower fork diagrams } c\mu^\tau \}.
\end{equation}
Here the symbol $\psbar$ has no particular meaning (as also $\sbar$ in \eqref{soerg:eq:98}) and is used only as a symbol to distinguish the basis of cell modules.
We remark that the difference with \eqref{soerg:eq:98} and \eqref{soerg:eq:100} is that now the permutation $\tau$ is fixed. As before, we put a grading on $V(\mu^\tau)$ by defining the degree of $(c\mu^\tau\psbar$ to be $\deg(c \mu^\tau)$, and we make it into an $A$--module through
\begin{equation}
  \label{soerg:eq:109}
  (a \lambda^\sigma b) (c \mu^\tau\psbar =
  \begin{cases}
    t^{\tau}_{(a \lambda^\sigma b) } (\mu^\tau)\cdot (a \mu^{\tau}  \psbar & \text{if } b=c^* \text{ and } (a \mu) \text{ is oriented},\\
    0 & \text{otherwise}.
  \end{cases}
\end{equation}
From Proposition~\ref{soerg:prop:9} we have that $t^\tau_{(a \lambda^\sigma b)}(\mu^\tau)$ does not depend on $\tau$. Hence \eqref{soerg:eq:109} is the same as
\begin{equation}
  \label{soerg:eq:110}
  (a \lambda^\sigma b) (c \mu^\tau\psbar =
  \begin{cases}
    (a \mu^{\tau} \psbar & \text{if } b=c^*=\overline \lambda,\, \sigma=e \text{ and } (a \mu) \text{ is oriented},\\
    0 & \text{otherwise}.
  \end{cases}
\end{equation}

It will follow from Proposition~\ref{soerg:thm:4} that this indeed defines an $A$--module structure. It is clear from \eqref{soerg:eq:110} that all cell modules $V(\mu^\tau)$ for a fixed $\mu$ are isomorphic (up to a degree shift). Explicitly we have $V(\mu^\tau) \cong V(\mu^e) \langle \deg(\tau)\rangle$. We recall that $\deg(\tau)=2 \len(\tau)$. Therefore for a weight $\mu \in \Gamma$ we define the \emph{proper standard module} $\overline \Delta(\mu)$ to be the vector space with basis
\begin{equation}
  \label{soerg:eq:115}
  \{ (\underline \lambda \mu\psbar \suchthat \text{for all }\lambda \in \Gamma \text{ such that } \lambda \subset \mu \}
\end{equation}
or, equivalently,
\begin{equation}
  \label{soerg:eq:116}
  \{( c \mu \psbar \suchthat \text{for all unenhanced oriented lower fork diagrams } c\mu \}.
\end{equation}
We put a grading on $\overline \Delta(\mu)$ by defining the degree of $(c\mu\psbar$ to be $\deg(c \mu)$, and we make it into an $A$--module through
\begin{equation}
  \label{soerg:eq:117}
  (a \lambda^\sigma b) (c \mu\psbar =
  \begin{cases}
    (a \mu \psbar & \text{if } b=c^*=\overline \lambda,\, \sigma=e \text{ and } (a \mu) \text{ is oriented},\\
    0 & \text{otherwise}.
  \end{cases}
\end{equation}
Of course we have an isomorphism $\overline\Delta(\mu)\cong V(\mu^e)$.

\begin{prop}
  \label{soerg:thm:4}
  Let $\mu \in \Gamma$. Enumerate the elements of\/ $\bbS_k$ as $\sigma_1,\sigma_2,\ldots,\sigma_{k!}=e$ in such a way that if\/ $\len(\sigma_i) > \len(\sigma_j)$ then $i<j$. Let $N(0)=\{0\}$ and for $i=1,\ldots,k!$ define $N(i)$ to be the subspace of $\Delta(\mu)$ generated by $N(i-1)$ and the vectors
  \begin{equation}
    \label{soerg:eq:111}
    \{ (c\mu^{\sigma_i}\sbar \suchthat \text{ for all oriented lower fork diagrams } c\mu^{\sigma_i} \}.
  \end{equation}
  Then
  \begin{equation}
    \label{soerg:eq:112}
    \{0\} = N(0) \subset N(1) \subset \cdots \subset N(k!) = \Delta(\mu)
  \end{equation}
  is a filtration of $\Delta(\mu)$ as an $A$--module such that
  \begin{equation}
    \label{soerg:eq:113}
    N(i)/N(i-1) \cong \overline\Delta(\mu)\langle 2\len(\sigma_i) \rangle.
  \end{equation}
\end{prop}

\begin{proof}
  It follows from Proposition~\ref{soerg:prop:9} that $N(i)$ is indeed a submodule of $\Delta(\mu)$. The map
  \begin{equation}
    \label{soerg:eq:114}
    \begin{aligned}
      f_i : \overline\Delta(\mu)\langle 2 \len(\sigma_i) \rangle & \longrightarrow N(i)/N(i-1)\\
      (c \mu \psbar & \longmapsto (c \mu^{\sigma_i}\sbar + N(i-1)
    \end{aligned}
  \end{equation}
  gives an isomorphism of graded vector spaces. The degree shift comes from
  \begin{equation}
    \label{soerg:eq:118}
    \deg( c\mu^{\sigma_i}) = \deg(c \mu) + 2 \len(\sigma_i).
  \end{equation}
  Through $f_i$ we can transport the module structure of $N(i)/N(i-1)$ to $\overline \Delta(\mu)$. The module structure on $N(i)/N(i-1)$ is described by \eqref{soerg:eq:99}. It follows that $\overline\Delta(\mu) \langle 2 \len (\sigma_i) \rangle$ is endowed with the module structure of $V(\mu^{\sigma_i})$ described by \eqref{soerg:eq:109}; this shows in particular that \eqref{soerg:eq:109} defines indeed an $A$--module structure. We have already argued that this is the same as the module structure described by \eqref{soerg:eq:117} on $\overline\Delta(\mu)$.
\end{proof}

\begin{prop}
  \label{soerg:thm:5}
  For $\mu \in \Gamma$, let $Q(j)$ be the submodule of $\oDelta(\mu)$ spanned by all homogeneous vectors of degree $\geq j$. Then
  \begin{equation}
    \label{soerg:eq:119}
    \oDelta(\mu)=Q(0) \supseteq Q(1) \supseteq Q(2) \supseteq \cdots
  \end{equation}
  is a (finite) filtration of $\oDelta(\mu)$ as an $A$--module such that
  \begin{equation}
    \label{soerg:eq:120}
    Q(j)/Q(j+1) \cong \bigoplus_{\substack{\lambda \subset \mu \text{ with}\\ \deg(\underline \lambda \mu) = j}} L(\lambda)\langle j \rangle
  \end{equation}
  for all $j\geq 0$.
\end{prop}

\begin{proof}
  Since $A$ is positively graded, it is clear that each $Q(j)$ is a submodule. The quotient $Q(j)/Q(j+1)$ has basis
  \begin{equation}
    \label{soerg:eq:121}
    \{ (\underline\lambda \mu \psbar + Q(j+1) \suchthat \text{for all } \lambda \in \Gamma \text{ such that } \lambda \subset \mu \text{ and }\deg(\underline \lambda \mu) = j\}.
  \end{equation}
  We need to show that for each $\lambda$ which occurs the one-dimensional subspace $Q'(\lambda)$ of $Q(j)/Q(j+1)$ spanned by $(\underline \lambda \mu \psbar + Q(j+1)$ is an $A$--module isomorphic to $L(\lambda) \langle j \rangle$. It is clear where the degree shift comes from. If $x \in A$ has $\deg(x)>0$ then obviously $x$ vanishes on $Q(j)/Q(j+1)$. So let us consider $e_\nu \in A$. It follows from \eqref{soerg:eq:117} that
  \begin{equation}
    \label{soerg:eq:123}
    e_\nu \cdot (\underline \lambda \mu\psbar =
    \begin{cases}
      (\underline \lambda \mu \psbar & \text{if } \nu=\lambda,\\
      0 & \text{otherwise}.
    \end{cases}
  \end{equation}
  Hence $Q'(\lambda)$ is isomorphic to $L(\lambda)$ after the opportune degree shift.
\end{proof}

\subsubsection{The Grothendieck group}
\label{soerg:sec:grothendieck-group}

The Grothendieck group $K(\gmod{A})$ of $\gmod{A}$ is a free $\Z$--module with basis given by equivalence classes of simple modules. The group $K(\gmod{A})$ becomes a $\Z[v,v^{-1}]$--module if we set $v[M]= [M\langle 1 \rangle]$ for all graded $A$--modules $M$. It is also free as a $\Z[v,v^{-1}]$--module, with basis $\{ [L(\lambda)] \suchthat \lambda \in \Gamma \}$.

For $\lambda,\mu \in \Gamma$, define
\begin{equation}
  \label{soerg:eq:124}
  d_{\lambda,\mu} =
  \begin{cases}
    v^{\deg (\underline \lambda \mu)} & \text{if }\lambda \subset \mu,\\
    0 & \text{otherwise}.
  \end{cases}
\end{equation}
By Propositions~\ref{soerg:thm:3}, \ref{soerg:thm:5} and \ref{soerg:thm:4} respectively  we have that
\begin{align}
  \label{soerg:eq:125}
  [P(\lambda)] & = \sum_{\mu \in \Gamma} d_{\lambda,\mu} [\Delta(\mu)],\\ 
  \label{soerg:eq:126}
  [\oDelta(\mu)] &= \sum_{\lambda \in \Gamma} d_{\lambda,\mu} [L(\lambda)],\\
  \label{soerg:eq:127}
  [\Delta(\mu)] &= [k]_0!\cdot  [\oDelta(\mu)],
\end{align}
where we set
\begin{equation}
  \label{soerg:eq:128}
  [k]_0 = \frac{v^{2k} - 1}{v^2-1} \qquad \text{and} \qquad [k]_0! = [k]_0[k-1]_0\cdots[1]_0.
\end{equation}
Since $d_{\lambda,\lambda}=1$, the matrix $(d_{\lambda,\mu})$ is upper triangular with determinant $1$, hence it is invertible over $\Z[v,v^{-1}]$. In particular, the proper standard modules give also a $\Z[v,v^{-1}]$--basis of $[\gmod{A}]$. On the other side, notice that the matrix $[k]_0! \Id$ is not invertible over $\Z[v,v^{-1}]$ unless $k=0,1$. In particular, standard and projective modules do not give a basis of the Grothendieck group in general.

We recall the definition of a graded \emph{properly stratified algebra} in the sense of \cite{MR2057398} (see also \cite{MR1921761}, \cite{MR2344576}).
\begin{definition}\label{def:5}
  Let $B$ be a finite dimensional associative graded algebra over a
  field $\K$ with a simple preserving duality and with equivalence
  classes of simple modules $\{ \mathbbol L(\lambda)\langle j \rangle \suchthat \lambda \in \Lambda, j \in \Z\}$ where $(\Lambda,\prec)$ is a partially ordered finite set. For each $\lambda
  \in \Lambda$ let
  \begin{enumerate}[(i)]
  \item \label{soerg:item:8} $\mathbbol P(\lambda)$ denote the
    projective cover of $\mathbbol L(\lambda)$,
  \item \label{soerg:item:9} $\mathbbol \Delta(\lambda)$ be the
    maximal quotient of $\mathbbol P(\lambda)$ such that
    $[\mathbbol\Delta(\lambda):\mathbbol L(\mu)]=0$ for all $\mu
    \succ \lambda$,
  \item \label{soerg:item:16} $\overline{\mathbbol \Delta}(\lambda)$
    be the maximal quotient of $\mathbbol \Delta(\lambda)$ such that
    $[\rad \overline{\mathbbol \Delta}(\lambda):\mathbbol L(\mu)]=0$
    for all $\mu \succeq \lambda$.
  \end{enumerate}
  Then $B$ is \emph{properly stratified} if the following conditions hold for
  every $\lambda \in \Lambda$:
  \begin{enumerate}[label=(PS\arabic*),ref=\arabic*]
  \item \label{soerg:item:17} the kernel of the canonical epimorphism
    $\mathbbol P(\lambda) \surto \mathbbol\Delta(\lambda)$ has a
    filtration with subquotients isomorphic to graded shifts of
    $\mathbbol\Delta(\mu)$, $\mu \succ \lambda$;
  \item \label{soerg:item:18} the kernel of the canonical epimorphism
    $\mathbbol\Delta(\lambda) \surto \overline{\mathbbol
      \Delta}(\lambda)$ has a filtration with subquotients isomorphic
    to graded shifts of $\overline{\mathbbol \Delta}(\lambda)$;
  \item \label{soerg:item:19} the kernel of the canonical epimorphism
    $\overline{\mathbbol \Delta}(\lambda) \surto \mathbbol
    L(\lambda)$ has a filtration with subquotient isomorphic to graded
    shifts of $\mathbbol L(\mu)$, $\mu \prec \lambda$.
  \end{enumerate}
\end{definition}
The modules $\mathbbol\Delta(i)$ and $\overline{\mathbbol\Delta}(i)$ are called \emph{standard} and \emph{proper standard} modules respectively.

\begin{theorem}
  \label{soerg:thm:6}
  For every block $\Gamma$ the algebra $A_\Gamma$ is a graded properly stratified algebra. The partially ordered set indexing the simple modules is $(\Gamma,\prec)$. The modules $\Delta(\mu)$ and $\oDelta(\mu)$ are the standard and proper standard modules respectively. Moreover, the diagonal matrix of the multiplicity numbers of the proper standard modules in the filtrations of the standard modules is a multiple of the identity.
\end{theorem}

\begin{proof}
  We already noticed that $A=A_\Gamma$ is a finite dimensional associative unital graded algebra over $\C$ with a duality with respect to which the simple modules are self-dual. For $\lambda \in \Gamma$ let $\mathbbol L(\lambda) = L(\lambda)$ and define $\mathbbol P(\lambda)$, $\mathbbol \Delta(\lambda)$ and $\mathbbol {\overline \Delta}(\lambda)$ as in \ref{soerg:item:8}, \ref{soerg:item:9}, \ref{soerg:item:16} above.
By the uniqueness of the projective cover we have $P(\lambda) \cong \mathbbol P(\lambda) $. From \eqref{soerg:eq:127} and \eqref{soerg:eq:126} we have that $\Delta(\lambda)$ is a quotient of $P(\lambda)$ such that $[\Delta(\lambda):L(\mu)]=0$ for every $\mu \succ \lambda$; from Proposition~\ref{soerg:thm:3} it follows that it is maximal with this property, hence $\Delta(\lambda) \cong \mathbbol \Delta(\lambda)$. By the same argument using \eqref{soerg:eq:126} and Proposition~\ref{soerg:thm:4} we get that $\oDelta(\lambda)\cong \mathbbol{\overline \Delta}(\lambda)$. Hence we need to show that properties (PS\ref{soerg:item:17}-\ref{soerg:item:19}) are satisfied. But this follows immediately from Propositions~\ref{soerg:thm:3}, \ref{soerg:thm:4} and \ref{soerg:thm:5}.
\end{proof}

\subsection{A bilinear form and self-dual projective modules}
\label{soerg:sec:bilinear-form}
We define a bilinear form on $A$ and we determine which projective modules are self-dual.

\subsubsection{Defect}
\label{soerg:sec:defect}

Let $\lambda$ be a weight in some block $\Gamma$. We say that an $\up$ of $\lambda$ is \emph{initial} if it has no $\down$'s on its left. Let us define the \emph{defect} of $\lambda$ to be
\begin{equation}
  \label{soerg:eq:129}
  \defect(\lambda) = \# \{\text{non initial }\up \text{'s of }\lambda \}.
\end{equation}

We have the following elementary result:
\begin{lemma}
  \label{soerg:lem:16}
  The maximal degree of\/ $e_\lambda A e_\lambda$ is $k(k-1) + 2\defect(\lambda) $ and the homogeneous subspace of maximal degree of $e_\lambda A e_\lambda$ is one dimensional.
\end{lemma}

\begin{proof}
  It is straightforward to notice that the homogeneous subspace of maximal degree of $e_\lambda A e_\lambda$ is one dimensional: the diagram of maximal degree is $\underline \lambda \eta^\sigma \overline \lambda$, where $\eta$ orients every fork of $\underline \lambda$ with maximal degree (that is, each $\down$ is at the rightmost position) and $\sigma$ is the longest element of $\bbS_k$. By definition, the degree of this diagram is obtained by adding $2\len(\sigma)$ to the sum of $2(m-1)$ for every $m$--fork of $\underline \lambda$. Hence, this degree is $2 \len(\sigma)$ plus twice the number of non-initial $\up$'s of $\lambda$.
\end{proof}

\begin{lemma}
  \label{soerg:lem:34}
  Consider $\lambda,\mu \in \Gamma$ and suppose that
  $e_\lambda A e_\mu $ is not trivial. Then the homogeneous
  subspaces of minimal and maximal degree of $e_\lambda A
  e_\mu$ are one dimensional. The minimal and the maximal
  degrees are, respectively,
  \begin{equation}
    \label{soerg:eq:135}
    \sum_{i=1}^{n-k} \abs{\down_i^\lambda - \down_i^\mu} \quad \text{and} \quad   k(k-1) + \sum_{i=1}^{n-k} \abs{\down_{i+1}^\mathrm{min} -1 - \down_{i}^\lambda} + \abs{\down_{i+1}^\mathrm{min} -1 - \down_{i}^\mu},
  \end{equation}
where we set $\down_i^{\mathrm{min}} = \min\{\down_i^\lambda,\down_i^\mu\}$ and $\down_{n-k+1}^\lambda=n+1$.
If\/ $\defect(\lambda) \geq \defect(\mu)$ then the sum of the two expressions of \eqref{soerg:eq:135} is equal to the maximal degree of\/ $e_\lambda A e_\lambda$.
\end{lemma}

\begin{proof}
  We use the condition \eqref{soerg:eq:21} to determine if a
  diagram is oriented. The minimal degree diagram is
  $\lambda \eta^e \mu$ where $\down^\eta_i =
  \max\{\down_i^\lambda,\down_i^\mu\}$.
  The
  maximal degree diagram is $\lambda \eta^{w_k } \mu$ where
  $w_k \in \bbS_k$ is the longest element and $\down^\eta_i
  = \min\{\down^\lambda_{i+1},\down^\mu_{i+1}\} - 1$. Computing their degrees we obtain exactly \eqref{soerg:eq:135}.

Let us now check the last assertion. The sum of the two expressions of \eqref{soerg:eq:135} is
\begin{equation}
  \label{soerg:eq:137}
  k(k-1) + \sum_{i=1}^{n-k} 2\left(\down_{i+1}^{\mathrm{min}} - 1 - \down_i^{\mathrm{min}} \right).
\end{equation}
This is the maximal degree of $e_\eta A e_\eta$ where $\eta \in \Gamma$ is the weight with $\down_i^\eta = \down_i^{\mathrm{min}}$. Of course $\defect(\eta) = \max\{\defect(\lambda),\defect(\mu)\}$, and by Lemma~\ref{soerg:lem:16} the maximal degrees of $e_\lambda A e_\lambda$ and $e_\eta A e_\eta$ are the same.
\end{proof}

Notice that a weight $\lambda$ is of maximal defect if and only if it starts with a $\down$.
If $\lambda$ is not of maximal defect, let $\tilde\lambda$ be obtained from $\lambda$ by swapping the first $\down$ and the first $\up$. Otherwise, let $\tilde\lambda=\lambda$. In particular, $\tilde\lambda$ is always of maximal defect.

\begin{lemma}
  \label{soerg:lem:36}
  For every $\lambda \in \Gamma$ the socle of $P(\lambda)$ contains a degree shift of $L(\tilde\lambda)$.
\end{lemma}

In facts, the socle of $P(\lambda)$ is simple, hence it is isomorphic to a degree shift of $L(\tilde\lambda)$, but we will not need this in what follows.

\begin{proof}
  It is straightforward to check that the diagram of maximal degree in $A e_\lambda$ is of type $\underline{\tilde\lambda} \eta^\sigma \overline {\lambda}$. The claim follows.
\end{proof}

\subsubsection{A bilinear form}
\label{sec:bilinear-form-1}

For every $\lambda \in \Gamma$ of maximal defect, let us choose a non-zero element $\xi_\lambda^{\mathrm{max}} \in e_\lambda A e_\lambda$ of maximal degree (for example, we can choose it to be the diagram $\underline \lambda \eta^\sigma \overline \lambda$ of the previous proof). For every element $z \in A$ write $e_\lambda z e_\lambda = t \xi_\lambda^{\mathrm{max}} + \text{terms of lower degree}$, and set $\Theta_\lambda(z) = t$. Moreover, define
\begin{equation}
  \label{soerg:eq:130}
  \Theta(z) = \sum_{\defect(\lambda) \text{ max}} \Theta_\lambda(z).
\end{equation}
Finally, define a bilinear form $\theta : A \times A \rightarrow
\C$ by setting $\theta(y,z) = \Theta(yz)$. Obviously, this form is associative in the sense that $\theta(y,zw)=\theta(yz,w)$ for every $y,z,w \in A$.

\begin{lemma}
  \label{soerg:lem:30}
  For every $\lambda$, the form $\theta$ restricted to $e_\lambda A e_\lambda$ is symmetric and non-degenerate.
\end{lemma}

\begin{proof}
  Let $\lambda$ correspond to $z \in D$. Up to a degree
  shift, $e_\lambda A e_\lambda \cong \sfZ_{z,z}$. Since
  $\sfZ_{z,z}$ is commutative, note that $\theta$ is
  symmetric on $e_\lambda A e_\lambda$. Consider the
  monomial basis $\{1 \mapsto \boldx^\boldi\}$ that consists
  of the elements of \eqref{soerg:eq:17} that are not
  divided by \eqref{soerg:eq:18}. It is clear that for every
  element $\phi$ in that basis there exists exactly one
  element $\phi^T$ in the same basis with
  $\theta(\phi,\phi^T) \neq 0$. This proves that the form is
  non-degenerate.
\end{proof}

Let $e_{\mathrm{def}}= \sum_{\defect(\lambda) \text{ max}} e_\lambda$.

\begin{lemma}
  \label{soerg:lem:29}
  The form $\theta$ restricted to $e_{\mathrm{def}} A \times A e_{\mathrm{def}}$ is non-degenerate.
\end{lemma}

\begin{proof}
  We may take $t \in e_\mu A e_\lambda$ for some $\lambda$ of maximal defect and suppose $\theta(y,t)=0$ for every $y \in e_\lambda A e_\mu$. Let $y_0$ be a generator of the minimal-degree subspace of $e_\lambda A e_\mu$ (which by Lemma~\ref{soerg:lem:34} is one dimensional). In particular, $\theta(y',y_0 t) = \theta(y' y_0, t)=0$ for every $y' \in e_\lambda A e_\lambda$. By Lemma~\ref{soerg:lem:30}, this implies that $y_0 t=0$. From the following Lemma~\ref{soerg:lem:31} it follows then that $t=0$.

The vice versa follows because $\theta(y,t)=\theta(t^\star,y^\star)$.
\end{proof}

\begin{lemma}
  \label{soerg:lem:31}
  Suppose $\lambda$ is of maximal defect and let $0 \neq t \in e_\mu A e_\lambda$. Let also $0 \neq y_0 \in e_\lambda A e_\mu$ be of minimal degree. Then $y_0 t \neq 0$.
\end{lemma}

\begin{proof}
  First, let $0 \neq t_0 \in e_\mu A e_\lambda$ be of minimal degree, and let us prove that $y_0 t_0 \neq 0$. By definition, $y_0 t_0 : 1 \mapsto \boldx^\boldh$, where $h_i = \abs{b_i^\lambda - b_i^\mu} \in \{0,1\}$. First let us suppose that $1 \mapsto \boldx^\boldh$ is an element of the basis \eqref{soerg:eq:17}, that is $h_i < b_i^\lambda$ for every $i$. It is quite easy to argue that for every $i$ there exist an index $j$ with $\down_i^\lambda \leq j < \down_{i+1}^\lambda$ and $b_i^\lambda = b_i^\mu$; in fact it is sufficient to choose $j=\down_i^\mu$ if $\down_i^\mu \geq \down_i^\lambda$ or $j=\down_i^\lambda$ otherwise. This means that $1 \mapsto \boldx^\boldh$ is not illicit (cf.\ Theorem~\ref{soerg:thm:2}), hence it is not zero.

  We should now consider the case in which $1 \mapsto
  \boldx^\boldh$ is not an element of the basis
  \eqref{soerg:eq:17}. This happens if $h_i = 1$ for some
  $i$ with $b_i^\lambda=1$ and $b_i^\mu=2$. Let $j$ be such
  that $\down_{j}^\lambda$ is the rightmost $\down$ in a
  position $\down_{j}^\lambda \leq i$. It is easy to argue
  that for $e_\mu A e_\lambda$ to be non-trivial we must
  actually have $\down_j^\lambda < i$. Let also
  $i'=\down_j^{\mathrm{max}} =
  \max\{\down_j^\lambda,\down_j^\mu\} < i$. Then we have
  $b^\lambda_{i'} = b^\mu_{i'} \geq 2$. Using the relation
  $h_1(x_1,\ldots,x_i)=0$ to write $\boldx^\boldh$ in our
  fixed monomial basis we get in particular a term divided
  by $x_{i'}$. Applying the technique of the previous
  paragraph to this term we get that $y_0 t_0 \neq 0$: the
  only thing to notice is that $x_{\down_j^\lambda}
  x_{\down_j^\lambda + 1} \cdots x_{\down_{j+1}^\lambda -
    1}$ never divides a monomial basis element, since
  $b_{\down_{j+1}^\lambda-1}^\lambda = 1$.

Now, it follows from the proof of Lemma~\ref{soerg:lem:30} that there is some element $u \in R$ such that $y_0 t_0 u$ generates the maximal degree subspace of $e_\lambda A e_\lambda$. In particular $y_0 t_0 u \neq 0$. By Lemma~\ref{soerg:lem:34}, $t_0 u$ is of maximal degree in $e_\mu A e_\lambda$. It is then clear by our characterization of $e_\mu A e_\lambda$ that there exists an element $u' \in R$ such that $u' t = t_0 u$. Now $y_0 t u' = y_0 u' t = y_0 t_0 u \neq 0$ implies that $y_0 t \neq 0$.
\end{proof}

\subsubsection{Self-dual projective modules}
\label{sec:self-dual-projectiveF}

Finally, we can determine which indecomposable projective modules are self-dual.

\begin{lemma}
  \label{soerg:lem:35}
  Let $\lambda$ be of maximal defect. Then $P(\lambda)$ is self-dual up to a degree shift. In particular, it is an injective module.
\end{lemma}

\begin{proof}
  By Lemma~\ref{soerg:lem:29}, the map $    y \longmapsto \theta(y^\star, \cdot)$
  defines an isomorphism between $P(\lambda)$ and its dual up to a degree shift.
\end{proof}

\begin{theorem}
  \label{soerg:thm:7}
  Let $\lambda \in \Gamma$. Then $P(\lambda)$ is an injective module if and only if $\lambda$ is of maximal defect.
\end{theorem}

\begin{proof}
  By Lemma~\ref{soerg:lem:35} if $\lambda$ is of maximal defect then $P(\lambda)$ is injective. On the other side, suppose $P(\lambda)$ is injective. Then $P(\lambda)$ is a tilting module, and by standard theory it is self dual (as an ungraded module). In particular, the socle of $P(\lambda)$ is $L(\lambda)$. By Lemma~\ref{soerg:lem:36}, $\lambda$ has to be of maximal defect.
\end{proof}

\begin{remark}
  \label{rem:4}
  Let $w^0 $ be the longest element of $D$. The weights
  $\lambda$ of maximal defect are exactly the ones that
  correspond to permutations $w_k z$, $z \in D$ which are in
  the same right Kazhdan-Lusztig cell of $w_k w^0$. This can
  be easily checked using the equivalence between
  Kazhdan-Lusztig cells and Knuth equivalence (see
  \cite[\textsection{}5]{MR560412}), and either applying
  directly the definition of Knuth equivalence or using its
  description through the Robinson-Schensted correspondence
  (cf.\ \cite[\textsection{}5.1.4]{MR0445948} and also
  \cite{MR2167481}). This gives another proof of a
  particular case of \cite[Theorem~5.1]{MR2369489} (for the relation with the category $\catO$ see Section~\ref{sec:category-cato} below).
\end{remark}

\subsection{Diagrammatic functors \texorpdfstring{$\calE_k$}{E} and \texorpdfstring{$\calF_k$}{F}}
\label{soerg:sec:functor-calf}

We construct now two functors
\begin{equation}\label{eq:24}
\begin{tikzpicture}[baseline=(A1.base)]
  \node (A1) at (0,0) {$\gmod{A_{n,k}}$};
  \node (A2) at (5,0) {$\gmod{A_{n,k+1}}.$};
  \draw[->,bend left=10] (A1) to node[above] {$\calF_k$} (A2);
  \draw[->,bend left=10] (A2) to node[below] {$\calE_k$} (A1);
\end{tikzpicture}
\end{equation}
They are important since they are the diagrammatic version of the corresponding functors from \cite{miopaperO} (see Section \ref{sec:category-cato} below).

Let us fix an integer $n$. For all $k=0,\ldots,n$ let us set in this section $A_k=A_{n,k}$.
Let $\Gamma_k^\down$ be the subset of weights of $\Gamma_k$ of maximal defect, and let $\Gamma_k^\up = \Gamma_k - \Gamma_k^\down$. Notice that given $\lambda \in \Gamma_k$ we have $\lambda \in \Gamma_k^\down$ if and only if the leftmost symbol of $\lambda$ is a $\down$, and conversely $\lambda \in \Gamma_k^\up$ if and only if the leftmost symbol of $\lambda$ is an $\up$. Let also
\begin{equation}
  \label{soerg:eq:139}
  e^\down_k= \sum_{\lambda \in \Gamma_k^\down} e_\lambda, \qquad   e^\up_k= \sum_{\lambda \in \Gamma_k^\up} e_\lambda.
\end{equation}
In the notation of the previous section, $e_k^{\down} = e_{\mathrm{def}}$.

Consider now $P^\down_k= A_k e^\down_k$, that is the sum of all indecomposable projective-injective $A_k$--modules. We want to describe a right $A_{k+1}$ action on it.

For any $\lambda \in \Gamma_k^\down$ let $\lambda^{(\up)}
\in \Gamma_{k+1}^\up$ be the weight obtained from $\lambda$
by substituting the leftmost symbol, which by assumption is
a $\down$, with an $\up$. Conversely, given $\mu \in
\Gamma_{k+1}^\up$ let $\mu^{(\down)} \in \Gamma_k^\down$ be
the weight obtained from $\mu$ after substituting the
leftmost symbol, which by assumption is an $\up$, with a
$\down$. Clearly the map $\lambda \mapsto
\lambda^{(\up)}$ defines a bijection $\Gamma_k^\down \mapto
\Gamma_{k+1}^\up$ with inverse $\mu \mapsto \mu^{(\down)}$.

\begin{lemma}
  \label{soerg:lem:37}
  Let $\lambda,\mu \in A_{k+1}^\up$. Then we have a natural
  $R$--modules isomorphism
  \begin{equation}
    \label{soerg:eq:140}
      \Hom_R(\sfC_\lambda, \sfC_\mu) \cong \Hom_R(\sfC_{\lambda^{(\down)}}, \sfC_{\mu^{(\down)}} )
  \end{equation}
  that induces a surjective map
  \begin{equation}
    \label{soerg:eq:141}
    e_\mu A_{k+1} e_\lambda \longrightarrow e_{\mu^{(\down)}} A_k e_{\lambda^{(\down)}}.
  \end{equation}
\end{lemma}

\begin{proof}
  Since the $\boldb$--sequences of $\lambda$ and $\lambda^{(\down)}$ are the same, the first claim follows. By Theorem~\ref{soerg:thm:2} the bimodule $\sfW_{\lambda^{(\down)},\mu^{(\down)}}$ is generated by $\sfW_{\lambda,\mu}$ together with the morphism $1 \mapsto x_1,\ldots,x_j$ where $j= \min\{\down_1^\lambda,\down_1^\mu \}$. Hence $e_{\mu^{(\down)}} A_k e_{\lambda^{(\down)}}$ is a quotient of $e_\mu A_{k+1} e_\lambda$.
\end{proof}

\begin{corollary}
  \label{soerg:cor:3}
  We have a surjective algebra homomorphism
  \begin{equation}
    \Psi: e^\up_{k+1} A_{k+1} e^\up_{k+1} \mapto e^\down_k A_k e^\down_k.\label{soerg:eq:142}
  \end{equation}
\end{corollary}

\begin{prop}
  \label{prop:2}
  We have a well-defined surjective algebra homomorphism
  \begin{equation}
    \label{eq:32}
    \begin{aligned}
      A_{k+1} / A_{k+1} e^\down_{k+1} A_{k+1}
      &\longrightarrow e_k^\down A_k e_k^\down\\
      [x] & \longmapsto \Psi(e_{k+1}^\up x e_{k+1}^\up)
    \end{aligned}
  \end{equation}
  for $x \in A_{k+1}$, where $\Psi$ is the homomorphism \eqref{soerg:eq:142}.
\end{prop}

\begin{proof}
  We need to show that \eqref{eq:32} does not depend on the
  particular representative $x$ chosen, or equivalently that
  $\Psi (e_{k+1}^\up x e_{k+1}^\up)=0$ for all $x \in
  A_{k+1} e^\down_{k+1}A_{k+1}$. By linearity, it suffices
  to consider the case $x \in A_{k+1} e_\nu A_{k+1}$ for
  $\nu \in \Gamma^\down_{k+1}$. Pick such an $x$ and fix
  $\lambda, \mu \in \Gamma_{k+1}^\up$. Choose some morphism
  $f \in \Hom_R(\sfC_\lambda,\sfC_\mu)$ which corresponds to
  $e_\mu x e_\lambda$ in the quotient
  $\Hom_R(\sfC_\lambda,\sfC_\mu)/\sfW_{\lambda,\mu}$.  Since
  $x \in A_{k+1} e_\nu A_{k+1}$, we can write $f$ as a
  composition $f_2 \circ f_1$ with $f_1 \in
  \Hom_R(\sfC_\lambda,\sfC_\nu)$ and $f_2 \in
  \Hom_r(\sfC_\nu, \sfC_\mu)$. By
  Corollary~\ref{soerg:cor:8}, $f_1$ is
  divisible by $x_1\cdots x_{\down_1^\lambda}$, hence also $f$ is. By
  Theorem~\ref{soerg:thm:2} (cf.\ also the proof of
  Lemma~\ref{soerg:lem:37} above) we have $f \in
  \sfW_{\lambda^{(\down)},\mu^{(\down)}}$, and hence
  $\Psi(e_\mu x e_\lambda) = 0$. Since $\lambda$ and $\mu$
  were chosen arbitrarily in $\Gamma^\up_{k+1}$, it follows
  that $\Psi(e^\up_{k+1} x e^\up_{k+1}) = 0$.

  The surjectivity of \eqref{eq:32} is a direct consequence of the surjectivity of \eqref{soerg:eq:142}.
\end{proof}

\subsubsection{The functor \texorpdfstring{$\calF_k$}{F}}
\label{sec:bimodule-bff_k}

Let us now define $\bfF_{k}$ to be the $(A_k,A_{k+1})$--bimodule $P_k^\down$, where the right $A_{k+1}$--structure is induced by the quotient map $A_{k+1} \mapto A_{k+1}/A_{k+1} e^\down_{k+1} A_{k+1}$ composed with \eqref{eq:32}. The bimodule $\bfF_k$ defines a right-exact functor
\begin{equation}
  \label{soerg:eq:143}
   \calF_k : \gmod{A_{k+1}} \xrightarrow{\, \bfF_k \otimes_{A_{k+1}} \cdot \,} \gmod{A_k}.
\end{equation}
For each indecomposable projective module $P(\mu) = A_{k+1} e_\mu$ we have
\begin{equation}
  \label{soerg:eq:144}
  \bfF_k \otimes_{A_{k+1}} (A_{k+1} e_\mu) =
  \begin{cases}
    A_k e_\lambda & \text{ if } \lambda^{(\up)} = \mu \text{ for some }\lambda \in \Gamma_k,\\
    0 & \text{otherwise}.
  \end{cases}
\end{equation}

\subsubsection{The functor \texorpdfstring{$\calE_k$}{E}}
\label{sec:bimodule-bfe_k}

The usual hom-tensor adjunction gives a natural isomorphism
\begin{equation}
  \label{eq:148}
  \Hom_{A_k} (\bfF_k \otimes_{A_{k+1}} M , N) \cong \Hom_{A_{k+1}} (M, \Hom_{A_k}(\bfF_k,N))
\end{equation}
for all $M \in \gmod{A_{k+1}}$, $N\in \gmod{A_{k}}$. 
Notice that we have a natural isomorphism $    \Hom_{A_k}(\bfF_k,N) \cong \Hom_{A_k}(\bfF_k,A_k) \otimes_{A_k} N$,
where $\Hom_{A_k}(\bfF_k,A_k)$ is regarded as a $(A_{k+1},A_k)$--bimodule. Let us therefore define $\bfE_k $ to be the $(A_{k+1},A_k)$--bimodule $\Hom_{A_k}(\bfF_k,A_k)$, so that the functor 
\begin{equation}
  \label{eq:22}
   \calE_k : \gmod{A_{k}} \xrightarrow{\, \bfE_k \otimes_{A_{k}} \cdot \,} \gmod{A_{k+1}}.
\end{equation}
is right adjoint to $\calF_k$. Since $\bfF_k$ is a projective $A_k$--module the functor $\calE_k$ is exact.

\begin{remark}
  Since $\bfF_k = A_k e^{\down}_k$ as a left $A_k$--module, we have
  $\Hom_{A_k}(\bfF_k,A_k) \cong e^{\down}_k A_k$ as a right
  $A_k$--module. Hence $\bfE_k$ is the $(A_{k+1},A_k)$--bimodule
  obtained from $\bfF_k$ by turning the left $A_k$--action (resp.\ the right $A_{k+1}$--action) into a right (resp.\ left) one using the anti-isomorphism $\star$ \eqref{soerg:eq:92} of $A_k$ (resp.\ $A_{k+1}$). Specifically, the left action of $\alpha \in A_{k+1}$ on $y \in \bfE_k$ is given by $\alpha \cdot y = y\alpha^\star$ and the right action of $\beta \in A_k$ is given by $y \cdot \beta= \beta^\star y$.
  \label{rem:3}
\end{remark}

\section{Category \texorpdfstring{$\catO$}{O}}
\label{sec:category-cato}

We recall now the definition of the BGG category $\catO(\gl_n)$
and of its subquotient categories $\catO^{\frakp,\frakq\pres}_0$ from
\cite{miopaperO}.
We start with some general facts about Serre subcategories and Serre quotient categories.

\subsection{Serre subcategories and Serre quotients}
\label{sec:serre-quotients}

Let $\calA$ be some abelian category which is equivalent to the
category of finitely generated modules over some finite-dimensional
$\C$--algebra. Let $L(\lambda)$ for $\lambda \in \Lambda$ be the simple
objects of $\calA$ up to isomorphism. For all $\lambda \in \Lambda$
let $P(\lambda)$ be the projective cover of $L(\lambda)$. Let $P=
\bigoplus_{\lambda \in \Lambda} P(\lambda)$ be a minimal projective
generator and let $R= \End_{\calA}(P)$. Then we have an equivalence of
categories $\calA \cong \rmod{R}$
via the functor $\Hom_\calA(P,\blank)$.

\subsubsection{Serre subcategories}
\label{sec:serre-subcategories}

For each subset $\Gamma \subseteq \Lambda$ define
$\calS_\Gamma$ to be the Serre subcategory of $\calA$ consisting of the
modules with all composition factors of type $L(\gamma)$ for $\gamma
\in \Gamma$.
Let $I_\Gamma$ be the two-sided ideal of $R=\End_\calA(P)$
generated by all endomorphisms which factor through some $P(\eta)$ for
$\eta \notin \Gamma$.  Then
\begin{equation}
  \label{eq:13}
  \calS_\Gamma \cong \rmod{R/I_\Gamma}.
\end{equation}
Notice that if we let $e_\lambda$ for $\lambda
\in \Lambda$ be the idempotent projecting onto $\End_\calA(P(\lambda))
\subset R$ and $e^\perp_\Gamma=\sum_{\eta \notin \Gamma} e_\eta$ then
$I_\Gamma = R e_\Gamma^\perp R$.

A complete set of pairwise non-isomorphic simple objects in $\calS_\Gamma$ is given by the $L(\gamma)$'s for $\gamma \in \Gamma$ and each of them has a projective cover $P_{\calS_\Gamma}(\gamma)$ in $\calS_\Gamma$, which is the biggest quotient of $P(\gamma)$ which lies in $\calS_\Gamma$.

\subsubsection{Serre quotients}
\label{sec:serre-quotients-1}

Let us denote by $\calA/\calS_\Gamma$ the Serre quotient category.
Analogously as above, we have an equivalence of categories
  \begin{equation}
    \label{eq:201}
    \calA/\calS_\Gamma \cong \rmod{\End_\calA(P_\Gamma)},
  \end{equation}
where $P_\Gamma=\bigoplus_{\eta \in \Lambda - \Gamma} P(\eta)$ (see for example \cite[Prop.~33]{MR2774639}).   The quotient functor is $Q = \Hom_\calA (P_\Gamma, \blank)$.
Notice that $\End_\calA(P_\Gamma) = e_\Gamma R e_\Gamma$ where $e_\Gamma = \sum_{\gamma \in \Gamma} e_\gamma$.

A complete set of pairwise non-isomorphic simple objects in $\calA/\calS_\Gamma$ is given by the $L(\eta)$'s for $\eta \in \Lambda - \Gamma$, with projective covers $P(\eta)$.
 
\begin{remark}
  The category $\calA/\calS_\Gamma$ is also equivalent to the category
  of $\Add(P_\Gamma)$--presentable modules, that is the full
  subcategory of $\calA$ consisting of all modules $M \in \calA$
  having a presentation
  \begin{equation}
    \label{eq:202}
    Q_1 \longrightarrow Q_2 \twoheadlongrightarrow M
  \end{equation}
  with $Q_1,Q_2 \in \Add(P_\Gamma)$. Here $\Add(P_\Gamma)$ is the
  additive full subcategory of $\calA$ consisting of all objects
  which admit a direct sum decomposition with summands being direct
  summands of $P_\Gamma$\label{rem:1}.
\end{remark}

\subsection{Category \texorpdfstring{$\catO$}{O}}
\label{sec:category-cato-1}

Let $\gl_n$ be the general Lie algebra of complex $n \times n$
matrices. Denote by $\frakh$ the Cartan subalgebra of all diagonal
matrices and by $\frakb=\frakh \oplus \frakn^+$ the Borel subalgebra of all upper triangular
matrices. Let $\alpha_1,\ldots,\alpha_{n-1}$ be the simple roots. We identify the Weyl group with $\bbS_n$.

The BGG category $\catO=\catO(\gl_n)=\catO(\gl_n;\frakb)$, introduced in \cite{MR0407097}, is the full subcategory of all $\gl_n$--modules $M$ which are
\begin{enumerate}[label=($\catO$\arabic*)]
\item finitely generated;
\item direct sum of weight spaces for the action of $\frakh$;
\item locally $\frakn^+$--finite (that is, for every $x \in M$ the $\frakn^+$--submodule generated by $x$ is finite dimensional).
\end{enumerate}

The category $\catO$ decomposes into blocks. Let $\catO_0$
be the block containing the trivial representation
$L(0)$. Then $\catO_0$ is a highest weight category with
simple modules $L(w \cdot 0)$ for $w \in \bbS_n$. In
particular, each simple module $L(w \cdot 0)$ has a
projective cover $P(w \cdot 0)$. For each $w$ in $\bbS_n$
there is a universal highest weight module $M(w \cdot 0) \in \catO_0$
with highest weight $w \cdot 0$, usually called \emph{Verma
  module}. The projective module $P(w \cdot 0)$ is also the
projective cover of $M(w \cdot 0)$.

If we let $P=\bigoplus_{w \in \bbS_n} P(w \cdot 0)$ be a minimal projective generator then we have $\catO_0 \cong \rmod{R}$
where $R = \End_\catO(P)$. For more details about $\catO$ we refer to \cite{MR2428237}.

\subsubsection{Parabolic category \texorpdfstring{$\catO$}{O}}
\label{sec:parab-categ-cato}

Fix a standard parabolic subalgebra $\frakp \subseteq \gl_n$
corresponding to a standard parabolic subgroup $W_\frakp \subseteq
\bbS_n$. Let $W^\frakp$ denote the shortest coset representatives for
$W_\frakp \backslash \bbS_n$. Then the \emph{parabolic category}
$\catO^\frakp_0$ is the Serre subcategory of $\catO_0$ generated  by all
simple modules of the type $L(w \cdot 0)$ for $w \in W^\frakp$. As in \eqref{eq:13}, if we let $I_\frakp$ be the two-sided ideal of $\End_\catO(P)$ generated by all morphisms which factor through some $P(z \cdot 0)$ for $z \notin W^\frakp$ then 
\begin{equation}
\catO^\frakp_0 \cong \rmod{\End_\catO(P)/I_\frakp}.\label{eq:17}
\end{equation}

\subsubsection{\texorpdfstring{$\frakq$}{q}-Presentable category \texorpdfstring{$\catO$}{O}}
\label{sec:frakq-pres-categ}

Let $\frakq \subseteq \gl_n$ be also a standard parabolic subalgebra corresponding to a parabolic subgroup $W_\frakq \subseteq \bbS_n$. Let as before $W^\frakq$ denote the shortest coset representatives for $W_\frakq \backslash \bbS_n$, and let $w_\frakq \in W_\frakq$ be the longest element, so that $w_\frakq W^\frakq$ is the set of the longest coset representatives. Let $\calS^\frakq$ be the Serre subcategory of $\catO_0$ generated by all simple modules $L(z \cdot 0)$ for $z \notin w_\frakq W^\frakq$. Then the \emph{$\frakq$--presentable category} $\catO^{\frakq\pres}_0$ is the Serre quotient $\catO_0/\calS^\frakq$ (the choice of this name is motivated by Remark \ref{rem:1}). As in \eqref{eq:201}, if we let $P_\frakq = \bigoplus_{w \in w_\frakq W^\frakq} P(w \cdot 0)$ then we have an equivalence of categories
\begin{equation}
\catO^{\frakq\pres}_0 \cong \rmod{\End_{\catO}(P_\frakq)}.\label{eq:18}
\end{equation}

\begin{remark}\label{rem:2}
  Note that for $w,z \in w_\frakq W^{\frakq}$ we have
  $\Hom_{\catO}(P(w \cdot 0), P(z \cdot 0)) =
  \Hom_{\catO^{\frakq\pres}_0} (P(w \cdot 0), P(z \cdot 0 ))$, since
  both the head of $P(w \cdot 0)$ and the socle of $P(z \cdot 0)$ are
  simple modules which are not in $\calS^{\frakq}$. For this it is
  essential that we are working with longest coset representatives.
\end{remark}

\subsubsection{Mixed parabolic and \texorpdfstring{$\frakq$}{q}-presentable category \texorpdfstring{$\catO$}{O}}
\label{sec:mixed-parab-frakq}

It is possible to mix the two constructions above for two parabolic subalgebras $\frakp,\frakq \subseteq \gl_n$ which are orthogonal, in the sense that the corresponding sets of simple roots are disjoint and hence $W_\frakp \times W_\frakq$ is a subgroup of $\bbS_n$. One can consider the Serre subcategory $\calS$ of $\catO^\frakp_0$ generated by all simple modules of the type $L(z \cdot 0)$ for $z \in W^\frakp$, $z \notin w_\frakq W^\frakq$ and then define $\catO^{\frakp,\frakq\pres}_0$ to be the Serre quotient $\catO^\frakp_0 / \calS$. Alternatively, one can reverse the construction and define $\catO^{\frakp,\frakq\pres}_0$ to be the Serre subcategory of $\catO^{\frakq\pres}_0$ generated by all simple modules of the form $L(w \cdot 0)$ for $w \in w_q W^\frakq \cap W^\frakp$. As one expects, the two definitions agree (see \cite{miopaperO}).

To get an explicit description of $\catO^{\frakp,\frakq\pres}_0$, let $P^{\frakp}_\frakq = \bigoplus_{w \in w_qW^\frakq \cap W^{\frakp}} P^{\frakp} (w \cdot 0)$, where $P^{\frakp}(w \cdot 0)$ is the projective cover of $L(w \cdot 0)$ in $\catO^{\frakp}_0$. Let moreover $\overline I_\frakp$ be the two-sided ideal of $\End_\catO(P_\frakq)$ generated by all endomorphisms which factor through some $P(z \cdot 0)$ for $z \in w_\frakq W^\frakq$, $z \notin W^\frakp$. Then we have equivalences of categories
\begin{equation}
  \label{eq:14}
  \catO^{\frakp,\frakq\pres}_0 \cong \rmod{\End_\catO(P^\frakp_\frakq)} \cong \rmod{\End_\catO(P_\frakq)/\overline I_\frakp }.
\end{equation}

\subsection{Soergel's theorems}
\label{sec:soergels-theorems}
Fix a positive integer $n$. Using the notation of Section \ref{soerg:sec:soergel-modules}, we recall the main results from \cite{MR1029692}.

\begin{theorem}[{\cite{MR1029692}}]
  \label{thm:9}
  Let $w_n$ be the longest element of\/ $\bbS_n$. We have a canonical
  isomorphism $\End_\catO(P(w_n \cdot 0)) \cong B$. The functor
  $\V=\Hom_\catO (P(w_n \cdot 0), \blank) : \catO_0 \mapto \lmod{B}$
  is fully faithful on projective modules. Moreover, for each $z \in
  \bbS_n$ the $B$--module $\V P(z \cdot 0)$ is isomorphic to the
  Soergel module $\sfC_z$ defined in
  \textsection\ref{sec:soergels-modules}.
\end{theorem}

In particular, it follows from Soergel's result and the discussion
in \textsection\ref{sec:category-cato-1} that $\catO_0 \cong
\rmod{\End_B(\bigoplus_{z \in \bbS_n} \sfC_z)}$. As explained in
\textsection\ref{soerg:sec:grading}, there is a natural way to
consider the Soergel modules as graded modules. It then makes sense to
define the graded version of the category $\catO$ to be
\begin{equation}
  \label{eq:19}
  \catOZ_0 = \rgmod{\End_B\big( \textstyle \bigoplus_{z \in \bbS_n} \sfC_z\big)}.
\end{equation}
This is the Koszul grading of $\catO_0$ \cite{MR1322847}.
Analogously one can define graded versions of the categories $\catO^{\frakp}_0$, $\catO^{\frakq\pres}_0$, $\catO^{\frakp,\frakq\pres}_0$.

We prove now a technical lemma we used in Section \ref{soerg:sec:soergel-modules}.

\begin{lemma}
  \label{lem:3}
  The module $\sfC_z$ is cyclic (generated by $1 \otimes \cdots
  \otimes 1$) if and only if $\calP_{e,z}=v^{\len(z)}$, i.e.\ if and
  only if $H_e$ appears exactly once with coefficient $v^{\len(z)}$ in
  the expression of the canonical basis element $C_z$.
\end{lemma}

\begin{proof}
  Let $\calP_{e,z}$ be the Kazhdan-Lusztig polynomial which gives the
  coefficient of $H_e$ in the expression of $C_z$ in the standard
  basis. Let $(P(z \cdot 0) : M(0))$ denote the multiplicity of the
  dominant Verma module $M(0)$ in some Verma flag of the
  indecomposable projective module $P(z \cdot 0)$ in the category
  $\catO(\gl_n)$. By the Kazhdan-Lusztig conjecture we have
  $\calP_{e,z}|_{v=1} = (P(z \cdot 0):M(0))$. By
  \cite[Lemma~7.3]{MR2017061}, $(P(z \cdot 0):M(0))$ is the
  cardinality of a minimal system of generators for $\sfC_z$.
\end{proof}

\subsection{The category corresponding to the algebra \texorpdfstring{$A_{n,k}$}{A}}
\label{sec:categ-corr-algebra}

For this section, fix two integers $n\geq 0$ and  $0 \leq k \leq n$. Let $W_k$, $W_k^\perp$ be the parabolic subgroups of $\bbS_n$ defined in \textsection\ref{soerg:sec:combinatorics}. Let $\frakq,\frakp \subseteq \gl_n$ be the standard parabolic subalgebras with $W_\frakq = W_k$ and $W_\frakp= W_k^\perp$ so that $W_\frakq \times W_\frakp = \bbS_k \times \bbS_{n-k} \subseteq \bbS_n$. We remark that the resulting category $\catOZ^{\frakp,\frakq\pres}_0$ is the category denoted by $\calQ_k(\compn)$ in \cite{miopaperO}.

As in Section \ref{soerg:sec:soergel-modules}, let $D$ be the set of shortest coset representatives for $\bbS_k \times \bbS_{n-k} \backslash \bbS_n$, that is $D = W^\frakq \cap W^\frakp$.
Recall the Definition \ref{soerg:def:2} of illicit morphisms.

\begin{prop}
  \label{prop:1}
  We have an equivalence of categories
  \begin{equation}
    \label{eq:15}
    \catO^{\frakp,\frakq\pres}_0 \cong \rmod{\bigg(\End \Big(\bigoplus_{z \in D} \sfC_{w_k z} \Big) \Big/ \{\text{illicit morphisms}\} \bigg)}.
  \end{equation}
\end{prop}

In particular, we can define the graded version $\catOZ^{\frakp,\frakq\pres}_0$ of $\catO^{\frakp,\frakq\pres}_0$ to be
\begin{equation}
  \label{eq:21}
      \catOZ^{\frakp,\frakq\pres}_0 = \rgmod{\bigg(\End \Big(\bigoplus_{z \in D} \sfC_{w_k z} \Big) \Big/ \{\text{illicit morphisms}\} \bigg)}.
\end{equation}

\begin{proof}
  By the discussion above we have $\catO^{\frakp,\frakq\pres}_0 \cong \rmod{\End_\catO(P_\frakq)/\overline I_\frakp}$. Here $P_\frakq = \bigoplus_{w \in w_\frakq W^{\frakq}} P(w \cdot 0)$. Since $\End_{\catO}(P_\frakq) = \bigoplus_{w,z \in w_\frakq W^{\frakq}} \Hom (P(w \cdot 0), P(z \cdot 0))$ and any morphisms with source or target some $P(y \cdot 0)$ for $y \notin W^\frakp$ lies in the ideal $\overline I_\frakp$, we have
  \begin{equation}
    \label{eq:16}
    \catO^{\frakp,\frakq\pres}_0 \cong \rmod{\End_\catO\Big(\bigoplus_{w \in D} P(w_\frakq w \cdot 0)\Big)\Big/\overline I_\frakp}.
  \end{equation}
  After applying the isomorphism induced by the functor
  $\funcV$, 
  the ideal $\overline I_\frakp$ becomes exactly the ideal generated
  by illicit morphisms. Hence the claim follows from Theorem
  \ref{thm:9}.
\end{proof}

If $Q(w_k z)$ is the projective cover of the simple module $L(w_k z \cdot 0)$ for all $z \in D$, it follows also that
\begin{equation}
  \label{eq:20}
  \Hom (Q(w_k z), Q (w_k z') )  \cong \Hom_R(\sfC_{w_k z},\sfC_{w_k z'}) / \sfW_{z,z'} = \sfZ_{z,z'}.
\end{equation}
We deduce then:

\begin{lemma}
  \label{lem:2}
  Let $z,z' \in D$, and let $\lambda,\mu \in \Gamma$ be the corresponding weights. The dimension of $ \sfZ_{z,z'}$ is $k!$ times the number of unenhanced weights $\eta$ such that $\underline \mu \eta \overline \lambda$ is oriented.
\end{lemma}

\begin{proof}
We have $\sfZ_{z,z'} \cong \Hom_{\catO^{\frakp,\frakq\pres}_0} (Q(w_k z),Q(w_k z'))$. The dimension of this homomorphism space is computed in \cite[Lemma 7.13]{miopaperO} in terms of evaluation of canonical basis diagrams labeled by standard basis diagrams. This translates immediately in terms of oriented fork diagrams (notice that a canonical basis diagram of \cite{miopaperO} is the same as a lower fork diagram and a standard basis diagram of \cite{miopaperO} is the same as an unenhanced weight).
\end{proof}

\begin{theorem}
  \label{soerg:cor:4}
Let $Q^\frakp_\frakq=\bigoplus_{x \in D} Q(w_k x)$ be the sum of all indecomposable projective modules (up to isomorphism) in $\catOZ^{\frakp,\frakq\pres}_0$. Then we have an isomorphism of graded algebras $A_{n,k} \cong \End_{\catOZ^{\frakp,\frakq\pres}_0}(Q^{\frakq}_\frakp)$. In particular we have an equivalence of categories
  \begin{equation}
    \rgmod{A_{n,k}} \cong \catOZ^{\frakp,\frakq\pres}_0.\label{soerg:eq:58}
  \end{equation}
\end{theorem}

\begin{proof}
 We just need to identify the quotient of the endomorphism algebra appearing in the r.h.s.\ of \eqref{eq:21} with $A_{n,k}$. This follows from Corollary \ref{soerg:cor:5}.
\end{proof}

In Section \ref{sec:diagr-algebra-calq_k} we focused on left $A_{n,k}$--modules, but the whole section could be rewritten for right modules. Alternatively, since the algebra $A_{n,k}$ has an anti-automorphism $\star$ \eqref{soerg:eq:92}, the categories of right and left graded $A_{n,k}$--modules are equivalent. Hence we actually have an equivalence
  \begin{equation}
    \gmod{A_{n,k}} \cong \catOZ^{\frakp,\frakq\pres}_0.\label{eq:23}
  \end{equation}
Although perhaps the equivalence \eqref{soerg:eq:58} is conceptually the right one, we personally prefer to work with left $A_{n,k}$--modules.

\subsection{The functors \texorpdfstring{$\calF_k$}{F} and \texorpdfstring{$\calE_k$}{E}}
\label{sec:funct-calf_k-cale_k}

We want now to relate the diagrammatic functors $\calF_k$ and $\calE_k$ defined in \textsection\ref{soerg:sec:functor-calf} with their Lie theoretical versions from \cite{miopaperO}, of which we briefly recall the definition.

Fix two integers $n>0$ and $0\leq k < n$. Let
$W_k,W_k^\perp,W_{k+1},W_{k+1}^\perp$ be the parabolic
subgroups of $\bbS_n$ defined in
\textsection\ref{soerg:sec:combinatorics}. Let $\frakq, \frakp \subseteq
\gl_n$ be the standard parabolic subalgebras with
$W_\frakq=W_k$ and $W_\frakp = W_k^\perp$, and let $\frakq',
\frakp' \subseteq \gl_n$ be the standard parabolic
subalgebras with $W_{\frakq'}=W_{k+1}$ and $W_{\frakp'} =
W_{k+1}^\perp$. Notice that $\frakp'\subset \frakp$ and
$\frakq \subset \frakq'$.

Let $H= \End_{\catO}(P^{\frakp'}_\frakq)$, where
$P^{\frakp'}_{\frakq}$ is a minimal projective generator of
$\catOZ^{\frakp',\frakq\pres}_0$. Let also
\begin{itemize}
\item $f_{\frakq'} \in H$ be the idempotent projecting onto
  the direct sum of the projective modules $P^{\frakp'}(x
  \cdot 0)$ for $x \in w_{\frakq'}W^{\frakq'} \cap
  W^{\frakp'}$,
\item $e^\frakp \in H$ be the idempotent
  projecting onto the direct sum of the indecomposable
  projective modules $P^{\frakp'}(x \cdot 0) \in
  \catOZ^{\frakp',\frakq\pres}_0$ for $x \in
  w_{\frakq'}W^{\frakq'} \cap W^{\frakp}$ but $x \notin
  w_{\frakq}W^{\frakq}$.
\end{itemize}
Then we have by the discussion in
\textsection\ref{sec:category-cato-1} (using the transitive
property of taking parabolic subcategories and presentable
quotient categories, see \cite{miopaperO}):
\begin{equation}
  \label{eq:146}
  A_k \cong H/H e^\frakp H \qquad \text{and} \qquad A_{k+1} \cong f_{\frakq'} H f_{\frakq'}.
\end{equation}

We have a diagram
\begin{equation}\label{eq:33}
\begin{tikzpicture}[baseline=(current bounding box.center)]
  \matrix (m) [matrix of math nodes, row sep=3em, column
  sep=2.5em, text height=1.5ex, text depth=0.25ex] {
    &  \gmod{H}&  \\
    \gmod{A_k} &  & \gmod{A_{k+1}}\\};
  \path[->] (m-2-1) edge[bend left=10] node[auto] {$ \frakj $} (m-1-2);
  \path[->] (m-1-2) edge[bend left=10] node[auto] {$ \frakz$} (m-2-1);
  \path[->] (m-2-3) edge[bend left=10] node[auto] {$ \fraki$} (m-1-2);
  \path[->] (m-1-2) edge[bend left=10] node[auto] {$ \frakQ$} (m-2-3);
\end{tikzpicture}
\end{equation}
where
\begin{itemize}
\item $\frakj = \Res^{A_k}_H$ is the restriction functor induced by the projection map $H \surto H/H e^\frakp H \cong A_k$;
\item $\frakz = (H/He^\frakp H) \otimes_H $ is the left adjoint of $\frakj$;
\item $\frakQ = \Res^{H}_{A_{k+1}}$ is the restriction functor induced by the inclusion $A_{k+1} \cong f_{\frakq'}H f_{\frakq'} \into H$;
\item $\fraki = H \otimes_{f_{\frakq'} H f_{\frakq'}} $ is the left adjoint of $\frakQ$.
\end{itemize}

Let us define $\calE_k= \frakQ \circ \frakj$ and $\calF_k = \frakz \circ \fraki$. We get then a pair of adjoint functors $(\calF_k,\calE_k)$:
\begin{equation}\label{eq:34}
\begin{tikzpicture}[baseline=(current bounding box.center)]
   \node (A) at (0,0) {$\gmod{A_k}$};
   \node (B) at (3,0) {$\gmod{A_{k+1}}$};
  \path[->] (A) edge[bend left=10] node[auto] {$\calE_k$} (B);
  \path[->] (B) edge[bend left=10] node[auto] {$\calF_k$} (A);
\end{tikzpicture}
\end{equation}

\begin{remark} \label{rem:5} Under the equivalences of
  categories 
  \begin{align}
    \gmod{A_k} &\cong \catOZ^{\frakp,\frakq\pres}_0, \\
    \gmod{H} &\cong \catOZ^{\frakp',\frakq\pres}_0 \\
    \gmod{A_{k+1}} &\cong
    \catOZ^{\frakp',\frakq'\pres}_0,\label{eq:35}
\end{align}
the functors just defined have a Lie theoretical
interpretation: $\fraki$ and $\frakj$ are inclusion of
subcategories, $\frakz$ is a Zuckermann's functor and
$\frakQ$ a coapproximation functor, see \cite{miopaperO}.
\end{remark}

\begin{prop}
  \label{soerg:prop:12}
The functor $\calF_k$ is naturally isomorphic to the functor $\bfF_k \otimes_{A_{k+1}}$.
\end{prop}

\begin{proof}
The functor $\calF_k$ is defined by
  \begin{equation}
    \label{soerg:eq:145}
    M \longmapsto (H/He^\frakp H) \otimes_{f_{\frakq'} H f_{\frakq'}} M,
  \end{equation}
that is the same as
  \begin{equation}
    \label{eq:147}
    M \longmapsto (H/He^\frakp H)\overline{f}_{\frakq'} \otimes_{f_{\frakq'} H f_{\frakq'}} M,
  \end{equation}
where $\overline f_{\frakq'}$ is the image of $f_{\frakq'}$ in $H/He^{\frakp} H$. Obviously $(H/He^\frakp H)\overline f_{\frakq'} = P_k^\down$ as a left $A_k$--module. It is easy to notice that also the right $A_{k+1}$--module structure is the same, since in both cases it is the natural structure induced by the bigger algebra $\End_{\catO}(P)$, where $P$ is a minimal projective generator of $\catOZ_0$.
\end{proof}

By the uniqueness of the adjoint functor we get also:
\begin{prop}
  \label{prop:17}
The functor $\calE_k$ is naturally isomorphic to the functor $\bfE_k \otimes_{A_k}$.
\end{prop}

Consider now $\calE_k$ as an object in the category of
functors between $\gmod{A_k}$ and $\gmod{A_{k+1}}$, and
$\calF_k$ as an object in the category of functors between
$\gmod{A_{k+1}}$ and $\gmod{A_k}$. As such, we can compute
their endomorphism rings:

\begin{theorem}
  \label{thm:6}
  We have $\End (\calE_k)\cong \End(\calF_k) \cong \C[x_1,\ldots,x_n]/I_k$ where $I_k$ is the ideal generated by the complete symmetric functions
  \begin{equation}
    \label{soerg:eq:147}
    \begin{aligned}
      h_{k+1}(x_{i_1},\ldots,x_{i_m}) &\quad  \text{for all }& 1 & \leq m
      \leq n-k,\\
      h_{n-m+1}(x_{i_1},\ldots,x_{i_m}) &\quad  \text{for all }& n-k+1 & \leq m \leq n.
    \end{aligned}
  \end{equation}
  In particular, $\calE_k$ and $\calF_k$ are indecomposable functors.
\end{theorem}

\begin{proof}
  Let us first compute $\End(\calF_k)$. By
  Proposition~\ref{soerg:prop:12}, we have $\End(\calF_k)
  \cong \End_{{A_k \otimes A_{k+1}^\op}}(\bfF_k)$. Since the
  structure of right $A_{k+1}$--module is induced by the
  surjective map \eqref{eq:32}, this is the same as
  $\End_{{A_k \otimes (e^\down_k A_k e^\down_k)^\op}}(\bfF_k)$, that is
  the center of $e^\down_k A_k e^\down_k$. This algebra is
  the endomorphism algebra of the indecomposable
  projective-injective modules of
  $\catOZ^{\frakp,\frakq\pres}_0$, where as before
  $\frakp,\frakq\subseteq\gl_n$ are the parabolic
  subalgebras corresponding to $k$. Since the
  projective-injective modules of
  $\catOZ^{\frakp,\frakq\pres}_0$ are the same as the
  projective-injective modules of $\catOZ^{\frakp}_0$, it is
  also the endomorphism algebra of the indecomposable
  projective-injective modules of $\catO^{\frakp}_0$. By a
  standard argument using the parabolic version of Soergel's
  functor $\bbV$ (see \cite[Section 10]{MR2017061}) it
  follows that this endomorphism algebra is isomorphic to
  the center of $\catO^{\frakp}_0$.  Brundan \cite[Main
  Theorem]{MR2414744} showed that this center is canonically
  isomorphic to $\C[x_1,\ldots,x_n]/I_k$, where $I_k$ is the
  ideal generated by
  \begin{equation}
    \label{soerg:eq:36}
    \begin{aligned}
      h_{r}(x_{i_1},\ldots,x_{i_m}) &\quad  \text{for all }& 1 & \leq m
      \leq n-k, &&r>k\\
      h_{r}(x_{i_1},\ldots,x_{i_m}) &\quad  \text{for all }& n-k+1 & \leq m \leq n, &&r> n-m.
    \end{aligned}
  \end{equation}
Notice that this result builds on a conjecture of Khovanov \cite[Conjecture 3]{MR2078414} (proved in \cite[Main Theorem]{MR2414744}, \cite[Theorem 1]{MR2521250}), that the center of $\catO^\frakp_0$ agrees with the cohomology ring of a Springer fiber. Under this identification, the presentation \eqref{soerg:eq:36} can be deduced from Tanisaki presentation \cite{MR685425} of the cohomology of the Springer fiber.
  Using \eqref{soerg:eq:3} one can easily prove that the polynomials \eqref{soerg:eq:36} generate the same ideal as \eqref{soerg:eq:147}.

  For $\calE_k$, by Proposition~\ref{prop:17} we have $\End(\calE_k) \cong \End_{{A_{k+1} \otimes A_k^\op}}(\bfE_k)$.
By Remark~\ref{rem:3}, it follows that
\begin{equation}
\End(\calE_k)\cong \End_{ A_{k+1}\otimes A_k^{\op}} (\bfE_k) \cong \End_{A_k \otimes A_{k+1}^{\op}}(\bfF_k) \cong \End(\calF_k).\label{eq:150}
\end{equation}
The middle isomorphism can be explained as follows: $\bfE_k$
and $\bfF_k$ have the same underlying vector space $V$;
since the action of $A_{k+1} \otimes A_k^\op$ on $\bfE_k$ is
just the action of $A_k \otimes A_{k+1}^\op$ on $\bfF_k$
twisted (see Remark \ref{rem:3}), a $\C$--linear endomorphism of
$V$ is $A_{k+1} \otimes A_{k}^\op$--equivariant (i.e.\ it is an
endomorphism of $\bfE_k$ as a $(A_{k+1},A_k)$--bimodule)
exactly when it is $A_{k} \otimes A_{k+1}^\op$--equivariant
(i.e.\ it is an endomorphism of $\bfF_k$ as a
$(A_k,A_{k+1})$--bimodule).

  The fact that the functors $\calE_k$ and $\calF_k$ are indecomposable follows since $\End(\calE_k) \cong \End(\calF_k)$ is a graded local ring.
\end{proof}

\nocite{MR573434,MR1029692,MR1857082,MR581584}
\providecommand{\bysame}{\leavevmode\hbox to3em{\hrulefill}\thinspace}
\providecommand{\MR}{\relax\ifhmode\unskip\space\fi MR }
\providecommand{\MRhref}[2]{%
  \href{http://www.ams.org/mathscinet-getitem?mr=#1}{#2}
}
\providecommand{\href}[2]{#2}

\end{document}